\documentclass[11pt,letterpaper]{amsart}
\makeatletter

\usepackage[utf8]{inputenc}
\usepackage[T1]{fontenc}
\usepackage{lmodern}
\usepackage[english]{babel}
\usepackage[stretch=50,shrink=50]{microtype}
\usepackage{amsmath,amssymb,amsfonts,amsthm,esint}
\usepackage{mathtools,accents}
\usepackage{mathrsfs}
\usepackage{aliascnt}
\usepackage{bm}
\usepackage[citecolor=blue,colorlinks,linkcolor=blue]{hyperref}
\usepackage{enumitem}
\usepackage{graphicx}
\usepackage{xcolor}

\mathtoolsset{showonlyrefs}
\usepackage[backend=biber]{biblatex}

\makeatletter
\def\newaliasedtheorem#1[#2]#3{
	\newaliascnt{#1@alt}{#2}
	\newtheorem{#1}[#1@alt]{#3}
	\expandafter\newcommand\csname #1@altname\endcsname{#3}
}
\makeatother

\numberwithin{equation}{section}

\newtheoremstyle{slanted}{\topsep}{\topsep}{\slshape}{}{\bfseries}{.}{.5em}{}

\theoremstyle{plain}
\newtheorem{theorem}{Theorem}[section]
\newaliasedtheorem{proposition}[theorem]{Proposition}
\newaliasedtheorem{lemma}[theorem]{Lemma}
\newaliasedtheorem{corollary}[theorem]{Corollary}
\newaliasedtheorem{counterexample}[theorem]{Counterexample}

\theoremstyle{definition}
\newaliasedtheorem{definition}[theorem]{Definition}
\newaliasedtheorem{question}[theorem]{Question}
\newaliasedtheorem{openquestion}[theorem]{Open Question}
\newaliasedtheorem{conjecture}[theorem]{Conjecture}

\theoremstyle{remark}
\newaliasedtheorem{remark}[theorem]{Remark}
\newaliasedtheorem{example}[theorem]{Example}

\newcommand{\N}{\mathbb{N}}

\newcommand{\setR}{\mathbb{R}}
\newcommand{\R}{\mathbb{R}}
\newcommand{\Z}{\mathbb{Z}}

\newcommand{\eps}{\varepsilon}
\renewcommand{\epsilon}{\varepsilon}
\renewcommand{\phi}{\varphi}
\renewcommand{\bar}{\overline}
\renewcommand{\hat}{\widehat}
\renewcommand{\tilde}{\widetilde}

\DeclareMathOperator{\Hess}{Hess}

\newcommand{\di}{\mathop{}\!\mathrm{d}}

\newcommand{\diam}{{\rm diam}}

\newcommand{\Ch}{{\sf Ch}}

\newcommand{\bb}[1]{\mathbb{#1}}
\newcommand{\ssf}[1]{\mathsf{#1}}
\newcommand{\sd}{\mathsf{d}}
\newcommand{\m}{\mathfrak{m}}
\newcommand{\length}{\mathrm{length}}
\newcommand{\sL}{\mathsf{L}}
\newcommand{\W}{\mathsf{W}}
\newcommand{\aH}{\mathscr{H}}
\newcommand{\Haus}{\mathscr{H}}
\newcommand{\mres}{\mathbin{\vrule height 1.6ex depth 0pt width
0.13ex\vrule height 0.13ex depth 0pt width 1.3ex}}
\newcommand{\norm}[1]{\left\lVert#1\right\rVert}

\DeclareMathOperator{\Lip}{Lip}

\DeclareMathOperator{\lip}{lip} 

\DeclareMathOperator{\Ric}{Ric}

\newcommand{\haus}{\mathscr{H}}
\newcommand{\Leb}{\mathscr{L}}

\newcommand{\dist}{\mathsf{d}}

\newcommand{\meas}{\mathfrak{m}}

\DeclareMathOperator{\CD}{CD}
\DeclareMathOperator{\RCD}{RCD}

\newfont{\tmpf}{cmsy10 scaled 2500}

\def\XXint#1#2#3{{\setbox0=\hbox{$#1{#2#3}{\int}$ }
		\vcenter{\hbox{$#2#3$ }}\kern-.6\wd0}}

\begin{document}

\title[]{New topological restrictions for spaces with nonnegative Ricci curvature}

\author[]{Alessandro Cucinotta, Mattia Magnabosco, and Daniele Semola}

\address{\parbox{\linewidth}{University of Oxford, Mathematical Institute\\
	 Woodstock Rd, Oxford OX2 6GG--UK\\[-4pt]\phantom{a}}}
\email{alessandro.cucinotta@maths.ox.ac.uk}

\address{\parbox{\linewidth}{University of Oxford, Mathematical Institute\\
	 Woodstock Rd, Oxford OX2 6GG--UK\\[-4pt]\phantom{a}}}
\email{mattia.magnabosco@maths.ox.ac.uk}

\address{\parbox{\linewidth}{University of Vienna, Faculty of Mathematics\\
	 Oskar-Morgenstern-Platz 1,
    1090 Wien -- Austria\\[-4pt]\phantom{a}}}
\email{daniele.semola@univie.ac.at}

\begin{abstract}
We obtain new topological restrictions for complete Riemannian manifolds with nonnegative Ricci curvature and $\RCD(0,n)$ spaces. Our main results are a Betti number rigidity theorem which answers a question open since work of M.-T.~Anderson in 1990, and a vanishing theorem for the simplicial volume generalizing a theorem of M.~Gromov from 1982. Combining such results we obtain a new proof of the classification of noncompact $3$-manifolds with nonnegative Ricci curvature, originally due to G.~Liu in 2011, which extends to the synthetic setting.
\end{abstract}

\maketitle


\section{Introduction}\label{sec:Intro}

The topology of complete noncompact Riemannian manifolds with nonnegative Ricci curvature has been studied extensively since the earliest developments of geometric analysis. In dimension $2$, where $\Ric\ge 0$ is equivalent to nonnegative sectional curvature, a complete classification up to diffeomorphism was obtained by S.~Cohn-Vossen in 1935, see \cite{Cohn-Vossen}. In dimension $3$, the classification was achieved much more recently in 2011 by G.~Liu, see \cite{Liu3d}. The strategy in \cite{Liu3d} is based on minimal surfaces methods. It develops in a highly nontrivial way some ideas introduced by R.~Schoen and S.-T.~Yau in \cite{SchoenYau} along their proof that a complete $(M^3,g)$ with $\Ric>0$ must be diffeomorphic to $\R^3$. For $n\ge 4$, a classification up to diffeomorphism seems far too ambitious. We note that even the classification up to homeomorphism of complete Ricci-flat $4$-manifolds is a longstanding open question. At the same time, several families of complete $(M^n,g)$ with $\Ric\ge 0$ and infinite topological type have been constructed in the past decades, see for instance \cite{AndersonKronheimerLeBrun,ShaYangII,Menguyinftop}.

\subsection{Main results}
In this paper we present a new proof of the classification of $3$-manifolds with $\Ric\ge 0$ up to diffeomorphism which:
\begin{itemize}
    \item[i)] is based on two key tools whose validity extends to all dimensions;
    \item[ii)] generalizes to the singular setting, thus yielding a classification of $\RCD(0,3)$ spaces without boundary up to homeomorphism.
\end{itemize}

The first step towards the classification of $3$-manifolds with $\Ric\ge 0$ in \cite{Liu3d} is proving that any such $(M^3,g)$ is either simply connected or its universal cover splits a line isometrically. Thanks to \cite[Lemma 2]{SchoenYau}, which shows that $\pi_2(M)$ is trivial unless the universal cover of $(M,g)$ splits a line isometrically, it suffices to deal with the case when $\pi_1(M)$ is torsion free.
Thus the key ingredient is the following:

\begin{theorem}[\cite{Liu3d}]\label{thm:Liugamma_intro}
Let $(M,g)$ be a complete Riemannian $3$-manifold with $\Ric\ge 0$. If $\pi_1(M)$ contains a subgroup isomorphic to $\Z$, then the universal cover $(\overline{M},\overline{g})$ of $(M,g)$ splits a line isometrically.     
\end{theorem}

Theorem \ref{thm:Liugamma_intro} strengthened an earlier result of Schoen and Yau who proved that if $(M^3,g)$ has $\Ric>0$ then it is simply connected, see \cite[Lemma 3]{SchoenYau}.
\medskip

Our first main result is a generalization of Theorem \ref{thm:Liugamma_intro} to all dimensions:

\begin{theorem}\label{thm:bettirigid_intro}
Let $(M,g)$ be a complete Riemannian $n$-manifold with $\Ric\ge 0$ for $n\ge 3$. If $\pi_1(M)$ contains a subgroup isomorphic to $\Z^{n-2}$, then the universal cover $(\overline{M},\overline{g})$ of $(M,g)$ splits a factor $\R^{n-2}$ isometrically.     
\end{theorem}

\begin{remark}
In \cite{Nabonnand}, P.~Nabonnand constructed a complete $(M^4,g)$ with $\Ric>0$ and $\pi_1(M^4)\simeq\Z$. L.~B\'erard Bergery generalized such construction later in \cite{BerardBergery} by building examples of $(M^n,g)$ with $\Ric>0$ and $\pi_1(M^n)\simeq \Z^{n-3}$. These examples show that Theorem \ref{thm:bettirigid_intro} is sharp, in the sense that its conclusion might fail if we just assume that $\pi_1(M^n)$ has a subgroup isomorphic to $\Z^{n-3}$. 
\end{remark}

\begin{remark}
Schoen and Yau conjectured in \cite[p. 211]{SchoenYau} that for every $n\ge 4$ and every complete Riemannian $(M^n,g)$ with $\Ric>0$ the rank of a free abelian subgroup of $\pi_1(M^n)$ should be less than or equal to $n-3$. Their conjecture was settled by M.-T.~Anderson in \cite{Anderson}. In the same paper, Anderson obtained Theorem \ref{thm:bettirigid_intro} under the additional assumption that $(M^n,g)$ has bounded sectional curvature, after settling the $3$-dimensional case jointly with L.~Rodriguez in \cite[Theorem 2]{AndersonRodriguez}. The validity of Theorem \ref{thm:Liugamma_intro} in all dimensions without further assumptions had remained an open question since then. Very recently, it was stated explicitly as a conjecture by Z.~Ye in \cite[Conjecture 1]{YeJGA} (see also the more recent \cite[Question 1.6]{PanYe}). 
\end{remark}

\begin{remark}
While we were writing up the paper, the interesting preprint \cite{HuangHuangBetti} by H.~Huang and X.-T.~Huang appeared on arXiv. One of their main results establishes Theorem \ref{thm:bettirigid_intro} under the additional assumption that the universal cover $(\overline{M},\overline{g})$ has maximal volume growth. Although there is an overlap between the two results (and the methods of the proof) they were obtained independently.  
\end{remark}

\begin{remark}
Theorem \ref{thm:bettirigid_intro} generalizes to $\RCD(0,n)$ metric measure spaces $(X,\dist,\haus^n)$, see Theorem \ref{thm:Bettirigid} below for the precise statement. Such generalization is crucial for the topological classification of $\RCD(0,3)$ spaces without boundary.
\end{remark}

In the second part of the proof in \cite{Liu3d} the author argues that a contractible $(M^3,g)$ with $\Ric\ge 0$ is simply connected at infinity. This is sufficient to complete the classification up to diffeomorphism in view of the previous considerations, \cite[Theorem 2]{HuschPrice}, \cite{Moise}, and G.~Perelman's solution of the Poincar\'e conjecture. A similar strategy was pursued earlier by Schoen and Yau in \cite{SchoenYau}, though they avoided relying on the Poincar\'e conjecture by proving directly that $M$ is irreducible, based again on minimal surfaces methods.
We do not know whether a contractible $(M^n,g)$ with $\Ric\ge 0$ must be simply connected at infinity when $n\ge 4$. To the best of our knowledge the question is open even for $n=4$ and $\Ric\equiv 0$.
\medskip

The second main result of this paper is a new topological restriction on manifolds supporting a (complete) metric with $\Ric\ge 0$ based on the notion of simplicial volume. The simplicial volume of a connected (possibly noncompact) manifold $M$ is a proper homotopy invariant denoted by $||M||\in [0,\infty]$. It was introduced by M.~Gromov in \cite{GromovVolbCoho}. Here we obtain the following:

\begin{theorem}\label{thm:simpvolintro}
   Let $n\ge 2$. If $(M^n,g)$ is a smooth, complete, oriented Riemannian manifold with $\Ric\ge 0$, then $||M||=0$. 
\end{theorem}

Despite the simplicial volume being a topological invariant, its interplay with (Ricci) curvature became clear since the earliest developments of the theory. In particular, Gromov in \cite[p. 14]{GromovVolbCoho} obtained the following \emph{Isolation Theorem} for the simplicial volume under lower Ricci curvature bounds:

\begin{theorem}[\cite{GromovVolbCoho}]\label{thm:IsolationThmGromov}
For every $n\in\N$ there exists $\eps(n)>0$ such that the following holds. Let $(M^n,g)$ be a complete Riemannian manifold with $\Ric\ge -(n-1)$. If $\mathrm{vol}(B_1(p))\le \eps(n)$ for every $p\in M$, then $||M||=0$. 
\end{theorem}

An immediate corollary of Theorem \ref{thm:IsolationThmGromov} is that the simplicial volume of any closed manifold with $\Ric\ge 0$ vanishes. Thanks to Theorem \ref{thm:simpvolintro}, the closedness assumption can be dropped.

\begin{remark}
Under an additional, most likely technical, condition, Theorem \ref{thm:simpvolintro} generalizes to $\RCD(0,n)$ spaces which are topological $n$-manifolds, see Theorem \ref{thm:simplicialvolthm} below for the precise statement. 
\end{remark}

While Theorem \ref{thm:simpvolintro} yields new topological restrictions to the existence of complete metrics with $\Ric\ge 0$ in all dimensions, see Proposition \ref{prop:applsimplvol} below, our most direct application is to the $3$-dimensional case. G.~Bargagnati and R.~Frigerio proved in \cite{BargagnatiFrigerio} that a contractible $3$-manifold with vanishing simplicial volume must be homeomorphic to $\R^3$. Combining their result with Theorem \ref{thm:simpvolintro} and Theorem \ref{thm:bettirigid_intro}, we obtain a new proof of the topological classification of open $3$-manifolds with $\Ric\ge 0$. Such proof is independent of minimal surfaces methods. More generally, we can classify $\RCD(0,3)$ spaces whose underlying space is a topological $3$-manifold up to homeomorphism.

\begin{theorem}\label{thm:classRCD03intro}
    Let $(X,\dist,\haus^3)$ be a non-compact $\RCD(0,3)$ space and assume that $X$ is a topological manifold. Then either $X$ is homeomorphic to $\R^3$ or its universal cover splits a line isometrically. In the second case, $X$ admits a complete smooth metric with nonnegative sectional curvature.
\end{theorem}

Theorem \ref{thm:classRCD03intro} answers a question raised by G.~Besson around seven years ago and advertised in various occasions since then. See, for instance, \cite[Question 4.1]{Besson3mfld}, \cite[Question 2.10]{BessonGallotscalRic} and \cite[Section 6.5.2]{WangPhD}. The result is original also in the case of noncollapsed Ricci limit spaces $(M^3_i,g_i,p_i)\to (X,\dist)$ with $\Ric_i\ge -1/i$ and $\mathrm{vol}(B_1(p_i))>v>0$ for every $i\ge 1$ and some $v>0$.

We expect that the ideas that we introduce for the proof of Theorem \ref{thm:classRCD03intro} will find applications in the study of collapsing $3$-manifolds under a lower Ricci curvature bound. For the sake of motivation, we note that the classification of open $3$-dimensional Alexandrov spaces with nonnegative curvature plays a key role in T.~Shioya and T.~Yamaguchi's study of collapsing $3$-manifolds under a lower sectional curvature bound \cite{ShioyaYamaguchi}. See also \cite{CaoGecollapsed,Bessierescollapsed,KleinerLottcollapse,MorganTiancollapse}.

Combining Theorem \ref{thm:classRCD03intro} with the recent works on orientability of $\RCD$ spaces \cite{Brenaorientability,QinDenRCD}, one obtains the full classification of non-collapsed $\RCD(0,3)$ spaces without boundary:

\begin{theorem} \label{T|classification}
    Let $(X,\sd,\aH^3)$ be an $\RCD(0,3)$ space without boundary. Then, either $(X,\dist)$ is an Alexandrov space of nonnegative sectional curvature, or $X$ is homeomorphic to one among: $\bb{R}^3$, $C(\bb{RP}^2)$, a spherical 3-manifold, a spherical suspension over $\bb{RP}^2$.
\end{theorem}

\begin{remark}
Regarding the more general case of $\CD(0,3)$ manifolds, when we drop the infinitesimally Hilbertian assumption, the classification remains open, although we can obtain partial progress with our methods, in particular in the contractible  case. The classification remains an open question also for $\RCD(0,3)$ spaces with nonempty boundary.
\end{remark}

\subsection{Discussion of the key new ideas}

The proof of Theorem \ref{thm:bettirigid_intro} hinges on two main steps. The first step is a blow-down argument. It is not difficult to check that the every equivariant Gromov-Hausdorff blow-down of $(\overline{M},\overline{g},\overline{p},\Z^{n-2})$ is an $\RCD(0,n)$ space whose isomorphisms group contains a closed $\R^{n-2}$ subgroup. Here $(\overline{M},\overline{g},\overline{p},\Z^{n-2})$ is the universal cover of $(M,g)$ endowed with the action by deck transformations of $\Z^{n-2}<\pi_1(M)$ and $\overline{p}\in\overline{M}$ is any reference point. See Lemma \ref{L|Gruppilimite} below for the precise statement, which is essentially \cite[Lemma 3.1]{PanYe}. Our first key ingredient is the following splitting result:

\begin{proposition}\label{prop:splitwithaction_intro}
  Let $(\widetilde{X},\widetilde{\dist},\widetilde{\meas},\widetilde{q},G)$ denote any equivariant blow-down of $(\overline{M},\overline{g},\overline{p},\Z^{n-2})$. Then $(\widetilde{X},\widetilde{\dist},\widetilde{\meas},\widetilde{q},G)$ splits a factor $\R^{n-2}$ isometrically and $\R^{n-2}<G$ acts by translations on this Euclidean factor.
\end{proposition}

Proposition \ref{prop:splitwithaction_intro} will follow from a more general splitting theorem via $\R^k$ action, valid for any $\RCD(0,N)$ space:

\begin{theorem} \label{T|splittingkesimo_intro}
    Let $(X,\sd,\m)$ be an $\RCD(0,N)$ space satisfying the following conditions for some integer $1\le k\le N$:
    \begin{enumerate}
        \item $\mathrm{dim}_{\mathrm{e}}(X) \leq k+1$, where $\mathrm{dim}_{\mathrm{e}}(X)$ denotes the essential (also known as rectifiable) dimension of $X$.
        \item $\int_1^{+\infty} \frac{t^{k+1}}{\m(B_t(p))} \, dt =+\infty$ for some $p \in X$.
        \item There exists a closed subgroup $H <\mathrm{Iso}(X)$ isomorphic to $\bb{R}^k$.
    \end{enumerate} Then, $X \cong Y \times \bb{R}^k$ as metric measure space for some $\RCD(0,N-k)$ space $(Y,\dist_Y,\meas_Y)$.
    \end{theorem}

\begin{remark}
There is no analogue of Proposition \ref{prop:splitwithaction_intro} under the weaker assumption that there is an effective action of $\R^{n-3}$ by measure preserving isometries, even in the case of complete Riemannian $n$-manifolds with $\Ric\ge 0$. Indeed, in \cite[p. 913]{KasueWashio} A.~Kasue and T.~Washio constructed a complete $(M^4,g)$ with $\Ric\ge 0$ which does not split any line isometrically, although its isometry group contains a closed $\R$ subgroup. More recently, J.~Pan and G.~Wei constructed examples of $\RCD(0,n)$ spaces with rectifiable dimension $2$ which do not split any line isometrically although the group of measure preserving isometries contains a closed $\R$ subgroup; see \cite{PanWeiGAFA} and their subsequent joint work with X.~Dai and S.~Honda where the construction was made more explicit. Their examples illustrate the role of the $\RCD$ dimension as opposed to the rectifiable dimension in Proposition \ref{prop:splitwithaction_intro}.  
\end{remark}

The proof of Proposition \ref{prop:splitwithaction_intro} relies on some recently developed $\RCD$ technology. The most delicate case is when the blow-down is homeomorphic to $\R^{n-2}\times\R_+$ and has rectifiable dimension $n-1$. We are going to argue that $(\widetilde{X},\widetilde{\dist},\widetilde{\meas})$ is a measured warped product, i.e., ignoring the possible nonsmoothness, the metric has the form $\di r^2+f^2(r)g_{\R^{n-2}}$ and $\overline{\meas}=h(r)\di r\di \Leb^{n-2}$, for some functions $f,h:\R_+\to\R_+$. The $\RCD(0,n)$ condition in this setting is equivalent to a lower bound on the weighted Ricci tensor interpreted in the sense of distributions, thanks to \cite{MondinoRyborz}. Such distributional lower Ricci bound yields some restrictions on $f$ and $h$ which, eventually, imply that $f$ must be constant. Hence $\widetilde{X}$ splits a factor $\R^{n-2}$.

\begin{remark}
In general, the splitting of all blow-downs of a complete manifold with $\Ric\ge 0$ does not imply the splitting of the original manifold. Kasue and Washio's construction from \cite[p. 913-914]{KasueWashio} was the first counterexample. Later Anderson observed that the $n$-dimensional Schwarzschild metrics on $\R^{2}\times S^{n-2}$ for $n\ge 4$ are Ricci-flat counterexamples, cf. with \cite[p. 71]{AndersonDuke}. Counterexamples can be constructed in dimension $3$ as well by using the methods developed by T.-H.~Colding and A.~Naber in \cite{ColdingNabercones}.
\end{remark}

\begin{remark}
For the sake of comparison, we note that the splitting of blow-downs, which is a key step of the proof in the general case, is an immediate consequence of the cone-splitting principle under the assumption that the universal cover has Euclidean volume growth. This is the case considered in the recent preprint \cite{HuangHuangBetti}.    
\end{remark}

That we can obtain the splitting of the universal cover starting from the splitting of the blow-downs is very much special to our situation.
Thanks to the Cheeger-Colding theory, see in particular \cite{ChCo1,ChCo3}, the following holds. If a blow-down of a complete $(M^n,g)$ with $\Ric\ge 0$ splits a factor $\R^k$ isometrically for some $0\le k\le n-1$, then there exist two sequences $\eps_i\to 0$, $R_i\to\infty$ and points $p_i\in M$ such that
\begin{equation}
    \dist_{\mathrm{GH}}\left(B_r(p_i),B^{\R^k\times Y_{r,i}}_r((0,q_{r,i}))\right)<\eps_ir\, ,\,\,  \text{for every}\,\,  i\in\N\,\, \text{and }\,\, 0<r<R_i\, ,
\end{equation}
where $Y_{r,i}$ denotes some complete metric space, $q_{r,i}\in Y_{r,i}$, and $\dist_{\mathrm{GH}}$ denotes the Gromov-Hausdorff distance. We will argue that the points $p_i$ can be chosen to remain at a uniformly bounded distance from some fixed point $p\in M$ as $i\to\infty$ and this is clearly enough to complete the proof of Theorem \ref{thm:bettirigid_intro}.
Ruling out the possibility that the points $p_i$ escape to infinity requires two new ideas. The first one is the use of J.~Cheeger and Naber's slicing theorem, see \cite[Theorem 1.23]{CheegerNaberCodim4}. Such tool crucially refines the general picture that we sketched above under the assumption that $k=n-2$. Very imprecisely, it guarantees that if the points $p_i$ escape to infinity, this can only happen in the directions of the (approximate) $(n-2)$-splitting. To deal with this scenario, the second key tool is a careful study of the behaviour of $(n-2)$ linearly independent harmonic functions with almost linear growth under compositions with the isometric action by $\Z^{n-2}$. The existence of such functions follows from Proposition \ref{prop:splitwithaction_intro} together with \cite[Theorem 1.10]{HuangHuangtransformation}, obtained by H.~Huang and X.-T.~Huang by developing the methods previously introduced by Cheeger, W.~Jiang, and Naber in \cite{CheegerJiangNaber}.

\begin{remark}
The proof of Theorem \ref{thm:bettirigid_intro} goes in the \emph{opposite direction} with respect to the proof of the analogous result in dimension $3$ in \cite{Liu3d}. Indeed, in \cite{Liu3d} an infinitesimal splitting originally obtained along a stable minimal hypersurface is propagated through the whole manifold. In our case we first obtain the splitting at infinity, i.e., at the level of blow-downs, and then upgrade such information to a splitting at finite scales.
\end{remark}

The proof of Theorem \ref{thm:simpvolintro} hinges on a general vanishing theorem for the simplicial volume of noncompact manifolds originally stated by Gromov in \cite[p. 58]{GromovVolbCoho} with a sketch of proof, and proved in all details more recently by Frigerio and M.~Moraschini in \cite[Theorem 9, Corollary 11]{FrigerioMoraschini}. Their result, stated below as Theorem \ref{T|Frig}, guarantees that the simplicial volume of an open triangulable $n$-manifold $M$ vanishes if one can construct a countable open cover of $M$ with relatively compact domains which is amenable at infinity (see Definition \ref{def:amenableatinfinity} below for the relevant terminology) and has multiplicity less than or equal to $n$. To construct such cover in our setting we distinguish two cases. If the volume growth of $(M^n,g)$ is Euclidean, then the cover basically consists of metric annuli. To obtain amenability at infinity we show that the relative fundamental group of a large metric annulus inside a slightly larger one is always finite in this setup, see Theorem \ref{T|ZhouAnelli} below for the precise statement. The proof relies on a contradiction argument based on the Cheeger-Colding theory. When the volume growth is not maximal the construction of the sought cover is more delicate. We are going to combine a conformal deformation argument together with a general result due to P.~Papasoglu in \cite{PapasogluGAFA} (motivated by some earlier work of L.~Guth \cite{GuthlargeballsAnnals,GuthballsJTA} and Gromov's \cite{GromovVolbCoho}) to show that $(M^n,g)$ has small $(n-1)$-Uryson width. This helps us to build a geometrically well-behaved cover of $M$ with multiplicity $\le n$. The amenability at infinity of such cover will follow from V.~Kapovitch and B.~Wilking's Margulis lemma for manifolds with nonnegative Ricci curvature from \cite{KapovitchWilking}.

The extension of Theorem \ref{thm:simpvolintro} to triangulable $\RCD(0,n)$ manifolds follows from the same strategy. The only point one has to be careful about is the application of the Margulis lemma. While a Margulis lemma for $\RCD$ spaces has been recently obtained in \cite{DengSantosZamoraZhao}, the local version corresponding to \cite[Theorem 1]{KapovitchWilking} is not yet available in the literature. In Appendix \ref{appendixA} we obtain a local Margulis lemma for $\RCD(K,N)$ spaces (and, more generally, for essentially nonbranching and semi-locally simply connected $\CD(K,N)$ spaces) where the bound on the index of the nilpotent subgroup is not effective. Despite being weaker than the statement obtained in the smooth case in \cite{KapovitchWilking}, our result, whose proof relies on \cite{BreuillardGreenTao}, is sufficient for the purposes of the present paper.

\medskip

\textbf{Acknowledgements:}

The last author is grateful to Andrea Mondino for bringing Besson's question about the classification of $\RCD(0,3)$ manifolds to his attention a few years ago, to Elia Bruè for useful conversations on the topics of the paper, and to Qin Deng and Sergio Zamora for useful a correspondence about \cite{DengSantosZamoraZhao}.

Part of the work on this project was carried out while the authors were attending the trimester program ``Metric Analysis'' at the Hausdorff Institute for Mathematics in Bonn, funded by the Deutsche Forschungsgemeinschaft (DFG, German Research Foundation) under Germany's Excellence Strategy - EXC-2047/1 - 390685813. The authors are grateful to the organizers of the trimester and to the staff of the HIM for the invitation and the excellent working conditions.

The second author acknowledges support from the Royal Society through the Newton International Fellowship (award number: NIF$\backslash$R1$\backslash$231659).

\section{Preliminaries}

In a metric space $(X,\sd)$, we denote by 
$\mathsf{Lip}(X)$ the space of all Lipschitz functions. Moreover, for a given Lipschitz function $\phi$, we denote by $\sL(\phi)$ the global Lipschitz constant and by $\lip(\phi)(\cdot)$ the local Lipschitz constant
\begin{equation}
    \lip(\phi)(x):=\limsup_{y \to x} \frac{|\phi(x)-\phi(y)|}{\sd(x,y)}\, .
\end{equation}
In the following, $\mathrm{Iso}(X)$ denotes the topological group of isometries of a metric space $X$. When working on a metric measure space $(X,\sd,\m)$, the notation $\mathrm{Iso}(X)$ will denote the collection of \emph{measure-preserving} isometries of $X$.

\subsection{Equivariant Gromov Hausdorff convergence}
We recall some basic facts about equivariant pointed Gromov Hausdorff convergence, the main references being \cite{FukayaConvergence, FukayaYamaguchi}.
In the next definition, a quadruple $(Y,\sd,G,p)$ consists of a pointed metric space $(Y,\sd,p)$ (which is assumed to be a complete length space), and a closed subgroup $G \subset \mathrm{Iso}(Y)$. We denote $G(R):=\{g \in G: \sd(p,gp) \leq R\}$.

\begin{definition} \label{D|epGH}[Equivariant Gromov-Hausdorff convergence]
Two spaces $(Y_1,\sd_1,y_1,G_1)$ and $(Y_2,\sd_2,y_2,G_2)$ are $\epsilon$-close in the \emph{equivariant pointed Gromov-Hausdorff} sense, if there exists a triple $(f,\psi,\phi)$ of maps
\[
f:B_{1/\epsilon}(y_1) \to B_{1/\epsilon}(y_2)\, , \,\, \phi:G_1(1/\epsilon) \to G_2(1/\epsilon)\, , \,\, 
\psi:G_2(1/\epsilon) \to G_1(1/\epsilon)\, ,
\]
with the following properties:
\begin{enumerate}
    \item $f(y_1)=y_2$;
    \item $f(B_{1/\epsilon}(y_1))$ is an $\epsilon$-net in $B_{1/\epsilon}(y_2)$;
    \item $|\sd_2(f(a),f(b))-\sd_1(a,b)| < \epsilon$ for every $a,b \in B_{1/\epsilon}(y_1)$;
    \item If $g \in G_1(1/\epsilon)$, and $x,gx \in B_{1/\epsilon}(y_1)$, then $\sd_2(f(gx),\phi(g)f(x)) < \epsilon$;
    \item \label{item5DefepGH}If $h \in G_2(1/\epsilon)$, and $x,\psi(h)x \in B_{1/\epsilon}(y_1)$, then $\sd_2(f(\psi(h)x),hf(x)) < \epsilon$.
\end{enumerate}
\end{definition}

\begin{lemma} \label{L|convergenceelementsequi}
    Let $(X_i,\sd_i,G_i,x_i)$ be a sequence converging in equivariant pGH sense to $(X,\sd,x,G)$. Let $\gamma_i \in G_i$ be such that $\gamma_i \cdot x_i$ converges to $y \in X$ in the space realizing the pGH convergence. Then there exists $\gamma \in G$ such that for every $y \in X$ and every sequence $y_i \to y$, it holds $\gamma_i \cdot y_i \to \gamma \cdot y$. In this case we say that $\gamma_i \to \gamma$.
\end{lemma}

\begin{theorem} \label{T|EquivaraintConvergence}
    Let $(Y_i,\sd_i,y_i)$ be a sequence of metric spaces converging in pGH sense to $(X,\sd,x)$. Let $G_i$ be closed subgroups of $\mathrm{Iso}(Y_i)$. Then, up to a not relabelled subsequence, there exists a closed subgroup $G<\mathrm{Iso}(X)$ such that
    \[
    (Y_i,\sd_i,G_i,y_i) \to (X,\sd,G,x)
    \]
    in equivariant pGH sense. Moreover,
    \[
    (Y_i/G_i,{\sd_i}_{/\sim},[y_i]) \to (X/G,{\sd}_{/\sim},[x])
    \]
    in pGH sense.
\end{theorem}

In the next definition, a quintuple $(X,\sd,\m,G,p)$ consists of a pointed metric measure space $(X,\sd,\m,p)$ and a closed subgroup $G \subset \mathrm{Iso}(X)$. 

\begin{definition} [Equivariant measured Gromov-Hausdorff convergence]
Two metric measure spaces $(Y_1,\sd_1,\m_1,G_1,y_1)$ and $(Y_2,\sd_2,\m_2,G_2,y_2)$ are $\epsilon$-close in the \emph{equivariant ponted measured Gromov-Hausdorff} sense, if there exists a triple $(f,\psi,\phi)$ of maps
\[
f:B_{1/\epsilon}(y_1) \to B_{1/\epsilon}(y_2), \,\,  \phi:G_1(1/\epsilon) \to G_2(1/\epsilon), \,\,  
\psi:G_2(1/\epsilon) \to G_1(1/\epsilon)\, ,
\]
satisfying the properties of Definition \ref{D|epGH} and such that $f_{\#}(\m_1 \mres B_{1/\epsilon}(y_1))$ is $\epsilon$-close to $\m_2 \mres B_{1/\epsilon}(y_2)$ w.r.t. to weak convergence in $Y_2$. We stress that, in this case, the groups $G_1$ and $G_2$ are required to act by measure-preserving isometries.
\end{definition}

Combining Gromov's Precompactness Theorem, the stability of the $\RCD(K,N)$ condition under pGH convergence \cite{GMS13}, and a simple argument which relies on Theorem \ref{T|EquivaraintConvergence}, one obtains the following.

\begin{proposition} \label{P|empGHconvergence}
    Let $(Y_i,\sd_i,\m_i,y_i)$ be a sequence of normalized $\RCD(K,N)$ spaces, and let $G_i$ be closed subgroups of $\mathrm{Iso}(Y_i)$. Then, up to passing to a subsequence, there exists an $\RCD(K,N)$ space $(X,\sd,\m,x)$, and a closed subgroup $G<\mathrm{Iso}(X)$, such that
    \[
    (Y_i,\sd_i,\m_i,G_i,y_i) \to (X,\sd,\m,G,x)
    \]
    in equivariant pmGH sense. 
\end{proposition}

In the next lemmas,  when we say that a subgroup $G < \mathrm{Iso}(X)$ is isomorphic to $\bb{Z}$ (or $\bb{R}$), we implicitly mean that $G$ is equipped with the topology induced by $\mathrm{Iso}(X)$, that $\bb{Z}$ (or $\bb{R}$) has its natural topology, and that there exists an isomorphism of topological groups $\phi:G \to \bb{Z}$ (or $\phi:G \to \bb{R}$). We denote this by writing $G \cong \bb{Z}$ (or $G \cong \bb{R}$).

\begin{lemma} \label{L|Romeo}
    Let $(X,\sd,\m)$ be an $\RCD(K,N)$. Let $\bb{R}^k \cong G < \mathrm{Iso}(X)$ be a closed subgroup. Then, the orbit $Gp \subset X$ is homeomorphic to $\bb{R}^k$.
    \begin{proof}
        Let $\phi:\bb{R}^k \to G$ be the isomorphism of topological groups given by the statement. Since $G$ acts freely, the map $\psi:G \to Gp \subset X$ is an homeomorphism, so that $\psi \circ \phi:\bb{R}^k \to Gp$ is a homeomorphism as well.
    \end{proof}
\end{lemma}

\begin{lemma} \label{L|Gruppilimite}
     Let $(X_i,\sd_i,\m_i,x_i)$ be a sequence of normalized $\RCD(K,N)$ spaces. Let $\bb{Z}^k \cong G_i< \mathrm{Iso}(X_i)$ be closed subgroups with generators $\{\gamma_{i,j}\}_{j=1}^k$ such that $\sd_i(x_i,\gamma_{i,j} \cdot x_i) \to 0$ as $i \uparrow \infty$, for every $j \in \{1,\cdots,k\}$. Then, up to passing to a subsequence, one has equivariant pmGH convergence
    \[
    (X,\sd_i,\m_i,G_i,x_i) \to (X,\sd,\m,G,x)
    \]
    for a closed subgroup $G<\mathrm{Iso}(X)$ such that $G$ contains a closed subgroup isomorphic to $\bb{R}^k$.
    \begin{proof}
        The proof is the same as \cite[Lemma 3.1]{PanYe} using that $\mathrm{Iso}(X)$ is a Lie Group by \cite{IsoLieRCD} or \cite{SosaLieRCD}.
    \end{proof}
\end{lemma}

\subsection{Some properties of \texorpdfstring{$\RCD$}{RCD} spaces}

We recall some specific properties of $\RCD(K,N)$ spaces that we will use later on in the paper.

\begin{definition}[Tangent spaces to an $\RCD(K,N)$ space]
    Let $(X,\sd,\m)$ be an $\RCD(K,N)$ space, and consider a sequence $r_i \downarrow 0$. For every $x \in X$, by the compactness of $\RCD(K,N)$ spaces w.r.t. pmGH convergence, up to a subsequence, the spaces $(X,\ssf{d}/r_i,\m(B_{r_i}(x))^{-1}\m,x)$ converge in pmGH sense to a limiting $\RCD(0,N)$ space $(X_\infty,\sd_\infty,\m_\infty,x_\infty)$. Such $X_\infty$ is called a \emph{tangent space to $X$ in $x$}. The set of tangent spaces of $X$ in $x$ is denoted $\mathrm{Tan}_x(X)$
\end{definition}

\begin{definition}[Tangent cones at infinity of an $\RCD(0,N)$ space]
    Let $(X,\sd,\m,x)$ be a pointed $\RCD(0,N)$ space, and consider a sequence $r_i \uparrow + \infty$. Up to a subsequence, the spaces $(X,\ssf{d}/r_i,\m(B_{r_i}(x))^{-1}\m,x)$ converge in pmGH sense to a limiting $\RCD(0,N)$ space $(X_\infty,\sd_\infty,\m_\infty,x_\infty)$. Such $X_\infty$ is called a \emph{tangent cone at infinity} (or \emph{blow-down}) of $X$.
\end{definition}

The proof of the next lemma is the same as \cite[Lemma 2.1]{LytchakStadler}.

\begin{lemma} \label{L|tangentconesconnected}
    Let $(X,\sd,\m)$ be an $\RCD(0,n)$ space. Then, the set of its tangent cones at infinity is connected in the pointed Gromov-Hausdorff topology. 
    
    Similarly, given a pointed $\RCD(K,N)$ space $(X,\sd,\m,x)$, the set $\mathrm{Tan}_x(X)$ is connected in the pointed Gromov-Hausdorff topology.
\end{lemma}

The next definition is given in \cite{Cheeger} for PI spaces.

\begin{definition} \label{D|CheegerDiff}
    Let $(X,\sd,\m)$ be an $\RCD(K,N)$ space, and let $x \in X$ be a point with $\mathrm{Tan}_x(X)=\{\bb{R}^d\}$. A function $f \in \Lip(X)$ is differentiable in $x$ if the following happens. Let $r_i \downarrow 0$, $\delta_i \downarrow 0$, and let
    \[
    \psi_i:(\bar{B}_{r_i}^X(x),\sd/r_i) \to (\bar{B}_1^{\bb{R}^d}(0),\sd_{eu})\, , \quad \quad \phi_i: (\bar{B}_1^{\bb{R}^d}(0),\sd_{eu}) \to (\bar{B}_{r_i}^X(x),\sd/r_i) 
    \]
    be $\delta_i$-GH approximations. Then, up to passing to a subsequence, there exists a linear function $f_\infty:\bb{R}^d \to \bb{R}$ with $\lip(f_\infty)=\lip(f)(x)$, such that
    \[
    \Big\|\frac{f-f(x)}{r_i}-f_\infty \circ \psi_i \Big\|_{\infty,\bar{B}^X_{r_i}(x)} +
    \Big\|\frac{f \circ \phi_i-f(x)}{r_i}-f_\infty  \Big\|_{\infty,\bar{B}^{\bb{R}^d}_{1}(0)} \to 0\, .
    \]
\end{definition}

The next theorem follows from \cite[Theorem 10.2]{Cheeger}.

\begin{theorem} \label{T|CheegerDiff}
    Let $(X,\sd,\m)$ be an $\RCD(K,N)$ space, and let $f \in \Lip(X)$. Then, $f$ is differentiable at $\m$-a.e. $x \in X$.
\end{theorem}

The following theorem follows from the proof of  \cite[Theorem 1.3]{MondinoWei}.

\begin{theorem} \label{T|compactsplittingcovering2}
    Let $(X,\sd,\m)$ be a $\RCD(0,n)$ space such that $\mathrm{Iso}(X)$ contains a closed subgroup isomorphic to $\bb{Z}^k$ acting cocompactly. Then, $X$ splits off a factor $\R^k$.
\end{theorem}

The next lemma follows repeating the proof of \cite[Proposition 3.3]{zbMATH07956928}.
\begin{lemma} \label{L|linetangentcone}
    Let $(X,\sd,\m)$ be an $\RCD(0,n)$ space such that all of its tangent cones at infinity are $1$-dimensional, and one tangent cone at infinity is $\bb{R}$. Then, $X \cong \bb{R} \times Y$ isometrically for some compact metric space $Y$.
\end{lemma}

The following lemma is well-known, thus its proof will be omitted.

\begin{lemma} 
\label{L|productisometry}
Let $(X,\sd,\m)$ be a metric measure space. If $X \cong \bb{R}^k \times Y$ and $Y$ contains no lines, then $\mathrm{Iso}(X) < \mathrm{Iso}(\bb{R}^k) \times \mathrm{Iso}(Y)$.
\end{lemma}

The next theorem follows from \cite[Theorem 1.10]{HuangHuangtransformation}.

\begin{theorem} \label{T|HuangHuang}
    Let $(X,\sd,\haus^n,p)$ be an $\RCD(0,n)$ space such that all its tangent cones at infinity are isometric to $\bb{R}^k \times Y$ where $Y$ is some metric space that contains no lines. Then there exists a harmonic map $u=(u_1,\cdots,u_k):X \to \bb{R}^k$ vanishing in $p \in X$ such that 
        \begin{enumerate}
            \item  
            $\{u_1,\cdots,u_k\}$ is a basis for
            \begin{align*}
                \Big\{
                v \in \W^{1,2}_{loc}(X): \Delta v=0, \text{and  for every } \epsilon>0
                 \text{ there exist } R_\epsilon,C_\epsilon>0 : \\
                 v(x) \leq C_\epsilon \sd(x,p)^{1+\epsilon}+C_\epsilon \text{ for every } x \in X
                \Big\}.
                \end{align*}
                \item For every $\epsilon>0$, there exists $R_\epsilon>0$ such that for every $R>R_\epsilon$ there exists a positive definite lower triangular matrix $T_R \in \bb{R}^{k \times k}$, such that $T_R u:B_R(p) \to \bb{R}^{k}$ is an $\epsilon$-splitting map.
        \end{enumerate}
\end{theorem}

\section{First Betti number rigidity}

The goal of this section is to prove the following splitting-type theorem for $\RCD(0,n)$ spaces:

\begin{theorem}\label{thm:Bettirigid}
    Let $(X,\dist,\haus^n)$ be an $\RCD(0,n)$ space for some $n\in\mathbb{N}$, $n\ge 2$. If $\pi_1(X)$ has a subgroup isomorphic to $\Z^{n-2}$ then the universal cover of $X$ splits a factor $\R^{n-2}$ as a metric measure space. 
\end{theorem}

It is clear that Theorem \ref{thm:Bettirigid} implies the corresponding result for complete Riemannian manifolds that we stated in the introduction as Theorem \ref{thm:bettirigid_intro}. 
\medskip

The first key ingredient towards the proof of Theorem \ref{thm:Bettirigid} is the following splitting theorem for the blow-downs of $X$:

\begin{proposition}\label{prop:blowdownsplit}
    Let $(\widetilde{X},\widetilde{\dist},\widetilde{\meas},\widetilde{q},G)$ denote any equivariant blow-down of the universal cover $(\overline{X},\overline{\dist},\haus^n,\overline{p},\Z^{n-2})$ of $(X,\dist,\haus^n)$, endowed with the action of $\Z^{n-2}<\pi_1(X)$ by deck transformations. Then $(\widetilde{X},\widetilde{\dist},\widetilde{\meas},\widetilde{q},G)$ splits a factor $\R^{n-2}$ as a metric measure space and there is a closed subgroup $\R^{n-2}<G$ acting by translations on such Euclidean factor. 
\end{proposition}

A corollary of Proposition \ref{prop:blowdownsplit} is that the blow-down of $(X,\dist,\meas)$ is unique as a metric space. Such information is crucial for the proof of the following:

\begin{proposition}\label{prop:harmonic_behaviour}
Under the assumptions of Theorem \ref{thm:Bettirigid}, there exist linearly independent nonconstant harmonic functions $u_1,\dots,u_{n-2}:\overline{X}\to\R$ such that the following hold:
\begin{itemize}
    \item[i)] for every $\eps>0$ there exists $C_{\eps}>0$ such that $|u_i(x)|\le C_{\eps}\dist(x,\overline{p})^{1+\eps}+C_{\eps}$ for every $x\in\overline{X}$ and $i=1,\dots,n-2$, i.e., the harmonic functions have almost linear growth;
    \item[ii)] for every $\eps>0$ there exists $R_{\eps}>0$ such that for every $r>R_{\eps}$ there is a lower triangular $(n-2)\times (n-2)$ matrix $T_{r,\eps}$ such that $T_{r,\eps}\circ U:B_r(\overline{p})\to \R^{n-2}$ is an $\epsilon$-splitting map. Here we set $U:=(u_1,\dots,u_{n-2}):\overline{X}\to \R^{n-2}$;
    \item[iii)] the map $\Z^{n-2}\ni\gamma\mapsto U(\gamma(\overline{p}))\in\R^{n-2}$ is an injective group homomorphism;
    \item[iv)] if $0<\eps<\eps_0$, then
\begin{equation}\label{eq:imagenondeg}
  \liminf_{r\to\infty}  \frac{\Leb^{n-2}\left(T_{r,\eps}\circ U(B_1(\overline{p}))\right)}{r^{n-2}}>0\, .
\end{equation}
\end{itemize}
\end{proposition}

Items i) and ii) of Proposition \ref{prop:harmonic_behaviour} follow from \cite[Theorem 1.10]{HuangHuangtransformation}, taking into account the metric uniqueness of the blow-down which follows from Proposition \ref{prop:blowdownsplit} (see Corollary \ref{cor:blowdownXbarunique}). On the other hand, the proofs of items iii) and iv) require some new ideas. The \emph{non-degeneracy} condition \eqref{eq:imagenondeg}, together with the almost-splitting property (with arbitrary precision) up to linear transformation in ii) and the equivariance from iii) play a key role in the proof of Theorem \ref{thm:Bettirigid}. 

\begin{remark}
For the sake of clarity we note that a posteriori the functions $u_i$ could just be taken to be the coordinates of the $\R^{n-2}$ factor of $\overline{X}$. It is straightforward to check that with such choice items i)--iv) hold and one can choose all the matrices $T_{r,\eps}=\mathrm{Id}$. However, at this stage of the proof we do not know that $\overline{X}$ splits a factor $\R^{n-2}$ yet.
\end{remark}

We will prove Proposition \ref{prop:blowdownsplit} in Section \ref{sec:splitblowdown} below. Proposition \ref{prop:harmonic_behaviour} is a slightly imprecise combination of Proposition \ref{P|HarmonicHuang} and Proposition \ref{P|slicing}, which are the bulk of Section \ref{sec:equivariantharmonic}. In Section \ref{sec:proofThmbettirigid} we will combine these tools to complete the proof of Theorem \ref{thm:Bettirigid}.

\subsection{Splitting of the blow-downs}\label{sec:splitblowdown}

The goal of this section is to prove Proposition \ref{prop:blowdownsplit}. The proof is much simpler if one restricts to the case where the universal covering of the space in question has Euclidean volume growth. In such case, which is the one dealt with in the recent \cite{HuangHuangBetti}, the sought splitting follows from the so-called cone-splitting principle. 
To treat the remaining cases, the key step is to prove the following splitting theorem via $\R^k$-action.

\begin{theorem} \label{T|splittingkesimo}
    Let $(X,\sd,\m)$ be an $\RCD(0,N)$ space satisfying the following conditions for some integer $1\le k\le N$:
    \begin{enumerate}
        \item $\mathrm{dim}_{\mathrm{e}}(X) \leq k+1$.
        \item $\int_1^{+\infty} \frac{t^{k+1}}{\m(B_t(p))} \, dt =+\infty$ for some $p \in X$.
        \item There exists a closed subgroup $H <\mathrm{Iso}(X)$ which is isomorphic to $\bb{R}^k$.
    \end{enumerate} Then, $X \cong Y \times \bb{R}^k$ as metric measure spaces for some $\RCD(0,N-k)$ space $(Y,\dist_Y,\meas_Y)$.
    \end{theorem}

As we shall see, Theorem \ref{T|splittingkesimo} easily implies the following uniqueness and structure theorem for the blow-downs of the universal cover in the setting of Proposition \ref{prop:blowdownsplit}.

\begin{corollary}\label{cor:blowdownXbarunique}
Let $(X,\dist,\haus^n)$ be an $\RCD(0,n)$ space for some $n\in\mathbb{N}$, $n\ge 2$. If $\pi_1(X)$ has a subgroup isomorphic to $\Z^{n-2}$, then the blow-down of the universal cover $(\overline{X},\overline{\dist},\haus^n)$ of $(X,\dist,\haus^n)$ is unique as a (pointed) metric space and (pointed) isometric to a metric cone (pointed at a tip). 
\end{corollary}

Assuming the validity of Theorem \ref{T|splittingkesimo} and Corollary \ref{cor:blowdownXbarunique}, whose proofs will be discussed later in this section, we complete the proof of Proposition \ref{prop:blowdownsplit}. 

\begin{proof}[Proof of Proposition \ref{prop:blowdownsplit}]
    We assume first that $\overline{X}$ has Euclidean volume growth. Then every blow-down $(\widetilde{X},\widetilde{\sd},\widetilde{q})$ is an $n$-dimensional metric cone with tip $\widetilde{q}$ by \cite[Theorem 1.1]{DePhilippisGiglicone} and \cite[Theorem 1.3]{DephilGigli} after \cite[Theorem 7.6]{ChCo1}. By Lemma \ref{L|Gruppilimite}, there exists a closed subgroup of $G$ isomorphic to $\bb{R}^{n-2}$ which acts freely on $\tilde{X}$ since the only compact subgroup of $\bb{R}^{n-2}$ is the trivial group. Hence, the set of tips of $\tilde{X}$ contains a subset homeomorphic to $\bb{R}^{n-2}$, so that $\tilde{X} \cong \bb{R}^{n-2} \times C(Z)$ isometrically, by cone-splitting, where $Z$ is either isometric to a circle $S^1_r$ for $0<r<1$ or to an interval $[0,a]$ for $0<a\le \pi$.

    Assume now that $\overline{X}$ does not have Euclidean volume growth. In this case, we have that $\mathrm{dim}_{\mathrm{e}}(\widetilde{X})\le \dim_{\haus}(\widetilde{X}) \leq n-1$, by \cite[Theorem 1.4]{DephilGigli} after \cite[Theorem 3.1]{Colding1}. By Lemma \ref{L|Gruppilimite} there exists a closed subgroup of $G$ isomorphic to $\bb{R}^{n-2}$. Moreover, by the Bishop-Gromov inequality, there exists $c>0$ such that 
    \begin{equation}
    \widetilde{\m}(B_r(\widetilde{q})) \leq cr^n \quad \text{for every } r>1\, .
    \end{equation}
    Hence, Theorem \ref{T|splittingkesimo} applies with $k=n-2$ thus yielding that $\widetilde{X} \cong \bb{R}^{n-2} \times Z$ as metric measure space. 
    
    By Corollary \ref{cor:blowdownXbarunique}, the blow-down of $\overline{X}$ is unique as a metric space and isometric to a metric cone. The second part of the statement follows immediately from \cite[Theorem 3.2]{Panalmoststable}. We note that such result is stated for smooth complete Riemannian manifolds therein but the proof works verbatim in the $\RCD$ case. 
\end{proof}

We are going to prove Theorem \ref{T|splittingkesimo} by induction on the dimension $k$ of the group $\bb{R}^k\cong H<\mathrm{Iso}(X)$. The most challenging part of the proof is the base step of the induction, corresponding to $k=1$.
As shown by Proposition \ref{P|alternative}, we can easily reduce to the case where $X/H \cong \bb{R}_+$ isometrically. 
The proof of the splitting theorem in this case consists of two main steps.
The first step of the proof is to show that $(X,\dist,\meas)$ 
is isomorphic to a measured warped product $(\bb{R}_+ \times_{\phi} \bb{R},\dist_{\phi},h(s)\Leb^1(\di s)\Leb^1(\di t))$ for some $\phi \in C(\bb{R}_+,(0,+\infty])$ and $h\in\Lip_{\mathrm{loc}}((0,+\infty))\cap L^1_{\mathrm{loc}}([0,+\infty))$. Establishing such structural result is the goal of Section \ref{subsubsec:warped} 
(see in particular Proposition \ref{prop:measuredwarped}).
In Section \ref{subsubsec:splittingk=1},
using the assumption that $X$ has non-negative Ricci curvature in synthetic sense, we deduce that the warping factor $\phi$ is superharmonic on a suitably weighted half-line. Such weighted half-line turns out to have slow volume growth, thanks to assumption (2), thus forcing the warping factor $\phi$ to be constant and yielding the sought splitting (see in particular Proposition \ref{T|keystepsplitting}). In Section \ref{subsubsec:induction} we deal with the inductive step. In Section \ref{subsubsec:structureblowdowns} we exploit Theorem \ref{T|splittingkesimo} to prove Corollary \ref{cor:blowdownXbarunique}.
\smallskip

\subsubsection{Proof of Theorem \ref{T|splittingkesimo}, base step: measured warped product structure}\label{subsubsec:warped}

Our setup in this section is that we consider an $\RCD(0,N)$ space $(X,\dist,\meas)$ such that $\dim_{\mathrm{e}}(X)\le 2$,
\begin{equation}\label{eq:volgrowthass}
    \int_1^{+\infty}\frac{t^2}{\meas(B_t(p))}\di t=+\infty\, ,
\end{equation}
for some $p\in X$ and there is a closed subgroup $H<\mathrm{Iso}(X)$ isomorphic to $\R$. To complete the proof of the base step of 
Theorem \ref{T|splittingkesimo} we are going to argue that $X$ splits a line isometrically. To do so, as mentioned above, our first goal is to show that $X$ is a measured warped product. Such result holds independently of the volume growth assumption \eqref{eq:volgrowthass}. More precisely, we are going to prove the following.

\begin{proposition}\label{prop:measuredwarped}
Let $(X,\dist,\meas)$ be an $\RCD(0,N)$ space with $\dim_{\mathrm{e}}(X)\le 2$ such that there is a closed subgroup $H<\mathrm{Iso}(X)$ isomorphic to $\R$. Then either $(X,\dist,\meas)$ splits a line or there exist an increasing function $h \in \Lip_{\mathrm{loc}}((0,+\infty)) \cap \sL^1_{\mathrm{loc}}(\bb{R}_+)$ and a continuous and strictly positive function $\phi:(0,+\infty)\to (0,+\infty)$ such that $(X,\dist,\meas,H)$ is equivariantly isomorphic to a measured warped product on $\R_+\times\R$ with metric and measure given respectively by
\begin{equation}\label{eq:measurewarpedprod}
    \di r^2+\phi(r)^2\di x^2\, ,\quad \quad h(r)\di r\di x\, ,
\end{equation} 
and with the $\R$-action being by translation on the $x$ component.
\end{proposition}

Note that the conclusion of Proposition \ref{prop:measuredwarped} would be pretty standard if $(X,\dist,\meas)$ is assumed to be a smooth weighted Riemannian manifold. Such assumption is unfair in our setup, and this is a source of several headaches. Indeed, we will be forced to prove the measured warped product structure simultaneously with some regularity of the warping function $\phi$ and weight function $h$.

\begin{remark}\label{rmk:secondcoordharmonic}
On any (sufficiently regular) measured warped product with the structure in \eqref{eq:measurewarpedprod}, the projection onto the Euclidean component $\R_+\times \R\to \R$ is a harmonic function. Such observation is at the root of all the subsequent arguments in this section.
\end{remark}

\begin{remark}
Proposition \ref{prop:measuredwarped} generalizes in particular the structural results obtained in the recent \cite{DaiHondaPanWei}. In their case, $(X,\dist,\meas)$ is the blow-down of a complete Riemannian manifold $(M,g)$ with $\Ric\ge 0$ and the explicit form of the metric $g$ was exploited to obtain the measured warped product structure of the blow-down.   
\end{remark}

\medskip

We start by observing that there is a natural $\RCD$ structure on the quotient space $X/H$.

\begin{remark}[Quotients of $\RCD$ spaces by a group isomorphic to $\bb{R}$]\label{rmk:doublequotient}
    Let $(X,\sd,\m)$ be an $\RCD(K,N)$ space and let $\bb{R}\cong H <\mathrm{Iso}(X)$ be a closed subgroup. Consider the discrete subgroup $\bb{Z} < H \cong\bb{R}$ generated by $1 \in \bb{R}$. We consider the quotient space $X/\bb{Z}$ equipped with the quotient distance $\sd'$ and the unique measure $\m'$ such that the projection $\pi:X \to X/\bb{Z}$ is locally measure preserving. Since $\pi$ is locally an isomorphism of metric measure spaces, the quotient space $(X/\bb{Z},\sd',\m')$ is locally $\RCD(K,N)$. Thus, by the local to global property \cite{CavallettiMilman,Liglobalization}, $(X/\bb{Z},\sd',\m')$ is $\RCD(K,N)$ as well. 
    
    Consider the action of $\bb{R}/\bb{Z}$ on $X/\bb{Z}$. Since $\bb{R}/\bb{Z}$ is a compact subgroup of $\mathrm{Iso}(X /\bb{Z})$, thanks to \cite{MondinoQuotients}, the quotient space $(X/\bb{Z})/(\bb{R}/\bb{Z})$ is an $\RCD(K,N)$ space when equipped with the quotient distance $\sd''$ and the quotient measure $\m''$. Moreover the metric space $((X/\bb{Z})/(\bb{R}/\bb{Z}), \sd'')$ is isomorphic to $X/\bb{R}$ with its natural quotient distance.
\end{remark}

We continue with the following observation regarding the dimension of the quotient of an $\RCD$ space with respect to an action of $\R$ by measure preserving isometries.

\begin{lemma} \label{L|genericquotients}
     Let $(X,\sd,\m)$ be an $\RCD(K,N)$ space such that $H <\mathrm{Iso}(X)$ is a closed subgroup isomorphic to $\bb{R}$.
     Consider the $\RCD(K,N)$ structure $(X/H,\tilde\sd,\tilde\m)$ on the quotient space $X/H$ given by Remark \ref{rmk:doublequotient}.
  Then, we have
    \[
    \mathrm{dim}_{\mathrm{e}}(X/H) \leq \mathrm{dim}_{\mathrm{e}}(X)-1\, .
    \]
    \begin{proof}
        Let $d \in \bb{N}$ be the essential dimension of $(X/H,\tilde\sd,\tilde\m)$ and let $d'$ be the essential dimension of $(X,\sd,\m)$. Call $\pi:X \to X/H$ the projection map. By \cite{MondinoNaber}, there exists $x \in X$ such that $\mathrm{Tan}_x(X)=\{\bb{R}^{d'}\}$ and $\mathrm{Tan}_{\pi(x)}(X/H)=\{\bb{R}^{d}\}$. By Lemma \ref{L|Gruppilimite}, for a sequence $r_i \uparrow + \infty$, we have the pointed equivariant Gromov-Hausdorff convergence
        \[
        (X,r_i\sd,H,x) \to (\bb{R}^{d'},\sd_{\mathrm{eu}},G,0)\, ,
        \]
        where $G < \mathrm{Iso}(\bb{R}^{d'})$, and $G$ contains a subgroup isomorphic to $\bb{R}$.
        By Theorem \ref{T|EquivaraintConvergence}, $\bb{R}^d \cong \bb{R}^{d'}/G$, so that $d \leq d'-1$, as claimed.
    \end{proof}
\end{lemma}

We record two well-known consequences of the splitting theorem for $\RCD(0,N)$ spaces from \cite{gigli2013splittingtheoremnonsmoothcontext}. Since their proofs are standard we omit them.

\begin{lemma} \label{L|linequotient}
     Let $(X,\sd,\m)$ be an $\RCD(0,N)$ space with a closed subgroup $H<\mathrm{Iso}(X)$ such that $X/H$ is homeomorphic (and hence isometric) to $\bb{R}$. Then, there exists a splitting $X \cong Y \times \bb{R}$ such that the $H$-action is trivial on the $\R$ factor. 
\end{lemma}

\begin{lemma} \label{L|compactquotients}
     Let $(X,\sd,\m)$ be a non-compact $\RCD(0,N)$ space such that $\mathrm{Iso}(X)$ contains a closed subgroup $H$, with $X /H$ compact. Then, $X \cong Y \times \bb{R}$ as metric measure spaces.
\end{lemma}

Combining Lemma \ref{L|genericquotients}, Lemma \ref{L|linequotient}, Lemma \ref{L|compactquotients} and the classification of $\RCD$ spaces with essential dimension less than or equal to one obtained in \cite{KL}, we immediately obtain the following.

\begin{lemma} \label{P|alternative}
   Under the same assumptions and with the same notation as in Proposition \ref{prop:measuredwarped},
       either $X \cong Y \times \bb{R}$ as a metric measure space and the $H$ action is by translation on the second factor, or $X/H \cong \bb{R}_+$ as a metric space.
\end{lemma}

Thanks to Lemma \ref{P|alternative}, to complete the proof of Proposition \ref{prop:measuredwarped} we can assume without loss of generality that $X/H$ is isometric to $\R_+$. This is the most delicate case in the proof, and it is our focus for the rest of the section.
The next two lemmas show that, in this case, $X$ admits an explicit and natural parametrization (as a set) with a map $\Phi:\bb{R} \times \bb{R}_+ \to X$.

\begin{lemma} \label{P|spacedecomposition}
    Let $(X,\sd,\m)$ be an $\RCD(0,N)$ space with a closed subgroup $H <\mathrm{Iso}(X)$ such that $X/H \cong \bb{R}_+$ as metric spaces. Then there exists a ray $r:[0,+\infty) \subset X$ such that
    \begin{equation} \label{E|unionrays}
    X=\bigcup_{t \geq 0} (H\cdot r(t))\, ,
    \end{equation}
    and
    \begin{equation} \label{E|equidistanceorbits}
    \sd(H \cdot r(t),H \cdot r(s))=|t-s| 
    \quad \text{ for every } t,s \geq 0\, .
    \end{equation}
    In particular, the canonical projection to the quotient space $\pi:X\to X/H\cong \R_+$ coincides with the distance $\dist(\cdot,H\cdot r(0)):X\to \R_+$ from the orbit of $r(0)$.
    \begin{proof}
        Consider the canonical projection map $\pi:X \to X/H \cong \bb{R}_+$. Fix $s>0$, and let $H \cdot x_0$ and $H \cdot x_s$ be the preimages of $0$ and $s$ via $\pi$. 
        By definition of quotient distance, we have that
        \[
        s=\sd(H \cdot x_0,H \cdot x_s)\, .
        \]
        It is elementary to check
        that there exists a geodesic $\gamma_s:[0,s] \to X$ realizing the distance between $H\cdot x_0$ and $H\cdot x_s$. Modulo translating $\gamma_s$ with an element of $H$, we may assume without loss of generality that $\gamma_s(0)=x_0$.

        We claim that with such choice $\pi(\gamma_s(t))=t$ for every $t \in [0,s]$.
        To this aim observe that $\pi:X \to X/H$ is a $1$-Lipschitz map, so that $\pi \circ \gamma_s$ is a $1$-Lipschitz curve. At the same time $\pi(\gamma_s(0))=0$ and $\pi(\gamma_s(s))=s$, so that $\pi(\gamma_s(t))=t$ for every $t \in [0,s]$. It follows that
        \[
        \sd(H \cdot \gamma_s(t_1),H \cdot \gamma_s(t_2))=|t_1-t_2| \quad \text{for every } t_1,t_2 \in [0,s]\, .
        \]
        Choose a sequence $s_i \uparrow + \infty$ and consider the corresponding geodesics $\gamma_{s_i} \subset X$. Modulo passing to a subsequence, they converge to a limiting ray $r:[0,+\infty)\to X$, with the property that $\pi(r(t))=t$ for every $t \geq 0$ and
        \[
        \sd(H \cdot r(t_1),H\cdot r(t_2))=|t_1-t_2| \quad \text{for every } t_1,t_2 \in \bb{R}_+\, .
        \]
    Since the image of $\pi \circ r$ is the whole $\bb{R}_+$, also \eqref{E|unionrays} follows.
    \end{proof}
    \end{lemma}

\begin{lemma} \label{C|Existencebijection}
    Let $(X,\sd,\m)$ be an $\RCD(0,N)$ space with a closed subgroup $H <\mathrm{Iso}(X)$ isomorphic to $\bb{R}$ such that $X/H \cong \bb{R}_+$.
    Let $r \subset X$ be the ray given by Lemma \ref{P|spacedecomposition}. Then the map
    \[
    \Phi:\R \times \bb{R}_+ \to X\, , \quad \Phi(t,s):=t \cdot r(s)
    \]
    is bijective. Moreover, denoting by $\sd(\cdot,H\cdot r(0))$ the distance from the orbit of $r(0)$ in $X$, the inverse $\Psi:=\Phi^{-1}:X \to \R \times \bb{R}_+$ can be written as $\Psi(x)=(f(x),\sd(x,H\cdot r(0))$ for some function $f:X \to \R$. Furthermore, $f:X \to \R$ is equivariant with respect to the $H\cong\R$-action on $X$ and the standard $\R$-action on $\R$ by translations.
    \begin{proof}
        The map $\Phi$ is surjective by \eqref{E|unionrays}. Thus we only need to check the injectivity. Let $(t_1,s_1),(t_2,s_2) \in H \times \bb{R}_+$ be such that 
        \[
        t_1 \cdot r(s_1)=t_2 \cdot r(s_2)\, .
        \]
        Then $r(s_1)$ and $r(s_2)$ belong to the same orbit of $H$ in $X$. Thus $s_1=s_2$ by \eqref{E|equidistanceorbits}. Since the action of $H$ is free (because the only compact subgroup of $H$ is the trivial group), it follows, $t_1=t_2$, thus concluding the proof that $\Phi$ is bijective. To prove the representation formula for the inverse $\Psi:=\Phi^{-1}$ we observe that, given any $x=\Phi(t,s)=t\cdot r(s)$,
        \begin{equation}
            \sd(x,H\cdot r(0)) = \sd(t\cdot r(s),H\cdot r(0))=\sd(H\cdot r(s),H\cdot r(0))= s\, ,
        \end{equation}
        where we applied \eqref{E|equidistanceorbits} for the last equality.
    \end{proof}
\end{lemma}

Our next key goal is to show that the function $f:X \to \bb{R}$ introduced in Lemma \ref{C|Existencebijection} above is harmonic. Such conclusion, and its proof, are of course motivated by the observation we made in Remark \ref{rmk:secondcoordharmonic}. Nevertheless, the full argument is quite technical and it requires a few intermediate steps. We start from a mild regularity result for $f$. It will be our starting point for the forthcoming computations.

\begin{lemma} \label{P|fLip}
   Under the same assumptions and with the same notation as in Lemma \ref{P|spacedecomposition},
    let $f:X \to \bb{R}$ be the map given by Lemma \ref{C|Existencebijection}. Then, $f \in \Lip_{\mathrm{loc}}(X)$.
    \begin{proof}
        Let $K \subset \subset X$ and assume by contradiction that there exist distinct points $x_i,y_i \in K$ such that
        \begin{equation} \label{E|NoLip}
        \frac{|f(x_i)-f(y_i)|}{\sd(x_i,y_i)} \to + \infty \quad \text{as } i \uparrow + \infty\, .
        \end{equation}
        Using the same notation as in Lemma \ref{C|Existencebijection}, we can write $x_i=t_i^x \cdot r(s_i^x)$ and $y_i=t_i^y \cdot r(s_i^y)$ for some $t_i^x,t_i^y\in\R$ and $s_i^x,s_i^y\in\R_+$, so that \eqref{E|NoLip} reads as
         \begin{equation} \label{E|NoLip2}
        \frac{\sd(x_i,y_i)}{|t_i^x-t_i^y|} \to 0 \quad \text{as } i \uparrow + \infty\, .
        \end{equation}
        By the equivariance of $f$ we can assume without loss of generality that $t_i^y=0$ for every $i\in\N$. Moreover, we can also assume that $t_i^x>0$ for every $i\in\N$.
        By the triangle inequality, we get that
        \begin{align}     \sd  (r(s_i^y),1\cdot r(s_i^y))        
         \leq & \Bigg\lfloor \frac{1}{t_i^x}  \Bigg\rfloor \sd \Big(r(s^y_i), (t_i^x) \cdot r(s^y_i) \Big)
       \\
       &+\sd \Bigg( \Bigg\lfloor \frac{1}{t_i^x} \Bigg\rfloor t_i^x \cdot r(s^y_i),1 \cdot r(s^y_i)\Bigg)\, .
        \end{align}
        We claim that the right-hand side converges to zero as $i \uparrow + \infty$. If the claim holds, we readily get a contradiction using that $H \cong \bb{R}$ acts freely on $X$. 
        To prove the claim, using \eqref{E|equidistanceorbits} we can control the first summand as
        \begin{align}
        \Bigg\lfloor \frac{1}{t_i^x}  \Bigg\rfloor \sd \Big(r(s^y_i), t_i^x \cdot r(s^y_i) \Big) \leq & \frac{1}{t_i^x}  \Big( \sd(x_i,y_i)+\sd(r(s_i^x),r(s_i^y))\Big) \\
        \leq &2 \frac{\sd(x_i,y_i)}{t_i^x}\, , 
        \end{align}
        and note that the last term converges to zero by \eqref{E|NoLip2}. Hence, to conclude, we need to show that 
        \[
        \sd \Bigg( \Bigg\lfloor \frac{1}{t_i^x} \Bigg\rfloor t_i^x \cdot r(s^y_i),1 \cdot r(s^y_i)\Bigg) \to 0\, .
        \]
        This follows combining that
        \[
        \Bigg\lfloor \frac{1}{t_i^x} \Bigg\rfloor t_i^x \to 1\, ,
        \]
        and that the sequence $r(s_i^y)$ is precompact in $X$.
    \end{proof}
\end{lemma}

\begin{corollary}\label{rmk:essdim}
    The map $\Phi:\R\times \R_+ \to X$ introduced in Lemma \ref{C|Existencebijection} is a homeomorphism. Moreover, $\mathrm{dim}_{\mathrm{e}}(X)=2$.    
    \begin{proof}
    The first conclusion follows from Lemma \ref{C|Existencebijection} and Lemma \ref{P|fLip}.
    In particular, $X$ is homeomorphic to $\R_+\times\R$ and thus it has topological dimension equal to $2$.  
    Combining with \cite[Theorem 1.1]{MondinoNaber}, it follows that the essential dimension of $X$ is at most $2$. 
    At the same time, Lemma \ref{L|genericquotients} implies that $\mathrm{dim}_{\mathrm{e}}(X) \geq 2$. 
\end{proof}
\end{corollary}

We continue by studying the structure of the measure $\meas$.

\begin{lemma} \label{L|measure}
    Under the same assumptions and with the same notation as in Lemma \ref{C|Existencebijection}, we have the following.
    There exists a non-decreasing function $h \in \Lip_{\mathrm{loc}}((0,+\infty)) \cap \sL^1_{\mathrm{loc}}(\bb{R}_+)$ such that, for every $A \subset \subset X$,
    \begin{equation} \label{E|measure1}
    \m(A)=\int_{\bb{R}_+}h(s) \, \Leb^1(f(A \cap \{\sd_{H \cdot r(0)}=s\})) \, ds\, .
    \end{equation}
    Moreover, $h$ is strictly positive on $(0,+\infty)$.
    \begin{proof}
        Let $A \subset \subset X$. By the coarea formula applied to the function $\sd_{H \cdot r(0)}$,
        \[
        \m(A)=\int_{\bb{R}_+}P(\{\sd_{H \cdot r(0)}>s\},A) \, ds\, ,
        \]
        where $P(\{\sd_{H \cdot r(0)}>s\},A)$ denotes the perimeter of the set $\{\sd_{H \cdot r(0)}>s\}$ evaluated on $A$.
        We claim that there exists a nonnegative function $h \in \sL^1_{\mathrm{loc}}(\bb{R}_+)$ independent of $A$ such that for $\Leb^1$-a.e. $s \in \bb{R}_+$ we have
        \begin{equation} \label{E|identitymeasures}
        P(\{\sd_{H \cdot r(0)}>s\},A)
        =
        h(s) \,\Leb^1(f(A \cap \{\sd_{H \cdot r(0)} = s\}))\, .
        \end{equation}
        Since the measure $P(\{\sd_{H \cdot r(0)}>s\},\cdot)$ is concentrated on the set $\{\sd_{H \cdot r(0)} = s\}$, and the restriction $f_{|\{\sd_{H \cdot r(0)}=s\}}:\{\sd_{H \cdot r(0)}=s\} \to \bb{R}$ is bijective by Lemma \ref{C|Existencebijection}, \eqref{E|identitymeasures} follows if we can prove the identity of measures
        \[
        f_{\#} P(\{\sd_{H \cdot r(0)}>s\},\cdot)=h(s) \Leb^1 \quad \text{for }\Leb^1 \text{-a.e. }s \in \bb{R}_+\, .
        \]
        Fix $s \geq 0$. For every $a<b$ and $t \in  \bb{R}$, by equivariance of $f$ and invariance of the perimeter we have 
        \[
        f_{\#} P(\{\sd_{H \cdot r(0)}>s\},\cdot) (a,b)=f_{\#} P(\{\sd_{H \cdot r(0)}>s\},\cdot) (a+t,b+t)\, .
        \]
        Since the only translation-invariant measures on $\bb{R}$ are multiples of $\Leb^1$, $f_{\#} P(\{\sd_{H \cdot r(0)}>s\},\cdot)=h(s) \Leb^1$ for some $g(s)\ge 0 $. Since this holds for every $s \geq 0$, \eqref{E|measure1} holds for some function $h \in \sL^1_{\mathrm{loc}}(\bb{R}_+)$. 
        \medskip
        
        We are left to show that $h \in \Lip_{\mathrm{loc}}((0,+\infty))$ and that $h$ is increasing and strictly positive on $(0,+\infty)$. Let $\bb{Z} < H \cong \bb{R}$ be the copy of $\bb{Z}$ in $H$ generated by $1 \in H \cong \bb{R}$. Consider the quotient space $X/\bb{Z}$ with its canonically induced $\RCD(0,N)$ structure $(X/\bb{Z},\sd',\m')$ (cf.~Remark \ref{rmk:doublequotient}). Let $f_{\bb{Z}}:X/\bb{Z} \to H/\bb{Z}$ be defined such that $f_{\bb{Z}}(\bb{Z}\cdot x):=\bb{Z} \cdot f(x)$, and let $\tilde{r} \subset X/\bb{Z}$ be the projection of $r$ under the canonical quotient map.
        Observe that 
        \begin{equation}
            \bigcup_{t \in (0,1)\, ,\,\, s \in \bb{R}_+} t \cdot r(s)
        \end{equation}is a fundamental domain of the covering $\pi:X \to X/\bb{Z}$. Therefore, applying \eqref{E|measure1}, we obtain that for every $A \subset \subset X/\bb{Z}$
        \begin{equation} \label{E|measure2}
        \m'(A)=\int_{\bb{R}_+} h(s)\, \Leb^1(f_{\bb{Z}}(A \cap \{\tilde{\sd}_{(H/\bb{Z}) \cdot \tilde{r}(0)}=s\})) \, ds\, .
        \end{equation}
        Consider now the quotient $X/H \cong(X/\bb{Z})/(H/\bb{Z})$, with its $\RCD(0,N)$ structure $(X/H,\sd'',\m'')$ (cf.~Remark \ref{rmk:doublequotient}). Since $X/H \cong \bb{R}_+$, \cite[Theorem A.2]{CavallettiMilman} implies that $\m''=\overline{h} \Leb^1$ for a non-decreasing, locally Lipschitz, and strictly positive function $\overline{h}:(0,+\infty) \to \bb{R}$. On the other hand, for every $s \geq 0$, we have that
        \[
        f_{\bb{Z}}( \{\tilde{\sd}_{(H/\bb{Z}) \cdot \tilde{r}(0)}=s\})= H/\bb{Z}
        \]
        and therefore
        \[
        \Leb^1(f_{\bb{Z}}( \{\tilde{\sd}_{(H/\bb{Z}) \cdot \tilde{r}(0)}=s\}))=1\, . 
        \]
        Thanks to \eqref{E|measure2}, we infer that $\overline{h}=h$, thus completing the proof.
    \end{proof}
\end{lemma}

The following corollary is an immediate consequence of Lemma \ref{L|measure}.

\begin{corollary} \label{C|measurenull}
    Under the same assumptions and with the same notation as in Lemma \ref{C|Existencebijection}, if
    $A \subset \subset X$ is a Borel set, then $\m(A)=0$ if and only if $\Leb^2(\Psi(A))=0$.
\end{corollary}

For our next purposes it is helpful to study the behaviour of the tangent spaces of $(X,\dist,\meas)$.

\begin{lemma}\label{lemma:tangents}
Let $(X,\sd,\m)$ be an $\RCD(0,N)$ 
    space with a closed subgroup $H <\mathrm{Iso}(X)$ isomorphic to $\R$ and such that $X/H \cong \bb{R}_+$ as metric spaces. Let $\pi:X\to [0,+\infty)$ be the canonical projection to the quotient. Then, we have the following:
    \begin{itemize}
        \item[i)] every $x\in \pi^{-1}((0,+\infty))$ is a regular point of $(X,\dist,\meas)$ and the (unique) equivariant blow-up of $(X,\dist,\meas,x,H)$ is $(\R^2,\dist_{\mathrm{eucl}},\Leb^2)$ endowed with an $\R$-action by translations;
        \item[ii)] for every $x\in \pi^{-1}(\{0\})$ the blow-up of $(X,\dist,\meas)$ at $x$ is homeomorphic to $\R\times\R_+$; in particular $x$ is not a regular point.
    \end{itemize}
\end{lemma}

\begin{proof}
If $x\in \pi^{-1}((0,+\infty))$, then the blow-up of $X/H\cong \R_+$ at $\pi(x)$ is isomorphic to $\R$. By Theorem \ref{T|EquivaraintConvergence} and Lemma \ref{L|linequotient}, every equivariant blow-up of $(X,\dist,\meas,x,H)$ admits a splitting $X=Y\times \R$ with the blow-up action being trivial on the $\R$-factor. By lower semicontinuity of the essential dimension under pmGH convergence, see \cite[Theorem 1.5]{KitabeppuLSC}, and Corollary \ref{rmk:essdim}, $\dim_{\mathrm{e}}(Y)\le 1$.
By Lemma \ref{L|Gruppilimite} the limit group contains a closed subgroup isomorphic to $\R$. Item (i) thus follows from the classification of $1$-dimensional $\RCD$ spaces, see \cite{KL}.

If $\pi(x)=0$, then every blow-up of $X/H\cong \R_+$ at $\pi(x)$ is isometric to $\R_+$. By Lemma \ref{L|Gruppilimite} the limit group $H_\infty$ of every equivariant blow-up $(X_\infty,H_\infty)$ of $(X,\dist,\meas,x,H)$ contains a closed subgroup $H'_\infty <H_\infty$ isomorphic to $\R$. 
By Corollary \ref{rmk:essdim}, lower semicontinuity of the essential dimension again, and Lemma \ref{L|genericquotients}, the quotient $X_\infty/H'_\infty$ has dimension at most $1$. At the same time, $X_\infty/H_\infty \cong \bb{R}_+$ by Theorem \ref{T|EquivaraintConvergence}, so that $X_\infty/H_\infty'$ is non-compact. In particular, $X_\infty/H_\infty'$ is either a ray or a line. If $X_\infty/H_\infty' \cong \bb{R}_+$, then $X_\infty$ is homeomorphic to $\bb{R} \times \bb{R}_+$ by Corollary \ref{rmk:essdim}.
If $X_\infty/H'_\infty \cong \bb{R}$, then $X_\infty$ is isometric to $\bb{R}^2$ by Lemma \ref{L|linequotient}. We show that this is not possible. Let $r \subset X$ be the ray given by Lemma \ref{P|spacedecomposition} and assume without loss of generality that $x=r(0)$. As $(X,\sd/r_i) \to X_\infty \cong \bb{R}^2$ for some sequence $r_i \downarrow 0$, up to a subsequence, the ray $r \subset X$ converges to a ray $r_\infty \subset \bb{R}^2$, and the sets $\{y \in X: \sd(y,r) \geq r_i\}$ converge to $\{y \in \bb{R}^2: \sd_{\mathrm{eucl}}(y,r_\infty) \geq 1\}$ in Hausdorff sense in the space realizing the pGH convergence. At the same time, one can write $\{y \in X: \sd(y,r) \geq r_i\}$ as the union of two sets at distance greater than $1$ from each other and distance smaller than $2$ from $r(0)$ in $(X,\sd/r_i)$. Hence, also $\{y \in \bb{R}^2: \sd_{\mathrm{eucl}}(y,r_\infty) \geq 1\}$ must admit such partition with $r(0)$ replaced by $r_\infty(0)$, which is a contradiction.
\end{proof}

\begin{lemma} \label{L|non-branchingtrick}
    Under the same assumptions and with the same notation as in Lemma \ref{P|spacedecomposition},
    let $\gamma:[0,T] \to  X$ be a segment. If the image of $\gamma$ has three distinct points lying in $\{\sd_{\bb{R}\cdot r(0)}=0\}$, then it is fully contained in $\{\sd_{\bb{R}\cdot r(0)}=0\}$.
    \begin{proof}
    The conclusion follows from Lemma \ref{lemma:tangents} thanks to the H\"older continuity of tangent cones in the interior of geodesics, see \cite{zbMATH08038284} after \cite{ColdingNabercones}.
    \end{proof}
\end{lemma}

\begin{remark}
Under the assumptions of Lemma \ref{P|spacedecomposition} it is possible that a segment $\gamma:[0,T]\to X$ has $\gamma(0)$ and $\gamma(T)$ lying in the singular orbit $\R\cdot r(0)$ with all the interior points lying in the regular part. Examples where this actually happens appear in \cite{PanWeiGAFA,DaiHondaPanWei}.  
\end{remark}

Motivated by Remark \ref{rmk:secondcoordharmonic}, our next goal is to show that the function $f$ we introduced before is harmonic. Such statement corresponds to Lemma \ref{T|harmonicity} below. The next two lemmas are instrumental to its proof. Oversimplifying, we are going to rely on a blow-up argument to compute the partial derivatives of any sufficiently regular test function by using $\nabla\dist_{H\cdot r(0)}$ and $\nabla f$ as an orthogonal basis for the tangent bundle of $X$.

\begin{lemma} \label{L|goodapproximtion}
     Under the same assumptions and with the same notation as in Lemma \ref{P|spacedecomposition}, 
    for every $s>0$ and every sequence $r_i \downarrow 0$, up to passing to a subsequence, there exist $\delta_i \downarrow 0$ and $\delta_i$-GH approximations
    \[
    \psi_i:(\bar{B}_{r_i}^X(r(s)),\sd/r_i) \to (\bar{B}_1^{\bb{R}^2}(0),\sd_{\mathrm{eucl}})
    \]
    satisfying the following conditions:
    \begin{enumerate}
        \item \label{Item1approximation} $\psi_i(r(s+r_i))=(0,1)$ and $\psi_i(r(s-r_i))=(0,-1)$.
        \item \label{Item2approximation}
        $\psi_i(H \cdot r(s)) \to \{y=0\}$ in Hausdorff sense in $\bb{R}^2$.
    \end{enumerate}
        Assume in addition that the function $f$ given by Lemma \ref{C|Existencebijection} is differentiable in $r(s)$ according to Theorem \ref{D|CheegerDiff}, and $\lip(f)(r(s))>0$. Modulo passing to a subsequence, the linear functions $f_\infty,\sd_\infty:\bb{R}^2 \to \bb{R}$ such that $\lip(\sd_\infty)=1$ and $\lip(f_\infty)=\lip(f)(r(s))$ 
        obtained by blowing up $\dist_{H\cdot r(0)}$ and $f$ at $r(s)$ have orthogonal gradients and they span the space of linear functions on $\bb{R}^2$.
    \begin{proof}
    Since $\mathrm{Tan}_{r(s)}(X)=\{\bb{R}^2\}$ by Lemma \ref{lemma:tangents}, up to passing to a (non-relabeled) subsequence, there exist $\delta_i \downarrow 0$, and $\delta_i$-GH approximations
    \[
    {\psi}_i:(\bar{B}_{ir_i}^X(r(s)),\sd/r_i) \to (\bar{B}_i^{\bb{R}^2}(0),\sd_{\mathrm{eucl}})\, ,
    \]
    with $\psi_i(r(s))=0$.
It is a standard observation that one can take the second components of the GH-maps $\psi_i$ to be just $(\dist_{H\cdot r(0)}-s)/r_i$, since $r(s)$ is an intermediate point of a ray for $\dist_{H\cdot r(0)}(\cdot)$ by \eqref{E|equidistanceorbits}. With such choice, condition \ref{Item1approximation} is trivially satisfied, up to a slight perturbation of the first component of the GH-maps. Condition \ref{Item2approximation} follows immediately from Lemma \ref{lemma:tangents} (i).
\medskip

    Consider now the linear maps $\sd_\infty,f_\infty:\bb{R}^2 \to \bb{R}$ obtained by blowing up $\dist_{H\cdot r(0)}$ and $f$ at $r(s)$ via Theorem \ref{D|CheegerDiff}. Since $\sd_{H \cdot r(0)}(r(s+r_i))=s+r_i$ and $\psi_i(r(s+r_i))=(0,1)$ for every $i$, it follows that $\sd_\infty((0,1))=1$. Since $|\nabla \sd_\infty | \leq 1$, $\nabla \sd_\infty$ is vertical. Since $f(r(s+r_i))=0$, $f_\infty((0,1))=0$, so that $\nabla f_\infty$ is horizontal. Since $\nabla f_\infty \neq 0$ by assumption, the set $\{\sd_\infty,f_\infty\}$ spans the space of linear functions on $\bb{R}^2$.
    \end{proof}
\end{lemma}

\begin{lemma} \label{L|keyharmonicity}
Under the same assumptions and with the same notation as in Lemma \ref{L|measure},  
    let $\phi \in \Lip_{c}(X)$, and let $\tilde{\phi}:\bb{R} \times \bb{R}_+ \to \bb{R}$ be defined as $\phi \circ \Psi^{-1}$, so that $\phi(x)=\tilde{\phi}(f(x),\sd_{\bb{R} \cdot r(0)}(x))$. Then we have the following:
    \begin{enumerate}
        \item \label{item1decomposition} For $h\Leb^1$-a.e. $s \in \bb{R}_+$ with $\lip(f) \neq 0$ on $(\sd_{\bb{R} \cdot r(0)})^{-1}(s)$,
        \begin{equation} \label{E:item1decomposition}
        \tilde{\phi}(\cdot,s) \in \Lip_c(\bb{R})\, .
        \end{equation}
        \item \label{E:item2decomposition}
        For $\m$-a.e. $x \in X$ with $\lip(f)(x) \neq 0$, $\tilde{\phi}(\cdot,\sd_{\bb{R} \cdot r(0)}(x))$ is differentiable in $f(x)$, with
        \begin{equation} \label{E|diffidentity}
        \nabla \phi \cdot \nabla f (x)=\tilde{\phi}_{x_1}(f(x),\sd_{\bb{R} \cdot r(0)}(x))|\nabla f|^2(x)\, .
        \end{equation}
        \item \label{E:item3decomposition}
        For $\m$-a.e. $x \in X$, $\tilde{\phi}(f(x),\cdot)$ is differentiable in $\sd_{\bb{R} \cdot r(0)}(x)$, with
        \begin{equation} \label{E|diffidentity2}
        \nabla \phi \cdot \nabla \sd_{\bb{R} \cdot r(0)} (x)=\tilde{\phi}_{x_2}(f(x),\sd_{\bb{R} \cdot r(0)}(x))\, .
        \end{equation}
        \end{enumerate}
        In particular, for $\m$-a.e. $x \in X$ with $\lip(f)(x) \neq 0$, $\tilde{\phi}(\cdot,\cdot)$ has partial derivatives in $(f(x),\sd_{\bb{R} \cdot r(0)}(x))$, and 
        \[
        |\nabla \phi |^2(x)=\tilde{\phi}^2_{x_2}(f(x),\sd_{\bb{R} \cdot r(0)}(x))+
        \tilde{\phi}^2_{x_1}(f(x),\sd_{\bb{R} \cdot r(0)}(x))|\nabla f|^2(x)\, .
        \]
        \begin{proof}
        \textbf{Proof of \ref{item1decomposition}:}
        Fix $\epsilon>0$. We show that, for $h\Leb^1$-a.e. $s \in \bb{R}_+$ such that $\lip(f) \geq \epsilon$ on $(\sd_{\bb{R} \cdot r(0)})^{-1}(s)$, \eqref{E:item1decomposition} holds.
            Fix such an $s \in (0,+\infty)$, and assume that for every  $x \in (\sd_{\bb{R} \cdot r(0)})^{-1}(s)$ 
            the function $f$ is differentiable in $x$. By Theorem \ref{T|CheegerDiff}, 
          $h\Leb^1$-a.e.\ $s \in \bb{R}_+$ such that $\lip(f) \geq \epsilon$ on $(\sd_{\bb{R} \cdot r(0)})^{-1}(s)$ has this property.          
            Let $t \in \bb{R}$ and consider and sequence $t_i \to t$. We show that there exists $C>0$ independent of $t_i$ and $t$ such that
            \begin{equation} \label{E|limsup}
            \limsup_{i \to + \infty} \frac{|\tilde{\phi}(t_i,s)-\tilde{\phi}(t,s)|}{|t_i-t|}<C\, ,
            \end{equation}
            which would then imply item \ref{item1decomposition}.
            Let $x_s:=r(s) \in (\sd_{\bb{R} \cdot r(0)})^{-1}(s)$. We have that
            \[
            \frac{|{\phi}(t_i \cdot x_s)-{\phi}(t \cdot x_s)|}{\sd(t \cdot x_s,t_i \cdot x_s)}<\ssf{L}(\phi)\, ,
            \]
            where $\ssf{L}(\phi)$ denotes the Lipschitz constant of $\phi$,
            so that
            \[
            \frac{|\tilde{\phi}(t_i,s)-\tilde{\phi}(t,s)|}{|t_i-t|} \frac{|t_i-t|}{\sd(t \cdot x_s,t_i \cdot x_s)}<\ssf{L}(\phi)\, .
            \]
            Hence, to conclude the proof of \eqref{E|limsup} it is enough to show that for $i$ large enough 
            \begin{equation} \label{E|eval2}
            \frac{|t_i-t|}{\sd(t \cdot x_s,t_i \cdot x_s)} \geq \frac{\epsilon}{2}>0\, .
            \end{equation}
            Observe that $r_i:=\sd(t \cdot x_s,t_i \cdot x_s) \to 0$ as $i \uparrow + \infty$ and, consider the GH maps
            \[
    \psi_i:(\bar{B}_{r_i}^X(x_s),\sd/r_i) \to (\bar{B}_1^{\bb{R}^2}(0),\sd_{\mathrm{eucl}})\, ,
    \]
    introduced in Lemma \ref{L|goodapproximtion}. Up to to subsequences, since $f$ is differentiable in $x_s$ with $\lip(f)(x) \geq \epsilon$, there exists a linear function $f_\infty:\bb{R}^2 \to \bb{R}$ with $\lip(f_\infty) \geq \epsilon$, such that
    \begin{equation} \label{E|diffconv}
    \Big\|\frac{f-f(x_s)}{r_i}-f_\infty \circ \psi_i \Big\|_{\infty,\, ,\bar{B}^X_{r_i}(x_s)} \to 0\, 
    \end{equation}
    as $i\to\infty$.
    Thanks to
    Lemma \ref{L|goodapproximtion}, 
    \begin{equation} \label{E|evaluationf}
    |f_\infty (\pm 1,0)| \geq \epsilon\, .
    \end{equation}
    The same lemma gives that  $\psi_i((t_i-t) \cdot x_s)$ converges to a point in $\{(\pm 1,0)\}$. Evaluating \eqref{E|diffconv} in $(t_i-t) \cdot x_s$, we deduce
    \[
    \Big| \frac{t_i-t}{\sd(t \cdot x_s,t_i \cdot x_s)}-f_\infty (\pm 1,0) \Big | \to 0\, ,
    \]
    which, combined with \eqref{E|evaluationf}, gives \eqref{E|eval2}. This concludes the proof of item \ref{item1decomposition}.
\medskip

    \textbf{Proof of \ref{E:item2decomposition}:}. Thanks to Lemma \ref{L|measure} and item \ref{item1decomposition}, for $\m$-a.e.\ $x \in X$ with $\lip(f)(x) \neq 0$, $\tilde{\phi}(\cdot,\sd_{\bb{R} \cdot r(0)}(x))$ is differentiable in $f(x)$. Hence, we only need to check \eqref{E|diffidentity}.
    
    Let $x \in X\setminus (\R\cdot r(0))$ be any point such that $\tilde{\phi}(\cdot,\sd_{\bb{R} \cdot r(0)}(x))$ is differentiable in $f(x)$, $\lip(f)(x) \neq 0$, and the functions $f$, $\phi$, $\sd_{\bb{R} \cdot r(0)}$, $\phi-f$, and $\phi+f$ are differentiable at $x$. Let $r_i \downarrow 0$ and, up to passing to a subsequence, consider the corresponding linear functions $(\phi+f)_\infty,(\phi-f)_\infty,\phi_\infty,\sd_\infty,f_\infty:\bb{R}^2 \to \bb{R}$ obtained by blowing-up as in Definition \ref{D|CheegerDiff}. By Lemma \ref{L|goodapproximtion}, there exist $a,b \in \bb{R}$ such that $\phi_\infty=a\sd_\infty+bf_\infty$. Furthermore, $(\phi \pm f)_\infty=\phi_\infty \pm f_\infty$. Hence, using that $\nabla d_\infty \cdot \nabla f_\infty=0$ by Lemma \ref{L|goodapproximtion},
    \begin{align}
    \nabla \phi \cdot \nabla f (x) =&\frac{\lip(\phi+f)^2(x)-\lip(\phi-f)^2(x)}{4} \\
    =&\frac{|\nabla(\phi+f)_\infty|^2-|\nabla(\phi-f)_\infty|^2}{4}\\
    =&\nabla \phi_\infty \cdot \nabla f_\infty=b |\nabla f_\infty|^2=b |\nabla f|^2(x)\, .
    \end{align}
    To conclude, we only need to show that $b=\tilde{\phi}_{x_1}(f(x),\sd_{\bb{R} \cdot r(0)}(x))$. To this aim, let $t_i \downarrow 0$ be such that $\sd(x,t_i \cdot x)=r_i$. 
    Then, as $i\to\infty$, we have that
    \begin{equation} \label{E|2}
    \Big|\frac{\phi(t_i\cdot x)-\phi(x)}{\sd(x,t_i \cdot x)}-a\sd_\infty(\psi_i(t_i\cdot x))-bf_\infty(\psi_i(t_i\cdot x)) \Big| \to 0\, .
    \end{equation}
    By item \ref{Item2approximation} of Lemma \ref{L|goodapproximtion}, using also that $\nabla \sd_\infty \subset \bb{R}^2$ is vertical, we infer that $\sd_\infty(\psi_i(t_i \cdot x)) \to 0$. Rewriting \eqref{E|2} in terms of $\tilde{\phi}$, we obtain
    \begin{equation} \label{E|3}
    \Big| \frac{\tilde{\phi}(f(x)+t_i,\sd_{\bb{R} \cdot r(0)}(x))-\tilde{\phi}(f(x),\sd_{\bb{R} \cdot r(0)}(x))}{t_i} \frac{t_i}{\sd(x,t_i \cdot x)}-bf_\infty(\psi_i(t_i \cdot x)) \Big| \to 0\, .
    \end{equation}
    Arguing as in the proof of item \ref{item1decomposition}, we obtain that
    \[
    \Big|\frac{t_i}{\sd(x,t_i \cdot x)}-f_\infty(\psi_i(t_i \cdot x)) \Big| \to 0, \quad \text{and} \quad f_\infty(\psi_i(t_i \cdot x)) \to f_\infty(\pm1,0) \neq 0\, .
    \]
    Since $\tilde{\phi}(\cdot,\sd_{\bb{R} \cdot r(0)}(x))$ is differentiable in $f(x)$ by assumption, \eqref{E|3} implies that $b=\tilde{\phi}_{x_1}(f(x),\sd_{\bb{R} \cdot r(0)}(x))$, as claimed.
    \medskip
    
    \textbf{Proof of \ref{E:item3decomposition}:} For every $t \in \bb{R}$, the map $\tilde{\phi}(t,\cdot)$ is Lipschitz and hence differentiable $\Leb^1$-a.e. in $\bb{R}_+$. By Lemma \ref{C|measurenull}, for $\m$-a.e. $x \in X$, the map $\tilde{\phi}(f(x),\cdot)$ is differentiable in $\sd_{\bb{R} \cdot r(0)}(x)$. The identity \eqref{E|diffidentity2} follows repeating the argument used for \eqref{E|diffidentity} and taking into account that $|\nabla \sd_{\bb{R} \cdot r(0)}|\equiv 1$.
\medskip

    Finally, for $\m$-a.e. point $x \in X$ where \eqref{E|diffidentity} and \eqref{E|diffidentity2} hold, we have
    \begin{align}
    |\nabla \phi|^2 & =|\nabla \phi_\infty|^2=|\nabla \phi_\infty \cdot \nabla \sd_\infty|^2+\Big |\nabla \phi_\infty \cdot \frac{\nabla f_\infty}{|\nabla f_\infty|} \Big|^2 \\ 
    & =  |\nabla \phi \cdot \nabla \sd_{\bb{R} \cdot r(0)} |^2(x)+|\nabla f|^{-2} |\nabla \phi \cdot \nabla f|^2(x) \\
     & =
    \tilde{\phi}^2_{x_2}(f(x),\sd_{\bb{R}\cdot r(0)}(x))+
        \tilde{\phi}^2_{x_1}(f(x),\sd_{\bb{R} \cdot r(0)}(x))|\nabla f|^2(x)\, .
    \end{align}
        \end{proof}
\end{lemma}

\begin{remark}
Item \ref{E:item3decomposition} in Lemma \ref{L|keyharmonicity} is well known, as it follows for instance from \cite[Theorem 4.5]{CM18}. The computation of the directional derivative in the other direction is more subtle instead.
\end{remark}

The next lemma and the subsequent corollary contain the main regularity result concerning the map $f:X \to \bb{R}$ introduced in Lemma \ref{C|Existencebijection}. 

\begin{lemma} \label{T|harmonicity}
    Under the same assumptions and with the same notation as in Lemma \ref{C|Existencebijection}, the function $f:X\to\R$ is harmonic.
    \begin{proof}
    Let $\Psi:X \to \bb{R} \times \bb{R}_+$ be the map given by Lemma \ref{C|Existencebijection} (with the usual identification $H \cong \bb{R}$), and let $h \in \sL^1_{\mathrm{loc}}(\bb{R}_+)$ be the density from Lemma \ref{L|measure}.
        Let $\phi \in \Lip_{c}(X)$, and define $\tilde{\phi}:\bb{R} \times \bb{R}_+ \to \bb{R}$ by $\tilde{\phi}:=\phi \circ \Psi^{-1}$, so that $\phi(x)=\tilde{\phi}(f(x),\sd_{\bb{R} \cdot r(0)}(x))$. To prove that $f$ is harmonic we need to check that
        $
        \int_X \nabla \phi \cdot \nabla f \, d\m =0\, .
        $
        By item \ref{E:item2decomposition} of Lemma \ref{L|keyharmonicity}, we have
        \[
        \int_X \nabla \phi \cdot \nabla f \, d\m=\int_X \tilde{\phi}_{x_1}(f(x),\sd_{\bb{R} \cdot r(0)}(x))|\nabla f|^2(x) \, d\m(x)\, .
        \]
        By Lemma \ref{L|measure} and the fact that $|\nabla f|$ is invariant by translations by $H \cong \bb{R}$, which follows from the equivariance of $f$, we deduce
        \begin{align} \label{E|harmonicity1}
        \nonumber
        \int_X \tilde{\phi}_{x_1}(f(x),\sd_{\bb{R} \cdot r(0)}(x)) &|\nabla f|^2(x) \, d\m(x) \\
        & = \int_{\bb{R}_+} |\nabla f|^2(\Psi^{-1}(0,s)) h(s) \int_{\bb{R}}
        \tilde{\phi}_{x_1}(t,s) \, dt \, ds\, .
        \end{align}
        By item \ref{item1decomposition} of Lemma \ref{L|keyharmonicity}, for $h\Leb^1$-a.e. $s \in \bb{R}_+$, we have that $\tilde{\phi}(\cdot,s) \in \Lip_c(\bb{R})$. Thus $\int_{\bb{R}}
        \tilde{\phi}_{x_1}(t,s) \, dt=0$. Combining with \eqref{E|harmonicity1}, the statement follows.
    \end{proof}
\end{lemma}

\begin{corollary} \label{C|mainregularity}
    Under the same assumptions and with the same notation as in Lemma \ref{C|Existencebijection},
    the following hold:
    \begin{enumerate}
        \item \label{itemCreg1}
        $|\nabla f|^2 \in \W^{1,2}_{\mathrm{loc}}(X) \cap \sL^{\infty}_{\mathrm{loc}}(X)$, and $|\nabla f|^2$ has non-negative measure-valued Laplacian. 
        \item \label{itemCreg2}
        $|\nabla f|^2$ admits a continuous representative, which is invariant under the action of $H$ and increases along the ray $r:(0,\infty)\to X$. Equivalently, the function canonically induced by $|\nabla f|^2$ on the quotient space $X/H$ is increasing.
        \item \label{itemCreg3}
        The function induced by $|\nabla f|^2$ on the quotient $X/H \cong \bb{R}_+$ belongs to $\W^{1,2}_{\mathrm{loc}}((0,+\infty)) \cap \W^{1,2}_{\mathrm{loc}}(\bb{R}_+,hd\Leb^1)$, where $(\bb{R}_+,\sd_{\mathrm{eucl}},hd\Leb^1)$ is the quotient $\RCD(0,N)$ structure on $X/H$ (see Remark \ref{rmk:doublequotient}).
        \item \label{itemCreg4}
        The continuous representative of $|\nabla f|^2$ is strictly positive on $X\setminus (\R\cdot r(0))$.
    \end{enumerate}
    \begin{proof}
    \textbf{Proof of \ref{itemCreg1}:}
        By Lemma \ref{P|fLip}, $|\nabla f|^2 \in \sL^{\infty}_{\mathrm{loc}}(X)$. By Lemma \ref{T|harmonicity} and \cite[Lemma $3.2$]{SavareSelfImprovement}, $|\nabla f|^2\in\W^{1,2}_{\mathrm{loc}}(X) \cap \sL^{\infty}_{\mathrm{loc}}(X)$ and has non-negative measure-valued Laplacian.
\medskip

        \textbf{Proof of \ref{itemCreg2} and \ref{itemCreg3}:} We first prove simultaneously that item
        \ref{itemCreg3} holds, and that $|\nabla f|^2$ has a continuous representative in $ X\setminus (\R\cdot r(0))$.        
        As we already observed, by the equivariance of $f$ with respect to the $\R$-action, for every $a\in\R$ we have that
        \begin{equation} \label{E|CorReg1}
        |\nabla f|^2(x)=|\nabla f|^2(a \cdot x) \quad \text{for }\m \text{-a.e. }x \in X\, .
        \end{equation}
        Fix $R>0$, and let $\phi \in \Lip(X)$ be a cut-off function such that $\phi(a \cdot x)=\phi(x)$ for every $x \in X$ and every $a \in \bb{R}$, $\phi \equiv 1$ on $\{\sd_{\bb{R} \cdot r(0)}<R\}$, and $\phi \equiv 0$ on $\{\sd_{\bb{R} \cdot r(0)}>R+1\}$. Such a function can be obtained as a composition of the form $k \circ \sd_{\bb{R} \cdot r(0)}$, for some $k \in \Lip_c(\bb{R}_+)$. Since $R>0$ is arbitrary, it is enough to show that $\phi |\nabla f|^2$ has a continuous representative. 
        For every $t >0$, let $f_t:=P_t(\phi |\nabla f|^2) \in \Lip(X)$. Given any $a \in \bb{R}$, by \eqref{E|CorReg1}, the invariance of the cut-off function $\phi$ and of the heat flow, $f_t$ is invariant under the $H\cong\R$-action as well. 
        Moreover, for every open set $A \subset \subset X$, we have
        \begin{equation} \label{E|CorReg2}
        \lim_{t \to 0} \Ch(f_t,A)=\Ch(\phi |\nabla f|^2,A)<+\infty\, .
        \end{equation} 
        Consider now the quotient space $X/\bb{Z}$ with its natural $\RCD(0,N)$ structure, and observe that by invariance the functions $f_t$ naturally induce functions $f_t^{\bb{Z}}:X/\bb{Z}\to \R$. From \eqref{E|CorReg2}, for every $A \subset \subset X/\bb{Z}$, we have that
        \begin{equation} \label{E|CorReg3}
        \limsup_{t \to 0} \Ch(f_t^{\bb{Z}},A)<+\infty\, .
        \end{equation}
        Consider now the space $X/\bb{R} \cong\bb{R}_+$ obtained by quotienting $X/\bb{Z}$ by $\bb{R}/\bb{Z}$, endowed with its $\RCD(0,N)$ structure $(\bb{R}_+,\sd_{\mathrm{eucl}},\tilde\m)$. 
        Consider the functions $f_t^{\bb{R}}:X/\bb{R} \to \bb{R}$ naturally induced by $f_t^{\bb{Z}}$ on $X/\bb{R}$. By \cite[Corollary $5.5$]{MondinoQuotients} and \eqref{E|CorReg3}, for every $A \subset \subset X/\bb{R}$, we have
        \[
        \limsup_{t \to 0} \Ch(f_t^{\bb{R}},A)<+\infty\, .
        \]
        Hence, the functions $f_t^{\bb{R}}$ converge in $\sL^2_{\mathrm{loc}}$ in $(\bb{R}_+,\sd_{\mathrm{eucl}},\tilde\m)$ to a limiting function $f_0^{\bb{R}} \in \W^{1,2}_{\mathrm{loc}}(\bb{R}_+,\tilde\m)$.
        By \cite[Theorem A.2]{CavallettiMilman}, $\tilde \m=h \, d\Leb^1$, where $h^{1/(N-1)}:(0,+\infty)\to (0,+\infty)$ is a positive and concave function. Hence, for every $\epsilon>0$, the functions $f_t^{\bb{R}}$ have uniformly bounded Cheeger energies in $((\epsilon,\epsilon^{-1}),\sd_{\mathrm{eucl}},\Leb^1)$, where now we are using the standard Euclidean structure. 
        Thus the functions $f_t^{\bb{R}}$ converge to $f^\bb{R}_0$ in $\sL^2_{\mathrm{loc}}$ on $((0,+\infty),\sd_{\mathrm{eucl}},\Leb^1)$ as well. Therefore $f_0^{\bb{R}}\in \W^{1,2}_{\mathrm{loc}}(0,+\infty)$ w.r.t. the standard Euclidean structure. In particular,  $f^\bb{R}_0$ is continuous in $(0,+\infty)$. 
        At the same time, $f_0^\bb{R}$ induces a continuous \emph{pointwise defined} function $F:\{\sd_{\bb{R} \cdot r(0)}>0\} \subset X \to \bb{R}$, which is the $\sL^{2}_{\mathrm{loc}}$ limit of the functions $f_t$ in $X$. It follows that $F$ is a continuous representative of $\phi|\nabla f|^2$ in $\{\sd_{\bb{R} \cdot r(0)}>0\}$. This proves item \ref{itemCreg3}, and shows that $|\nabla f|^2$ has a continuous representative in $\{\sd_{\bb{R} \cdot r(0)}>0\} \subset X$, as we claimed.

        We now claim that $f_0^{\bb{R}}$ is monotone in an interval $(0,\eta)$ for $\eta>0$ small enough. Since $\Delta (\phi |\nabla f|^2) \geq 0$ in $\{\sd_{\bb{R} \cdot r(0)}<R\}$, by the strong maximum principle \cite[Theorem 2.8]{GigliRigonimax}, the function $f_0^{\bb{R}}$ is constant in a neighbourhood of each local maximum $(0,R)$. Hence, if $f_0^\bb{R}$ is not monotone in $(0,\eta)$, there exists a triplet
        \[
        0<a<b<c<\eta\, ,
        \]
        such that
        \begin{equation}
            f_0^\bb{R}(b) < \min\{f_0^\bb{R}(a),f_0^\bb{R}(c)\}.
        \end{equation}
        Thus, if $f_0^\bb{R}$ is not monotone in any $(0,\eta)$, there exist two triplets
        \[
        0<a<b<c<a'<b'<c'<R\, ,
        \]
        such that
        \begin{equation}
            f_0^\bb{R}(b) < \min\{f_0^\bb{R}(a),f_0^\bb{R}(c)\},
            \quad f_0^\bb{R}(b') < \min\{f_0^\bb{R}(a'),f_0^\bb{R}(c')\}\, .
        \end{equation}
        It follows that ${f_0^\bb{R}}_{|[b,b']}$ has a global maximum point $m \in (b,b')$. Thus, by the strong maximum principle again, $f_0^\bb{R}$ must be constant in $[b,b']$, a contradiction.
        This shows that there exists $\eta>0$ such that $f_0^\bb{R}$ is monotone in $(0,\eta)$. 
        Since $\phi |\nabla f|^2$ is essentially bounded, $f_0^{\bb{R}}$ is bounded. Thus $f_0^{\bb{R}}$ extends to a continuous monotone function on $[0,R)$. Hence, $\phi|\nabla f|^2$ has a continuous representative on $X$. Using again that $\Delta (\phi|\nabla f|^2) \geq 0$ on $\{\sd_{\bb{R} \cdot r(0)}<R\} \subset X$, it follows that $f_0^{\bb{R}}$ is increasing in $[0,R)$.
        Since by construction $f_0^{\bb{R}}$ is the function induced on $X/\bb{R}$ by the continuous representative of $\phi |\nabla f|^2$, $|\nabla f|^2$ is increasing on the quotient $X/\bb{R}$ or, equivalently, on $r$.
\medskip

        \textbf{Proof of \ref{itemCreg4}:}
        If $|\nabla f|^2$ vanishes on $r(s)$ for some $s>0$, then it vanishes on the whole $\{\sd_{\bb{R} \cdot r(0)} < s \}$ by invariance and monotonicity. If this is the case, then $f$ must be constant in $\{\sd_{\bb{R} \cdot r(0)} < s \}$, a contradiction.
    \end{proof}
\end{corollary}

Roughly speaking, we proved in Lemma \ref{L|keyharmonicity} that gradients in $(X,\dist)$ can be computed as on a manifold with a warped product metric of the form $dr^2+|\nabla f(r)|^{-2}dx^2$. By Lemma \ref{L|measure}, the measure $\meas$ on $(X,\dist)$ has a similar warped product structure. By Corollary \ref{C|mainregularity}, the warping function has a reasonable amount of regularity.
Making rigorous the isomorphic identification of $(X,\dist,\meas)$ with a measured warped product is still quite technical at this stage, but the proof below follows a by now well established strategy essentially introduced in the proof of the splitting theorem for $\RCD(0,N)$ spaces in \cite{gigli2013splittingtheoremnonsmoothcontext}.

\begin{lemma} \label{L|PsiBilip}
   Under the same assumptions and with the same notation as in Lemma \ref{C|Existencebijection}, 
    the restriction
    \[
    \Psi^{-1}_{|\bb{R} \times (0,+\infty)}:\big(\bb{R} \times (0,+\infty),\sd_{\mathrm{eucl}}\big) \to  \big(\{\sd_{\bb{R} \cdot r(0)}>0\},\sd \big) \subset (X,\sd)
    \]
    is locally Lipschitz.
    \begin{proof}
        We denote by $f:X \to \bb{R}$ the map introduced in Lemma \ref{C|Existencebijection} under the usual identification $H \cong \bb{R}$. We recall that $\Psi^{-1}(a,b)=a \cdot r(b)$. Let $K \subset \subset \bb{R} \times (0,+\infty)$ be open. 
        We claim that there exists $C>1$ such that 
        \begin{equation}\label{E|LipLoc1}
        |f(a \cdot y)-f(y)|=|a| \geq C \sd(y,a \cdot y) \quad \text{for every } y,a \cdot y \in \Psi^{-1}(K)\, .
        \end{equation}
        If the claim holds, given pairs $(a,b),(\tilde{a},\tilde{b}) \in K$, we deduce
        \begin{multline}
        \sd(\Psi^{-1}(a,b),\Psi^{-1}(\tilde{a},\tilde{b}))  \leq |b-\tilde{b}|+\sd(a \cdot r(b),[(\tilde{a}-a)+a] \cdot r(b)) 
        \\
        \leq |b-\tilde{b}|+C^{-1}|a-\tilde{a}| \leq \sqrt{2}C^{-1} \sd_{\mathrm{eucl}}((a,b),(\tilde{a},\tilde{b}))\, ,
        \end{multline}
        which implies the statement.
        To finish the proof we are left to prove \eqref{E|LipLoc1}. By item \ref{itemCreg4} of Corollary \ref{C|mainregularity}, there exists $0<\tilde{C}<1$ such that $|\nabla f| \geq \tilde{C}$ on $\Psi^{-1}(K)$. Let $y \in \Psi^{-1}(K)$ and let $a >0$ be such that $a \cdot y \in \Psi^{-1}(K)$. Let $y_i \to y$ be differentiability points of $f$ in $\Psi^{-1}(K)$. Since $|\nabla f| \geq \tilde{C}>0$ on $\Psi^{-1}(K)$, using Lemma \ref{L|goodapproximtion}, for every $i\in\N$ we have that
        \[
        \frac{a/j}{\sd(a/j\cdot y_i,y_i)}
        =
        \frac{f(a/j\cdot y_i)-f(y_i)}{\sd(a/j\cdot y_i,y_i)} \geq \tilde{C}/2 \quad \text{for every } j \in \bb{N} \text{ large enough.}
        \]
        Hence
        \begin{align}
        a=&j(a/j) \geq \frac{\tilde{C}}{2} \Big(\sd(y_i,a/j \cdot y_i)+\cdots +\sd\big([(j-1)a/j]\cdot y_i,a\cdot y_i\big)\Big)\\ 
        \geq& \frac{\tilde{C}}{2} \sd(y_i,a\cdot y_i)\, .
        \end{align}
        Passing to the limit as $y_i \to y$, we obtain \eqref{E|LipLoc1} when $a>0$. The case $a <0$ is analogous.
    \end{proof}
\end{lemma}

\begin{remark}
 The examples discussed in \cite{PanWeiGAFA,DaiHondaPanWei} show that, in general, the map $\Psi:(X,\dist)\to \R\times \R_+$ need not be globally biLipschitz.   
\end{remark}

By the invariance of $|\nabla f|$ with respect to the $H\cong\R$-action we can identify it with a function on $X/\R\cong\R_+$. We will tacitly make such identification from now on.
On $\bb{R} \times \bb{R}_+$, we consider the $C^0$ Riemannian metric (with extended values) given by $g(x_1,x_2):= |\nabla f|^{-2}(x_2) \, dx_1^2+dx_2^2$. 
Since the continuous representative of $|\nabla f|$ is strictly positive on $(0,+\infty)$ by Corollary \ref{C|mainregularity}, every two points in $\bb{R} \times \bb{R}_+$ can be joined by a curve of finite length w.r.t.\ $g$. Recall that the length of an absolutely continuous curve $\gamma \in AC([0,1] ,\bb{R} \times \bb{R}_+)$ with respect to $g$ is defined as
\begin{equation}
    \mathrm{length}_g(\gamma):=\int_0^1 g({\gamma}',{\gamma}')^{1/2} \, dt\, .
\end{equation}
Thus, the metric $g$ induces a real-valued distance $\sd_g$ on $\bb{R} \times \bb{R}_+$ by setting
\begin{equation}
    \sd_g(a,b):=\inf \big\{ \mathrm{length}_g(\gamma) :\gamma \in AC([0,1] ,\bb{R} \times \bb{R}_+) \text{ and }\gamma(0)=a,\gamma(1)=b\big\}\, .
\end{equation}

The next two lemmas establish some technical variants of the Sobolev-to-Lipschitz property on $(X,\dist)$ and $(\R\times\R_+,\dist_g)$ which are going to helpful later to complete the proof of Proposition \ref{prop:measuredwarped}.

\begin{lemma} \label{L|SobToLip}
    Consider the $C^0$ Riemannian metric (with extended values) $g(x_1,x_2):= |\nabla f|^{-2}(x_2) \, dx_1^2+dx_2^2$ and the induced real-valued distance $\sd_g$ on $\bb{R} \times \bb{R}_+$.   
    If $\phi \in \Lip_{\mathrm{loc}}(\bb{R} \times (0,+\infty),\sd_g)$ has $\lip(\phi) \leq 1$ at $\Leb^2$-a.e. point, then $\phi$ admits a $1$-Lipschitz extension to $(\bb{R} \times \bb{R}_+,\sd_g)$.
    \begin{proof}
        Let $a,b \in \bb{R} \times (0,+\infty)$. We need to show that $|\phi(a)-\phi(b)| \leq \sd_g(a,b)$. Without loss of generality we assume that $a,b \in \bb{R} \times (1,+\infty)$. By the monotonicity of $|\nabla f|^2$ obtained in Corollary \ref{C|mainregularity}, the set $\bb{R} \times (1,+\infty)$ is geodesically convex in $(\bb{R} \times \bb{R}_+,\sd_g)$. Hence, for every $\epsilon>0$, there exists a smooth curve $\gamma_\epsilon:[0,1] \to \bb{R} \times (1,+\infty)$ such that $\sd_g(a,b) \geq \length(\gamma_\epsilon)-\epsilon$. By perturbing $\gamma_\epsilon$, one can construct a curve $\tilde{\gamma}_\epsilon:[0,1] \to \bb{R} \times (0,+\infty)$ such that $\sd_g(a,b) \geq \length(\tilde{\gamma}_\epsilon)-2\epsilon$ and, for $\Leb^1$-a.e.\ $t \in [0,1]$, $\lip(\phi)(\tilde{\gamma}_\epsilon(t)) \leq 1$.
         Therefore 
         \begin{align*}
         |\phi(a)-\phi(b)| \leq \int_0^1 (\phi \circ \tilde{\gamma}_\epsilon)' \, dt \leq  \int_0^1 g(\tilde{\gamma}_\epsilon',\tilde{\gamma}_\epsilon')^{1/2} \, dt\\
         = \length(\tilde{\gamma}_\epsilon) \leq \sd_g(a,b)+2\epsilon\, ,
         \end{align*}
         where for the second inequality we used \cite[Proposition 4.22]{MondinoRyborz} which guarantees that 
         \begin{equation}
             \lim_{t \to 0} \frac{\sd_g \big(\tilde{\gamma}_\epsilon(t+s),\tilde{\gamma}_\epsilon(s) \big)}{s}=
             g(\tilde{\gamma}_\epsilon'(s),\tilde{\gamma}_\epsilon'(s))^{1/2} 
             \quad \text{ for } \Leb^1 \text{-a.e. }  s \in (0,1)\, .
\end{equation}
         Since $\epsilon>0$ is arbitrary, the statement follows.
    \end{proof}
\end{lemma}

\begin{lemma} \label{L|isometryfix}
Under the same assumptions and with the same notation as in Lemma \ref{P|spacedecomposition},
let $\phi \in \Lip_{\mathrm{loc}}(X \setminus \{\sd_{\bb{R}\cdot r(0)}=0\}) \cap C(X)$ be such that $|\nabla \phi| \leq 1$ at $\m$-a.e. point. Then, $\phi$ is $1$-Lipschitz.
\begin{proof}
    Our first claim is that $\lip (\phi) \leq 1$ everywhere on $X \setminus \{\sd_{\bb{R}\cdot r(0)}=0\}$. To this aim, fix $x \in X \setminus \{\sd_{\bb{R}\cdot r(0)}=0\}$, and let $\eta \in \Lip_c( X \setminus \{\sd_{\bb{R}\cdot r(0)}=0\})$ be a function which is equal to $1$ on a neighbourhood of $x$. By the Bakry-Émery contraction estimate from \cite[Theorem 6.2]{AGS2},
    \begin{equation}
        \lip(P_t(\eta \phi)) \leq P_t(\lip(\eta \phi)) \quad \text{pointwise}\, .
    \end{equation}
    On a sufficiently small neighbourhood of $x$, $P_t(\lip(\eta \phi)) \leq 1+o(1)$ and $P_t(\eta \phi) $ converges uniformly to $\phi$ as $t\to 0$, by the well-known properties of the heat flow on $\RCD(K,N)$ spaces. It follows that, possibly considering a smaller neighbourhood of $x$, $\phi$ is a uniform limit of $(1+o(1))$-Lipschitz functions. The claim follows.

    Let $a,b \in X$ and let $\gamma \subset X$ be any geodesic connecting them. By Lemma \ref{L|non-branchingtrick}, either $\gamma$ has at most two points in $ \{\sd_{\bb{R}\cdot r(0)}=0\}$, or it is fully contained in $ \{\sd_{\bb{R}\cdot r(0)}=0\}$. In the former case, $\phi \circ \gamma$ is absolutely continuous and
    \begin{equation}
        |\phi(a)-\phi(b)| \le \int_0^{\length(\gamma)}(\phi \circ \gamma)'(t) \, dt \leq \length(\gamma) \leq \sd(a,b)\, .
    \end{equation}
    So, assume that $\gamma \subset  \{\sd_{\bb{R}\cdot r(0)}=0\}$. Fix $\epsilon>0$. Let $x'$ be the midpoint of $\gamma$. Let $x''$ be the point at distance $\epsilon$ from $x'$ on the ray starting from $x'$  which realizes the distance from $\{\sd_{\bb{R}\cdot r(0)}=0\}$.
    Let $\gamma_a$ and $\gamma_b$ be geodesics from $a$ to $x''$, and from $x''$ to $b$. These geodesics have at most two points each in $ \{\sd_{\bb{R}\cdot r(0)}=0\}$, and they have length at most $\sd(a,b)/2 + \epsilon$. Let $\eta$ be the curve (parameterized by arc length) obtained by joining these two geodesics, and observe that
    \begin{equation}
        |\phi(a)-\phi(b)| \le \int_0^{\length(\eta)}(\phi \circ \eta)'(t) \, dt \leq \length(\eta) \leq \sd(a,b)+2\epsilon\, .
    \end{equation}
    The conclusion follows by the arbitrariness of $\epsilon$.
\end{proof}
\end{lemma}

We are ready to complete the proof of the measured warped product structure.

\begin{proof}[Proof of Proposition \ref{prop:measuredwarped}]
Thanks to Lemma \ref{P|alternative} it suffices to prove the following statement: 
\medskip

Let $(X,\sd,\m)$ be an $\RCD(0,N)$ space such that $\mathrm{Iso}(X)$ contains a closed subgroup $H<\mathrm{Iso}(X)$ isomorphic to $\bb{R}$ and $X/H \cong \bb{R}_+$. Let $\Psi:X \to \bb{R} \times \bb{R}_+$, $f:X \to \bb{R}$ and $r:[0,+\infty)\to X$ be as in Lemma \ref{C|Existencebijection}, under the usual identification $H \cong \bb{R}$. On $\bb{R} \times \bb{R}_+$, consider the $C^0$ Riemannian metric (with extended values) given by $g(x_1,x_2):= |\nabla f|^{-2}(x_2) \, dx_1^2+dx_2^2$ and the induced real-valued distance $\sd_g$, and let $h$ be the density from Lemma \ref{L|measure}.
    Then
    \[
    \Psi:(X,\sd,\meas) \to (\bb{R} \times \bb{R}_+,\sd_g,h(x_2)dx_1dx_2)
    \]
    is a measure preserving isometry.   
\medskip

The fact that $\Psi$ is measure preserving readily follows from Lemma \ref{L|measure}. In the remaining part of the proof we prove that it is an isometry.   
    We denote by $\nabla_g$ gradients w.r.t.\ $g$, and by $|\cdot|_g$ the norm of a vector w.r.t.\ $g$.  
        Let $\phi \in \Lip(X)$ be such that $\ssf{L}(\phi) \leq 1$ and define $\tilde{\phi}:\bb{R} \times \bb{R}_+ \to X$ as $\tilde{\phi}:=\phi \circ \Psi^{-1}$. We claim that $\tilde{\phi}$ is $1$-Lipschitz w.r.t.\ $\sd_g$.
        By Lemma \ref{L|PsiBilip} and Corollary \ref{C|mainregularity}, $\tilde{\phi} \in \Lip_{\mathrm{loc}}(\bb{R} \times (0,+\infty),\sd_g)$. 
        By Lemma \ref{L|keyharmonicity}, combined with the fact that $|\nabla f|^2 \neq 0$ on $\m$-a.e.\ $x \in X$ by Corollary \ref{C|mainregularity},
        \begin{align} \label{E|isometrychiave}
        \nonumber
       |\nabla_g \tilde{\phi}|^2_g(\Psi(x))
        =&
        \tilde{\phi
        }^2_{x_2}
        (f(x)
        ,\sd_{\bb{R} \cdot r(0)}(x))+
        \tilde{\phi}^2_{x_1}(f(x),\sd_{\bb{R} \cdot r(0)}(x))|\nabla f|^2(x)
        \\
        =&
        |\nabla \phi |^2(x) \quad
        \text{for } \m \text{-a.e. } x \in X\, .
        \end{align}
        By \cite[Proposition 4.25]{MondinoRyborz}, $\lip_g(\tilde{\phi})=|\nabla_g \tilde{\phi}|_g $ on $\Leb^2$-a.e.\ point of $\bb{R} \times \bb{R}_+$, where $\lip_g(\tilde{\phi})$ denotes the local Lipschitz constant of $\tilde{\phi}$ w.r.t.\ $\sd_g$. By Corollary \ref{C|measurenull} and \eqref{E|isometrychiave}, we infer that $\lip_g(\tilde{\phi})\leq 1$ on $\Leb^2$-a.e.\ point of $\bb{R} \times (0,+\infty)$. Hence $\tilde{\phi}$ is $1$-Lipschitz in $(\bb{R} \times \bb{R}_+,\sd_g)$ by Lemma \ref{L|SobToLip}.

        Fix any $x \in X$ and let $\sd_x:X \to \bb{R}$ be defined as $\sd_x(y):=\sd(x,y)$. The function $\sd_x \circ \Psi^{-1}$ is $1$-Lipschitz in  $(\bb{R} \times \bb{R}_+,\sd_g)$ by the previous part of the proof. Hence
        \begin{equation}\label{eq:isoineq1}
        \sd(x,y)=|\sd_x \circ \Psi^{-1} (\Psi(x))-\sd_x \circ \Psi^{-1}(\Psi(y))| \leq \sd_g(\Psi(x),\Psi(y))\, .
        \end{equation}

        Let $\tilde{\phi} \in \Lip(\bb{R} \times \bb{R}_+,\sd_g)$ be a $1$-Lipschitz function, and let $\phi:X \to \bb{R} \times \bb{R}_+$ be defined as $\phi:= \tilde{\phi} \circ \Psi$. We claim that $\phi$ is $1$-Lipschitz. 
        By Lemma \ref{P|fLip}, $\phi \in \Lip_{\mathrm{loc}}(X \setminus \{\sd_{\bb{R}\cdot r(0)}=0\})$.
        Arguing as before, we deduce that $|\nabla \phi| \leq 1$ on $\m$-a.e.\ point of $X$.  Thus $\phi$ is $1$-Lipschitz on $X$ by Lemma \ref{L|isometryfix}.
        Consider any point $a \in \bb{R} \times \bb{R}_+$, and let $\sd_{g,a}:\bb{R} \times \bb{R}_+ \to \bb{R}$ be defined as $\sd_{g,a}(b):=\sd_g(a,b)$. By the previous considerations $\sd_{g,a} \circ \Psi$ is $1$-Lipschitz on $X$. Hence,
        \begin{equation}\label{eq:isoineq2}
        \sd_g(a,b)=|\sd_{g,a} \circ \Psi (\Psi^{-1}(a))-\sd_{g,a} \circ \Psi (\Psi^{-1}(b))| \leq \sd(\Psi^{-1}(a),\Psi^{-1}(b))\, .
        \end{equation}
        The combination of \eqref{eq:isoineq1} and \eqref{eq:isoineq2} shows that $\Psi$ is an isometry, thus completing the proof.
    \end{proof}

\subsubsection{Proof of Theorem \ref{T|splittingkesimo}, base step: splitting}\label{subsubsec:splittingk=1}

Theorem \ref{T|keystepsplitting} below corresponds to the base case $k=1$ in Theorem \ref{T|splittingkesimo}. Proving it is the goal of this section.

\begin{theorem} \label{T|keystepsplitting}
 Let $(X,\sd,\m)$ be an $\RCD(0,N)$ space such that $\mathrm{Iso}(X)$ contains a closed subgroup $H<\mathrm{Iso}(X)$ isomorphic to $\bb{R}$, $\dim_{\mathrm{e}}(X)\le 2$, and  
    \begin{equation}\label{eq:growthbasestep}
    \int_1^{+\infty} \frac{t^2}{\m(B_t(p))} \, dt =+\infty\, .
    \end{equation}
   Then $X \cong Y \times \bb{R}$ as metric measure spaces for some $\RCD(0,N-1)$ space $(Y,\dist_Y,\meas_Y)$.
\end{theorem}

The proof of Theorem \ref{T|keystepsplitting} relies on the measured warped product rigidity obtained in Proposition \ref{prop:measuredwarped}, and on the equivalence between synthetic and distributional lower bounds on the weighted Ricci curvature valid for sufficiently regular weighted Riemannian metrics (Theorem \ref{T|VanessaMondino} below).
\medskip

Given a Riemannian metric $g$ on $\bb{R}^n$ such that $g_{ij} \in \W^{1,2}_{\mathrm{loc}}(\bb{R}^n) \cap C(\bb{R}^n)$, one can compute the distributional Christoffel symbols, defined as
\begin{equation} \label{E|Christ}
\Gamma^{k}_{ij}
:= \frac{1}{2} \, g^{k\ell}
\left(
  \partial_i g_{j\ell}
  + \partial_j g_{i\ell}
  - \partial_\ell g_{ij}
\right)\, .
\end{equation}
Thanks to the regularity of $g$, the Christoffel symbols belong to $\sL^2_{\mathrm{loc}}(\bb{R}^n)$.
Hence, one can define the distributional Ricci curvature as
\begin{equation} \label{E|Christ1}
(\Ric_g)_{ij}
:= \partial_k \Gamma^{k}_{ij}
  - \partial_j \Gamma^{k}_{ik}
  + \Gamma^{k}_{ij} \Gamma^{\ell}_{k\ell}
  - \Gamma^{k}_{i\ell} \Gamma^{\ell}_{jk}\, .
\end{equation}
Given a function $V \in \W^{1,2}_{\mathrm{loc}}(\bb{R}^n) \cap C(\bb{R}^n)$, its distributional Hessian is defined as
\begin{equation} \label{E|Christ2}
(\Hess_g(V))_{ij}
:= {\partial^2_{i,j} V}
- \Gamma^{k}_{ij} \, {\partial_k V}\, .
\end{equation}

The next result is proved in \cite{MondinoRyborz}.

\begin{theorem} \label{T|VanessaMondino}
    Let $(M^n,g)$ be a Riemannian manifold with boundary with $g$ a continuous Riemannian metric with Christoffel symbols in $\sL^2_{loc}(M)$. Let $V \in C^0(M) \cap \W^{1,2}_{loc}(M)$, and assume that $(M,\sd_g,e^{-V} d\ssf{Vol}_g)$ is an $\RCD(0,\infty)$ space. Then,
    \[
    \Ric_M+\Hess_V \geq 0 \quad \text{on } M \setminus \partial M \text{ in distributional sense}.
    \]
\end{theorem}

Consider now a function $\phi \in \W^{1,2}_{\mathrm{loc}}((0,+\infty)) \cap C((0,+\infty))$, and the manifold $(0,+\infty)\times\R$ equipped with the warped product metric $g(s,t):=ds^2+\phi^2(s) dt^2$. We will denote as $(0,+\infty) \times_\phi \bb{R}$ the resulting Riemannian manifold. Specializing \eqref{E|Christ} to this simplified setting, we deduce
\begin{align*}
    & \Gamma^s_{t,t}=-{\phi \phi'}, \\
     & \Gamma^t_{t,s}=\Gamma^t_{s,t}=\frac{\phi'}{\phi}, 
     \\
     & \Gamma^s_{s,s}=\Gamma^t_{t,t}=\Gamma^t_{s,s}=0\, .
\end{align*}
Hence, using \eqref{E|Christ1}, we obtain
\begin{equation} \label{E|Christ3}
    (\Ric_{(0,+\infty) \times_\phi  \bb{R}})_{t,t}:=-(\phi \phi')'+(\phi')^2\, .
\end{equation}
Similarly, using \eqref{E|Christ2}, we deduce
\begin{equation} \label{E|Christ4}
    (\Hess_{(0,+\infty) \times_\phi  \bb{R}}(V))_{t,t}:=\phi\phi'V'\, .
\end{equation}

    \begin{proof}[Proof of Theorem \ref{T|keystepsplitting}]
  
        Thanks to Proposition \ref{prop:measuredwarped}, we can assume that $(X,\dist,\meas)$ is equivariantly isomorphic to a measured warped product  $\R_+\times_{\phi}\R$, with metric and measure given respectively by
\begin{equation}\label{eq:measurewarpedprod2}
    g(r,x):=\di r^2+\phi(r)^2\di x^2\, ,\quad \quad \meas:=h(r)\di r\di x\, ,
\end{equation} 
for some increasing function $h \in \Lip_{\mathrm{loc}}((0,+\infty)) \cap \sL^1_{\mathrm{loc}}(\bb{R}_+)$ and some continuous and strictly positive function $\phi:(0,+\infty)\to (0,+\infty)$, with the $\R$-action being by translation on the $x$ component and $\dist$ the length distance induced by $g$. Under such identification, the equivariant function $f:X\to\R$ introduced in Lemma \ref{C|Existencebijection} admits the clear expression $f(r,x)=x$ and we have $\phi=|\nabla f|^{-1}$, by direct computation.

The $2$-dimensional Hausdorff measure with respect to $g$ on $X$ can be written in coordinates as $\phi(r)drdx$. By \eqref{eq:measurewarpedprod2}, there exists a function  $F:\bb{R}_+ \to \bb{R}_+$ such that $F\phi \in \sL^1_{\mathrm{loc}}(\bb{R}_+,\Leb^1)$ and 
\begin{equation}
    \meas=F(r)\phi(r)drdx\, .
\end{equation}

        Let $\sd_{g,0}:\bb{R} \times \bb{R}_+ \to \bb{R}_+$ be defined by $\sd_{g,0}(a):=\sd_g(a,\bb{R} \times \{0\})$. According to \cite[Section 7]{CavallettiMilman}, consider the disintegration of $\m$ induced by the $1$-Lipschitz function $\sd_{g,0}$. By \cite[Main Theorem 1.1]{CavallettiMilman} and \cite[Theorem A.2]{CavallettiMilman},
        \begin{equation}
         (F\phi)^{\frac{1}{N-1}}:(0,+\infty)\to (0,+\infty)   
        \end{equation}
        is a concave function. Since 
        $\phi \in C^0 (0,+\infty) \cap \W^{1,2}_{\mathrm{loc}}(0,+\infty)$
        and it is non-increasing and strictly positive by Corollary \ref{C|mainregularity}, $F \in C^0 [0,+\infty) \cap \W^{1,2}_{\mathrm{loc}}(0,+\infty)$.
        Using again that 
        $\phi\in C^0 (0,+\infty) \cap \W^{1,2}_{\mathrm{loc}}(0,+\infty)$ we deduce that the distributional Christoffel symbols of $g$ are in $\sL^2_{\mathrm{loc}}((0,+\infty)\times \bb{R})$. Note that this condition is independent of whether we consider as reference measure $\m$ or $\Leb^2$. The remaining part of the proof is divided into three steps. In Step 1 we use the equivalence between synthetic and distributional lower bound on the Ricci curvature to deduce that $\phi$ is super-harmonic on the weighted half-lines $([\delta,+\infty),\dist_{\mathrm{eucl}},F\Leb^1)$ for every $\delta>0$. In Step 2 we reformulate the growth condition \eqref{eq:growthbasestep} into a growth condition for $F$. In Step 3 we combine the results of the previous two steps to show that $\phi$ is constant. This is clearly enough to complete the proof of the splitting of $(X,\dist,\meas)$.

\medskip

        \textbf{Step 1:} We claim that for every $\delta>0$, and every non-negative function $\psi \in \Lip_c([\delta,+\infty)])$, there holds
        \begin{equation} \label{E|superharmonicity}
        \int_{[\delta,+\infty)} \phi'\psi' \,  F d \Leb^1 \geq 0\, .
        \end{equation}
        
        Note that the integral in \eqref{E|superharmonicity} is well defined and finite thanks to Corollary \ref{C|mainregularity}.        
        Set $V:=-\log(F):\bb{R}_+ \to \bb{R}$. By Theorem \ref{T|VanessaMondino}, since $(X,\dist,\meas)$ is an $\RCD(0,N)$ space, and thus in particular it is an $\RCD(0,\infty)$ space, we have that
        \begin{equation} \label{E|DistrRicci}
        \Ric_{\bb{R}_+ \times_\phi  \bb{R}}+\Hess_{\bb{R}_+ \times_\phi  \bb{R}}(V) \geq 0 \quad \text{distributionally on } (0,+\infty)\times \bb{R} \, .
        \end{equation}
        Using \eqref{E|Christ3} and \eqref{E|Christ4}, we obtain that
        \begin{equation}\label{eq:distrphiV}
            -(\phi \phi')'+(\phi')^2+\phi\phi'V' \geq 0 \quad \text{in distributional sense in } (\delta,+\infty)\, .
        \end{equation}
        By a straightforward manipulation, \eqref{eq:distrphiV} is equivalent to
        \[
        \phi''-\phi'V' \leq 0 \quad \text{in distributional sense on } (\delta,+\infty)\, ,
        \]
        which in turn is equivalent to
        \begin{equation} \label{E|superharmonicityInterior}
            \int_\delta^{+\infty} \phi' \psi' \, F d\Leb^1 \geq 0 \quad \text{for every non-negative } \psi \in \Lip_c((\delta,+\infty))\, .
        \end{equation}
         
       We continue proving \eqref{E|superharmonicity} for any non-negative function $\psi \in \Lip_c([\delta,+\infty))$. For every $\epsilon>0$, consider the function $\eta_{\epsilon}$ such that $(1-\eta_\epsilon) \in \Lip_c([\delta,+\infty))$ defined by 
        \begin{equation}
        \eta_\epsilon(t):=1 \wedge (1-\epsilon^{-1}(t-\delta-\epsilon)) \vee 0 \, .
        \end{equation}
        Using the Leibniz rule for the second equality, we have that
        \begin{align*} 
       & \int_{[\delta,+\infty)} \phi'\psi' \,  F d \Leb^1 \\
        &= \int_{[\delta,+\infty)}  \phi' (\psi \eta_\epsilon)' \, F d \Leb^1 
        + \int_{[\delta,+\infty)}  \phi' (\psi (1-\eta_\epsilon))'\, F d \Leb^1
        \\
        & = \int_{[\delta,+\infty)} \phi'(\psi  \eta_\epsilon'+\eta_\epsilon  \psi') \, F d \Leb^1 
        + \int_{[\delta,+\infty)}  \phi' (\psi (1-\eta_\epsilon))'\, F d \Leb^1. 
        \end{align*}
        Since $\phi' \eta_\epsilon' \geq 0$ in $(0,+\infty)$ by Corollary \ref{C|mainregularity}, passing to the limit in the previous equation as $\epsilon \to 0$ and taking into account \eqref{E|superharmonicityInterior}, we deduce \eqref{E|superharmonicity}.
\medskip

        \textbf{Step 2:}
        Let $A:\bb{R}_+ \to \bb{R}_+$ be defined by $A(r):=\int_0^rF(t)\, dt$. Note that $A(r)$ is the measure of the $r$-ball centered at $0$ in the weighted half-line $([0,+\infty),\dist_{\mathrm{eucl}},F\Leb^1)$.
        We claim that
        \begin{equation} \label{E|weightedparabolicity}
        \int_1^{+\infty} \frac{t}{A(t)} \, dt=+\infty\, .
        \end{equation}
        
        By triangle inequality, there holds
        \[
        B^g_r((0,0)) \supset \Big\{(a,b) \in \bb{R} \times \bb{R}_+: b \in (0,r/2), \, a \in \Big(-r \frac{\phi(b)^{-1}}{2},r \frac{\phi(b)^{-1}}{2} \Big) \Big\}\, .
        \]
        Hence,
        \[
        \m(B_r^g(0,0)) \geq \int_0^r \int_{-r \frac{\phi(s)^{-1}}{2}}^{r \frac{\phi(s)^{-1}}{2}} \phi(s)F(s) \, dx ds=r \int_0^rF(t) \, dt=rA(r)\, .
        \]
        Therefore, using \eqref{eq:growthbasestep}, we conclude that 
        \[
        \int_1^{+\infty} \frac{t}{A(t)} \, dt \geq
        \int_1^{+\infty} \frac{t^2}{\m \big(B^g_t((0,0))\big)} \, dt
        =+\infty\, ,
        \]
        thus proving \eqref{E|weightedparabolicity}.
\medskip

        \textbf{Step 3:} We claim that $\phi$ is constant. As noted before, this is clearly enough to prove that $(X,\dist,\meas)$ splits a line.

        By Lemma \ref{L|intbypartsyau} below and Step 1, for every $\delta>0$ we have that
        \begin{equation} \label{Eintbypartsparabolicity}
        \int_\delta^{+\infty} \psi^2 \frac{{\phi'}^2}{(1+\phi)^2 } \, Fdt \leq 16\int_\delta^{+\infty} {\psi'}^2 \, Fdt \, ,
        \end{equation}
        for every non-negative function $\psi \in \Lip_c([\delta,+\infty))$.
        Fix any $R>\delta$. We claim that $\phi$ is constant on $[\delta,R]$. Using the notation of Step 2, consider the functions $\psi_i \in \Lip_c([\delta,+\infty))$ defined by
        \[
        \psi_i(t):=\Big(\int_{R \vee t \wedge i}^i \frac{s}{A(s)} \, ds\Big)\Big/ \Big(\int_{R}^i \frac{s}{A(s)} \, ds \Big)\, .
        \]
        We claim that 
        \begin{equation} \label{E|vanishingparabolicity}
            \int_\delta^{+\infty} {\psi'_i}^2 \, Fdt \to 0 \quad \text{as } i \to +\infty\, ,
        \end{equation}
        which, combined with \eqref{Eintbypartsparabolicity}, would imply that $\phi$ is constant on $[\delta,R]$.
      To prove the claim, we can compute
      \begin{align*}
          \int_\delta^{+\infty} {\psi'_i}^2 \, Fdt &=
          \Big( \int_R^{i} t^2F(t)/A^2(t) \, dt \Big) 
          \Big/ \Big(\int_{R}^i \frac{s}{A(s)} \, ds \Big)^2 \\
          & = \Big(-\frac{i^2}{A(i)}+\frac{R^2}{A(R)}+2\int_R^i \frac{s}{A(s)} \, ds \Big)\Big/ \Big(\int_{R}^i \frac{s}{A(s)} \, ds \Big)^2\, .
      \end{align*}
      The right-hand side of the previous equality converges to zero as $i \to + \infty$ thanks to Step 2, thus proving \eqref{E|vanishingparabolicity}. It follows from \eqref{Eintbypartsparabolicity} that $\phi$ is constant on $[\delta,R]$. By the arbitrariness of $\delta$ and $R$, $\phi$ is constant on $(0,+\infty)$, as we claimed.
    \end{proof}

The next lemma follows by adapting Yau's result from \cite{Y} (see also \cite[Lemma $7.1$]{LiG}).

\begin{lemma} \label{L|intbypartsyau}
    Let $f,u \in \W^{1,2}_{\mathrm{loc}}(\bb{R}^n) \cap C(\bb{R}^n)$ be positive functions such that
    \begin{equation}
        \int_{\bb{R}^n} \nabla u \cdot \nabla \phi \, f d\Leb^n \geq 0 \quad \text{for every non-negative } \phi \in \W^{1,2}(\bb{R}^n) \cap C_c(\bb{R}^n)\, .
    \end{equation}
    Then,
    \[
    \int_{\bb{R}^n} \phi^2 \frac{|\nabla u|^2}{(1+u)^2} \, f d\Leb^n \leq 16 \int_{\bb{R}^n}|\nabla \phi|^2 \, f d \Leb^n\, , 
    \]
    for every non-negative $\phi \in \W^{1,2}(\bb{R}^n) \cap C_c(\bb{R}^n)$.
    \begin{proof}
        We define $v \in \W^{1,2}_{\mathrm{loc}}(\bb{R}^n) \cap C(\bb{R}^n)$ as $v:=(1+u)^{-1}$. Using the chain rule and the Leibniz rule, for every $\phi \in \W^{1,2}(\bb{R}^n) \cap C_c(\bb{R}^n)$, we get
        \[
        0 \geq \int_{\bb{R}^n} \nabla (\phi^2 v) \cdot \nabla v \, f d\Leb^n
        =2 \int_{\bb{R}^n} \phi v \nabla \phi \cdot \nabla v \, f d\Leb^n +\int_{\bb{R}^n}\phi^2 |\nabla v|^2 \, f d\Leb^n\, ,
        \]
        so that
        \begin{equation} \label{E1ss}
        \int_{\bb{R}^n} \phi^2 |\nabla v|^2 \, f d\Leb^n \leq 2 \int_{\bb{R}^n}|\phi v 
        \nabla \phi \cdot \nabla v| \, f d\Leb^n\, .
        \end{equation}
        Applying Young's inequality we get
        \[
        2 \int_{\bb{R}^n} \phi v \nabla \phi \cdot \nabla v \, f d\Leb^n
        \leq \frac{1}{2} \int_{\bb{R}^n} \phi^2 |\nabla v|^2 \, f d\Leb^n
        +8 \int_{\bb{R}^n} |\nabla \phi|^2 v^2 \, f d\Leb^n\, ,
        \]
        which combined with \eqref{E1ss} gives
        \[
        \int_{\bb{R}^n} \phi^2 |\nabla v|^2 \, f d\Leb^n \leq
        16 \int_{\bb{R}^n} |\nabla \phi|^2 v^2 \, f d\Leb^n\, .
        \]
        Recalling the definition of $v$, the statement follows.
    \end{proof}
\end{lemma}

\subsubsection{Proof of Theorem \ref{T|splittingkesimo}, inductive step}\label{subsubsec:induction}

The goal of this section is complete the inductive proof of Theorem \ref{T|splittingkesimo} by proving the inductive step.
The next three lemmas are needed to carry out such inductive step.

\begin{lemma} \label{L|inductivestep}
    Let $(X,\sd,\m)$ be an $\RCD(0,N)$ space satisfying the following conditions:
    \begin{enumerate}
        \item $\mathrm{dim}_{\mathrm{e}}(X) \leq k+1$;
        \item $\int_1^{+\infty} \frac{t^{k+1}}{\m(B_t(p))} \, dt =+\infty$ for some $p \in X$;
        \item There exists a closed subgroup $\R^k\cong H <\mathrm{Iso}(X)$.
    \end{enumerate} 
     Let $H'' \leq H$ be a closed subgroup isomorphic to $\bb{R}$. Consider the quotient space $(X/H'',\sd'')$ equipped with the measure $\m''$ provided by Remark \ref{rmk:doublequotient}. Then, denoting by $B''_t(p'')$ balls in $(X/H'',\sd'')$ centered in $p'':=H'' \cdot p$, we have that: 
     \begin{enumerate}
        \item\label{item1induction} $\mathrm{dim}_{\mathrm{e}}(X/H'') \leq k$.
        \item\label{item2induction} $\int_1^{+\infty} \frac{t^{k}}{\m''(B''_t(p''))} \, dt =+\infty$.
        \item\label{item3induction}  $\mathrm{Iso}(X/H'')$ contains a closed subgroup isomorphic to $\bb{R}^{k-1}$.
    \end{enumerate} 
    \begin{proof}
        Item \ref{item1induction} immediately follows from Lemma \ref{L|genericquotients}.             To prove item \ref{item3induction}, it suffices to observe that the quotient $\bb{R}^{k-1}\cong H/H''$ is isomorphic to a closed subgroup of $\mathrm{Iso}(X/H'')$.
        Hence, we only need to show item \ref{item2induction}, i.e.~that
        \begin{equation} \label{E|quotientstepfinale}
            \int_1^{+\infty} \frac{t^{k}}{\m''(B''_t(p''))} \, dt =+\infty\, .
        \end{equation}
        Let $H' \leq H''$ be the subgroup isomorphic to $\bb{Z}$ generated by $1 \in H'' \cong \bb{R}$, and consider the natural $\RCD(0,N)$ structure $(X/H',\sd',\m')$. We denote by $B'_t(p')$ balls in $(X/H',\sd')$ centered in $p':=p+H'$. As an intermediate step, we show that there exists $c>0$ such that
        \begin{equation} \label{E|quotientintermediofin}
            \m''(B''_t(p'')) \leq \m'(B'_{t+c}(p')) \quad \text{for every } t >0\, .
        \end{equation}
        To this aim, we observe that by definition of $\m''$,
        \begin{equation} \label{E|quotientintermedio}
        \m''(B''_t(p''))=\m'\{x \in X/H':\sd'(x+H''/H',p') \leq t\}\, .
        \end{equation}
        Since $H''/H'$ is a compact group, there exists $c>0$, such that $p'+H''/H' \subset B'_c(p')$. Combining with \eqref{E|quotientintermedio}, \eqref{E|quotientintermediofin} follows.
        Our next claim is that there exists $c'>0$ such that
        \begin{equation} \label{E|growthcover}
        t \m'(B'_{t/2}(p')) \leq c' \m(B_t(p)) \quad \text{for every } t \geq 0 \text{ large enough}\, . 
        \end{equation}
        Note that the combination of \eqref{E|growthcover} and \eqref{E|quotientintermediofin} would give \eqref{E|quotientstepfinale}. To prove the claim, consider the projection $\pi:X \to X/H'$, and observe that it is a covering map. 
       Let $F \subset X$ be the fundamental domain containing $p$ provided by Lemma \ref{L|fundamentaldomain}. By construction,
         \begin{equation}
            \m'(B'_t(p')) \leq \m(B_t(p) \cap F)\, ,\quad \forall t>0\, .
         \end{equation}
        Assume w.l.o.g.\ that the generator $h$ of $H' \cong \bb{Z}$ satisfies $\sd(p,p+h)=1$. By the triangle inequality,
        \begin{align}
            B_t(p) \supset \bigcup_{i=1}^{\lfloor t/2 \rfloor} h^i+(B_{t/2}(p) \cap F)\, .
        \end{align}
        Since $\m((h^i+F) \cap F)=0$ by Lemma \ref{L|fundamentaldomain} again, we deduce that
        \begin{equation}
            \m(B_t(p)) \geq \lfloor t/2 \rfloor \m'(B'_{t/2}(p'))\, ,
        \end{equation}
        thus yielding \eqref{E|growthcover}.
    \end{proof}
\end{lemma}

The following lemma is well known, thus its elementary proof will be omitted.

\begin{lemma}\label{lem:splittingISO}
    Let $(X,\sd)$ be a metric space, and let $H < \mathrm{Iso}(X \times \bb{R})$ be a closed subgroup. Consider the resulting quotient space $(X \times \bb{R})/H$, and assume that it is isometric to $Y \times \bb{R}$ for some metric space $(Y,\sd_y)$. Let $\pi:X \times \bb{R} \to Y \times \bb{R}$ be the projection on the quotient, and assume that $\pi((t,x))=(t,y)$ for some $x \in X$, $y \in Y$ and every $t\in \R$. Then
    \begin{equation}
    H<\mathrm{Iso}(X) \times \{\mathrm{Id}\}<\mathrm{Iso}(X\times\R)\, .
    \end{equation}
\end{lemma}

Applying iteratively Lemma \ref{lem:splittingISO} and the splitting theorem for $\RCD(0,N)$ spaces from \cite{gigli2013splittingtheoremnonsmoothcontext}, and using that lines can be lifted under metric submetries, we obtain the following: 

\begin{lemma} \label{L|liftsplit}
    Let $(X,\sd,\m)$ be an $\RCD(0,N)$ space, and let 
    $H<\mathrm{Iso}(X)$ be a closed subgroup.
    Assume that $X/H \cong \bb{R}^k \times Y$ isometrically. Then, $X \cong \bb{R}^k \times Y'$ isometrically, and
    \begin{equation}
        H<\mathrm{Iso}(X') \times \{\mathrm{Id}\}<\mathrm{Iso}(Y'\times\R^k)\, .
    \end{equation}
\end{lemma}

    \begin{proof}[Proof of Theorem \ref{T|splittingkesimo}]
        We argue by induction. The base step corresponds to Theorem \ref{T|keystepsplitting}.
        For the inductive step, we assume that the statement holds for every $j$ such that $1\le j\le  k-1$, and we show that it holds for $k$. Let $H' \cong \bb{R}$ be a closed subgroup of $H$, and consider the quotient space $(X/H',\sd'',\m'')$ with the induced $\RCD(0,N)$ structure, as in Remark \ref{rmk:doublequotient}.
        By Lemma \ref{L|inductivestep} the space $X/H'$ satisfies the hypotheses of the theorem for $k-1$. Thus, by the inductive hypothesis, $X/H' \cong \bb{R}^{k-1} \times Y$ as metric measure spaces, for some $\RCD(0,N-k+1)$ space $(Y,\dist_Y,\meas_Y)$. 
        By Lemma \ref{L|liftsplit}, we also have that $X \cong \bb{R}^{k-1} \times Y'$ isometrically, and that 
        \begin{equation}
        H'<\{\mathrm{Id}\}\times\mathrm{Iso}(Y')<\mathrm{Iso}(\bb{R}^{k-1}) \times \mathrm{Iso}(Y')\, .
        \end{equation}
        Hence, $Y'$ splits off a line isometrically thanks to the base case $k=1$. 
    \end{proof}

\subsubsection{Proof of Corollary \ref{cor:blowdownXbarunique}}\label{subsubsec:structureblowdowns}

We distinguish two cases, according to the volume growth rate at infinity. The case of Euclidean volume growth is much easier and does not require Theorem \ref{T|splittingkesimo}. 

\begin{lemma} \label{L|uniquetangentconemaximal}
    Let $(X,\sd,\haus^n)$ be an $\RCD(0,n)$ space with a closed subgroup $\Z^{n-2}\cong H < \mathrm{Iso}\{X\}$ and such that 
     \begin{equation} \label{E|volumegrowth}
            \lim_{s \to + \infty} \frac{\haus^n(B_s(p))}{s^n}=v>0 \quad \text{for some } p \in X\, .
        \end{equation}
    Then, either $X \cong \bb{R}^{n-1} \times Y$ isometrically, or the blow-down $X_\infty$ of $X$ is unique and isometric to  $\bb{R}^{n-2} \times Z$, for some $Z \in \{C(S^1_r),C([0,a])\}$ with $r \in (0,1)$, $a \in (0,\pi)$.    
    In the latter case, if $(X_\infty,H_\infty)$ is an equivariant blow-down of $(X,H)$, then $H_\infty < \mathrm{Iso}(\bb{R}^{n-2}) \times \mathrm{Iso}(Z)$ and $\pi_{\mathrm{Iso}(\bb{R}^{n-2})}(H_\infty)\cong\R^{n-2}$ acts by translations on $\R^{n-2}$.
    \begin{proof}
         By \eqref{E|volumegrowth} and \cite{DePhilippisGiglicone,DephilGigli} after \cite[Theorem 7.6]{ChCo1}, every blow-down $(X_\infty,\sd_\infty,p_\infty)$ of $X$ is a $n$-dimensional metric cone centered at its tip $p_\infty$. Fix one such $X_\infty$. By Lemma \ref{L|Gruppilimite}, there exists a closed subgroup of $ \mathrm{Iso}(X_\infty)$ isomorphic to $\bb{R}^{n-2}$. 
         By cone-splitting,
         $X_\infty \cong \bb{R}^{n-2} \times C(Z)$ isometrically, where $Z \in \{S^1_r,[0,a]\}$. If $Z=S^1_1$, then $X \cong \bb{R}^n$ by \cite[Theorem 1.6]{DephilGigli}. If $Z=[0,\pi]$, then $X \cong \bb{R}^{n-1} \times \bb{R}_+$ by \cite[Theorem 8.1]{zbMATH07514027}. Thus, we reduce to the case when $Z \in \{C(S^1_r),C([0,a])\}$ and $r \in (0,1), a \in (0,\pi)$.        
         By volume convergence, see \cite[Theorem 1.3]{DephilGigli} after \cite{ColdingConv,Colding1},
         $\haus^n(B_1(p_\infty))=v$. 
         Thus, the radius $r \in (0,1)$ or the parameter $a \in (0,\pi)$ such that $X_\infty \cong \bb{R}^{n-2} \times C(S^1_r)$ or $X_\infty \cong \bb{R}^{n-2} \times C([0,a])$  are determined by the asymptotic volume ratio of $X$. In particular, there are at most two non-isometric choices for the blow-down of $X$. By Lemma \ref{L|tangentconesconnected}, the blow-down of $X$ must be unique. 
\medskip

The second part of the statement now readily follows from \cite[Theorem 3.2]{Panalmoststable}.

    \end{proof} 
\end{lemma}

In the case where the volume growth is non-Euclidean, we begin by studying the blow-down of the quotient space with respect to the $\Z^{n-2}$-action.

    \begin{lemma}\label{lemma:blowdownquot}
    Let $(X,\sd,\haus^n)$ be an $\RCD(0,n)$ space with a closed subgroup $\Z^{n-2}\cong H < \mathrm{Iso}\{X\}$ and such that
     \begin{equation}
            \lim_{s \to + \infty} \frac{\haus^n(B_s(p))}{s^n}=0 \quad \text{for some } p \in X\, .
        \end{equation}
    Then, either $X \cong \bb{R}^{n-2} \times Z$ as metric measure spaces, or the blow-down of $X/H$ is unique up to pointed isometry and pointed isometric to $(\bb{R}_+,\sd_{\mathrm{eucl}},0)$.
    \begin{proof}
        Set $Y:=X/H$, and consider its induced $\RCD(0,n)$ structure $(Y,\bar{\sd},\bar{\m})$. By Theorem \ref{T|compactsplittingcovering2}, we can assume without loss of generality that $Y$ is non-compact.
        Consider any blow-down $(Y_\infty,\bar\sd_{\infty},\bar\m_{\infty},y_\infty)$ of $Y$. Up to taking a subsequence, the corresponding rescalings of $X$ converge to an equivariant blow-down $(X_\infty,\sd_\infty,\m_\infty,H_\infty,x_\infty)$ of $(X,\sd,\m,H)$.
        By Lemma \ref{L|Gruppilimite}, $H_\infty$ contains a closed subgroup isomorphic to $\R^{n-2}$.
               By Theorem \ref{T|EquivaraintConvergence}, $X_\infty/H_\infty \cong Y_\infty$ isometrically. Thus, by Lemma \ref{L|genericquotients}, $\mathrm{dim}_{\mathrm{e}}(Y_\infty) \leq 1$.
       Since $Y$ is non-compact, $\mathrm{dim}_{\mathrm{e}}(Y_\infty) =1$ and $Y_\infty$ is non-compact. Thus, by the classification of $1$-dimensional $\RCD$ spaces from \cite{KL}, $Y_\infty$ is isometric to $\bb{R}$ or $\bb{R}_+$. 
       
       If $Y$ has a blow-down isometric to $\R$, then $Y \cong \bb{R} \times Y'$, for some compact $\RCD(0,n-1)$ space $Y'$ by Lemma \ref{L|linetangentcone}. 
       By Lemma \ref{L|liftsplit}, $X \cong \bb{R} \times X'$, $H <\{\mathrm{Id}\}\times \mathrm{Iso}(X')<\mathrm{Iso}(\bb{R}) \times  \mathrm{Iso}(X')$. 
       Hence, with a slight abuse of notation, $X'/H \cong Y'$. By Theorem \ref{T|compactsplittingcovering2} again, $X'$ splits off $n-2$ Euclidean factors.

       Assume that every blow-down $(Y_\infty,y_\infty)$ of $Y$ is isometric to $\bb{R}_+$. In this case, we claim that $y_\infty=0$. If not, by taking a blow-up of $Y_\infty$ at $y_\infty$ and using that the blow-up of a blow-down is still a blow-down, we obtain a blow-down of $Y$ isometric to $\bb{R}$, thus reaching a contradiction.
    \end{proof}
    \end{lemma}

The combination of Lemma \ref{L|uniquetangentconesubmaximal} below with Lemma \ref{L|uniquetangentconemaximal} is clearly sufficient to prove Corollary \ref{cor:blowdownXbarunique}.

\begin{lemma} \label{L|uniquetangentconesubmaximal}
    Let $(X,\sd,\haus^n)$ be an $\RCD(0,n)$ space with a closed subgroup $\Z^{n-2}\cong H < \mathrm{Iso}\{X\}$ and such that
     \begin{equation}\label{eq:noneuclblowdown}
            \lim_{s \to + \infty} \frac{\haus^n(B_s(p))}{s^n}=0 \quad \text{for some } p \in X\, .
        \end{equation}
    Then, either $X \cong \bb{R}^{n-2} \times Y$ isometrically, or the blow-down $X_\infty$ of $X$ is unique up to pointed isometry and isometric to $\bb{R}^{n-2} \times Z$, where $Z \in \{\bb{R},\bb{R}_+\}$.
    Moreover, if $(X_\infty,H_\infty)$ is an equivariant blow-down of $(X,H)$, then $H_\infty^2:=\{h^2:h \in H_\infty\}$ acts freely on $X_\infty$. 
    \begin{proof}
       Set $Y:=X/H$. By Lemma \ref{lemma:blowdownquot}, we can assume without loss of generality that the blow-down of $Y$ is unique (up to pointed isometry) and pointed isometric to $(\bb{R}_+,\sd_{\mathrm{eucl}},0)$.
       Let $(X_\infty,\sd_\infty, \m_\infty,H_\infty,x_\infty)$ be an equivariant blow-down of $(X,\sd,\m,H)$. By Theorem \ref{T|EquivaraintConvergence}, $X_\infty/H_\infty=\bb{R}_+$, with the orbit $x_\infty\cdot H_\infty$ corresponding to $0 \in \bb{R}_+$. 
        By Lemma \ref{L|Gruppilimite}, $H_\infty$ contains a closed subgroup isomorphic to $\R^{n-2}$.
        By \eqref{eq:noneuclblowdown}, $(X_\infty,\sd_{\infty},\m_{\infty})$ is an $\RCD(0,n)$ space with $\dim_{\mathrm{e}}(X_\infty)\leq n-1$. Thus, by Theorem \ref{T|splittingkesimo}, $X_\infty:=\bb{R}^{n-2} \times Z'$, where $\mathrm{dim}_{\mathrm{e}}(Z')  \leq 1$. Since $X_\infty/H_\infty\cong\R_+$, by Lemma \ref{L|genericquotients} we infer that $\dim_{\mathrm{e}}(Z')=1$.
    If a blow-down $X$ is isometric to $\bb{R}^{n-2} \times K$ where $K$ is compact then by taking another blow-down we obtain a contradiction.
       Therefore, all blow-downs of $X$ are isometric to $\bb{R}^{n-2} \times Z$ with $Z \in \{\R,\R_+\}$.           Moreover, if $X_\infty \cong \bb{R}^{n-2} \times \bb{R}_+$, then $x_\infty \in \partial X_\infty$. Hence, the space of \emph{pointed} blow-downs of $X$ contains at most two elements. Since such space is connected by Lemma \ref{L|tangentconesconnected}, it follows that blow-down of $X$ is unique up to pointed isometries.
       
     To prove the second part of the statement, 
    let $G_\infty$ denote a closed subgroup of $H_\infty$ isomorphic to $\bb{R}^{n-2}$ obtained from Lemma \ref{L|Gruppilimite}. Observe that $H_{\infty}$ is abelian, since all the approximating groups along the blow-down sequence are.
       The quotient group $H_\infty/G_\infty$ acts effectively by measure preserving isometries on $X_\infty/
       G_\infty$. By Lemma \ref{L|genericquotients}, $X_\infty/
       G_\infty$ has essential dimension at most $1$, and $(X_\infty/
       G_\infty)/(H_\infty/G_\infty)  \cong \bb{R}_+$. It follows that $X_\infty/
       G_\infty$ is isometric to either $\bb{R}$ or $\bb{R}_+$, and that $H_\infty/G_\infty$ is either generated by a reflection of $\bb{R}$, or the identity on $\bb{R}_+$. In both cases, if $h \in H_\infty$, then $h^2 \in G_\infty$. Since $G_{\infty}\cong\R^{n-2}$, $h^2$ acts freely on $X_{\infty}$.
    \end{proof}
\end{lemma}

\begin{corollary} \label{L|LemmaTriv}
    Let $(X,\sd,\haus^n)$ be an $\RCD(0,n)$ space with a closed subgroup $\bb{Z}^{n-2}\cong H < \mathrm{Iso}\{X\}$, and assume that $X$ does not split a factor $\R^{n-2}$.     
    Then there is $k \in \{n-2,n-1\}$ such that the blow-down of $X$ is unique (up to pointed isometry), pointed isometric to $(\bb{R}^k \times Y,(0,y))$ for some pointed metric space $(Y,\dist_Y,y)$, and $\mathrm{Iso}(\bb{R}^k \times Y) \cdot (0,y) \subset \bb{R}^k \times \{y\}$.
    \begin{proof}
        The conclusion follows combining Lemmas \ref{L|uniquetangentconesubmaximal} and \ref{L|uniquetangentconemaximal} with Lemma \ref{L|productisometry}.
    \end{proof}
\end{corollary}

\subsection{Equivariant harmonic functions with almost linear growth}\label{sec:equivariantharmonic}

We continue to work under the assumption that $(X,\dist,\haus^n)$ is an $\RCD(0,n)$ space such that $\mathrm{Iso}(X)$ contains a closed subgroup $H\cong \Z^{n-2}$.
Our first goal in this section is to construct an $\R^k$-valued harmonic map $u:X\to\setR^k$, with $k\in\{n-2,n-1\}$ determined by the dimension of the maximal Euclidean factor split-off by the blow-down of $X$ which is: 
\begin{itemize}
\item[i)] equivariant with respect to the $H^2<H$-action;
\item[ii)] almost-splitting up to linear transformation at every sufficiently large scale;
\item[iii)] \emph{non-degenerate} along the $H^2$-orbit of the base-point, after the linear corrections appearing in ii). 
\end{itemize}
For the precise meaning of the non-degeneracy condition in iii) we address the reader to the statement of Proposition \ref{P|slicing}. 
\medskip

Proposition \ref{P|HarmonicHuang} below addresses the first two items of the plan.

\begin{proposition} \label{P|HarmonicHuang}
    Let $(X,\sd,\haus^n)$ be an $\RCD(0,n)$ space with a closed subgroup $\bb{Z}^{n-2}\cong H < \mathrm{Iso}\{X\}$. Let $H^2:=\{h^2:h \in H\}$.
    Then there exist $k\in\{n-2,n-1,n\}$ and a harmonic map $u:X \to \bb{R}^{k}$ satisfying the following:
    \begin{enumerate}
        \item For every $\gamma \in H$ we have $u \circ \gamma^2(\cdot)=u(\cdot)+u(\gamma^2 \cdot p)$, and the map $\phi:H^2 \to \bb{R}^k$ given by $\phi(\gamma^2):=u(\gamma^2 \cdot p)$ is an injective group homomorphism.
        \item For every $\epsilon>0$ there exists $R_\epsilon>0$ such that for every $R>R_\epsilon$ there exists a positive definite lower triangular matrix $T_R \in \bb{R}^{k \times k}$ such that $T_R u:B_R(p) \to \bb{R}^{k}$ is an $\epsilon$-splitting map.
    \end{enumerate}
    \begin{proof}
        The statement holds trivially if $X \cong \bb{R}^{n-2} \times Z$. Thus we assume without loss of generality that this is not the case. By Lemmas \ref{L|uniquetangentconemaximal} and \ref{L|uniquetangentconesubmaximal}, the blow-down $X_\infty$ of $X$ is unique (up to isometry) and isometric to $\bb{R}^k \times Y$ for some $k \geq n-2$ and some metric space $(Y,\dist_Y)$ that does not split a line.
        By Theorem \ref{T|HuangHuang}, there exists a harmonic map $u=(u_1,\cdots,u_k):X \to \bb{R}^k$ vanishing in $p \in X$ such that: 
        \begin{enumerate}
            \item \label{item1HuangHuang} 
            $\{u_1,\cdots,u_k\}$ is a basis for
            \begin{multline}
                \Big\{
                v \in \W^{1,2}_{\mathrm{loc}}(X): \Delta v=0, \text{and  for every } \epsilon>0
                 \text{ there exist } R_\epsilon,C_\epsilon>0 : \\
                 v(x) \leq C_\epsilon \sd(x,p)^{1+\epsilon}+C_\epsilon \text{ for every } x \in X
                \Big\}\, .
                \end{multline}
                \item \label{item2HuangHuang}
                For every $\epsilon>0$, there exists $R_\epsilon>0$ such that for every $R>R_\epsilon$ there exists a positive definite lower triangular matrix $T_R \in \bb{R}^{k \times k}$, such that $T_R u:B_R(p) \to \bb{R}^{k}$ is an $\epsilon$-splitting map.
        \end{enumerate}
\medskip
        
        Let $\gamma \in H$. We claim that $u \circ \gamma^2(\cdot)=u(\cdot)+u(\gamma^2 \cdot p)$. To this aim observe that along any sequence of rescalings converging to the equivariant blow-down $(X_\infty,H_\infty,p_\infty)$ of $(X,H)$, $\gamma$ converges to a limiting element $\gamma_\infty \in H_\infty$ by Lemma \ref{L|convergenceelementsequi} and $\gamma^2$ converges to $\gamma_\infty^2$. Since $\gamma_\infty$ fixes $p_\infty$, by Lemmas \ref{L|uniquetangentconemaximal} and \ref{L|uniquetangentconesubmaximal}, we get that $\pi_{\mathrm{Iso}(\bb{R}^k)}(\gamma_\infty^2)=\mathrm{Id}_{\bb{R}^k}$.
        
        Consider a sequence $\epsilon_i \downarrow 0$, and corresponding $R_i \uparrow + \infty$ and $T_i \in \bb{R}^{k \times k}$ such that $T_i u:B_{R_i}(p) \to \bb{R}^{k}$ is an $\epsilon_i$-splitting map for every $i\in\N$.
        By \ref{item1HuangHuang}, there exist matrices $A_i \in \bb{R}^{k \times k}$, such that
        \begin{equation} \label{E|cambiobaseHuang}
            \frac{T_i u \circ \gamma^2 (\cdot)}{R_i}=\frac{A_i T_iu(\cdot)+T_iu(\gamma^2 \cdot p)}{R_i}\, .
        \end{equation}
        To conclude the proof of the claim, it suffices to show that $A_i=\mathrm{Id}_{\bb{R}^k}$ for every $i \in \bb{N}$.
        
        Since the sequence of maps $T_i u /R_i$ converges to $\pi_{\bb{R}^k}$ in the blow-down and $\gamma^2$ converges to an element $\gamma_\infty^2$ such that $\pi_{\mathrm{Iso}(\bb{R}^k)}(\gamma_\infty^2)=\mathrm{Id}_{\bb{R}^k}$, passing to the limit in \eqref{E|cambiobaseHuang} we obtain that $A_i \to \mathrm{Id}_{\bb{R}^k}$. Note that the matrices $A_i$ are all similar to each other by \eqref{E|cambiobaseHuang}. Thus, all their (complex) eigenvalues are equal to $1$. To conclude the proof it suffices to show that all the matrices $A_i$ are diagonalizable. To this aim, let $B \in \bb{R}^{k \times k}$ be such that
        \begin{equation} \label{E|cambiobaseHuang2}
            u \circ \gamma^2(\cdot)=Bu(\cdot)+u(\gamma^2 \cdot p)\, .
        \end{equation}
        The matrix $B$ is easily seen to be similar to the matrices $A_i$, so that it suffices to show that $B$ is diagonalizable.
        Iterating \eqref{E|cambiobaseHuang2}, we deduce
        \begin{equation} \label{E|cambiobaseHuang3}
            u \circ i\gamma^2(\cdot)=B^iu(\cdot)+u(i\gamma^2 \cdot p) \quad \text{ for every } i \in \bb{N}\, ,
        \end{equation}
        where $i\gamma^2$ denotes the sum in $H^2$ of $i$ copies of $\gamma^2$.
        Let $M \in \bb{C}^{k \times k}$ be the invertible matrix such that $MBM^{-1}=J$ is the Jordan form of $B$.
        From \eqref{E|cambiobaseHuang3} we deduce that
        \begin{equation} \label{E|cambiobaseHuang4}
            Mu \circ i\gamma^2(\cdot)=J^iMu(\cdot)+Mu(i\gamma^2 \cdot p) \quad \text{ for every } i \in \bb{N}\, .
        \end{equation}
         By the triangle inequality there exists $c>0$ such that $\sd(i\gamma^2 \cdot p,p) \leq ci$ for every $i \in \bb{N}$. Fix now $R>0$. By the Cheng-Yau gradient estimate \cite[Theorem 1.2]{HuaKellXia} and the almost linear growth assumption in item \ref{item1HuangHuang}, there exists a constant $c(R,n)>0$ such that
         \begin{equation}\label{eq:CY}
             \sup_{B_R(i\gamma^2 \cdot p)} |\nabla u| \leq c(R,n)i^{1/2} \quad \text{for every } i \in \bb{N}\, .
         \end{equation}
         Hence, from \eqref{E|cambiobaseHuang4} and \eqref{eq:CY} we deduce that 
         \begin{align} \label{E|cambiobaseHuang5}
         \nonumber
             \sup_{x \in B_R(p)} |J^i M u| = \sup_{x \in B_R(p)} |Mu (i\gamma^2\cdot x) - Mu (i\gamma^2\cdot p)| \\
             \leq |M| R \,\sup_{ B_R(i\gamma^2\cdot p)} |\nabla u| \leq Ci^{1/2}\, ,
             \end{align}
         for some constant $C>0$.
         Our next claim is that if $R>0$ is sufficiently large then
         \begin{equation}\label{eq:splitnondege}
         \Leb^{n-2}\big(u(B_R(p))\big)>0\, .
\end{equation}
         Note that the claim is almost straightforward in the smooth Riemannian case. Indeed, if $R>0$ is sufficiently large then $T_Ru$ is an almost splitting map. Thus, by a standard propagation of regularity argument originally due to \cite{Colding1}, there exists a point $x\in B_R(p)$ such that $\{\nabla u_1(x),\dots,\nabla u_k(x)\}$ are linearly independent. In the $\RCD$ case one can argue in a similar way to find a regular point $x\in B_R(p)$ and $0<r<R$ such that $T_Ru:B_r(x)\to\R^{n-2}$ is still an almost splitting map. Then one can notice that $T_Ru:B_r(x)\to\R^{n-2}$ can be extended to an almost splitting map $\overline{T_Ru}:B_{r/2}(x)\to\R^n$ (with the first $n-2$ components given by those of $T_Ru$). From \cite{BruePasqualettoSemolarectRCD} we have that $\Leb^n(\overline{T_Ru}(B_r(x)))>0$. Thus, from Fubini we deduce that $\Leb^{n-2}(B_r(x))>0$ and \eqref{eq:splitnondege} follows. 
         By \eqref{eq:splitnondege} there exist linearly independent vectors $\{e_i\}_{i=1}^{n-2} \subset Mu(B_R(p)) \subset \bb{C}^{n-2}$. Thus, possibly increasing $C>0$, \eqref{E|cambiobaseHuang5} implies that
         \begin{equation} \label{E|cambiobaseHuang6}
             |J^i| \leq Ci^{1/2}\, .
         \end{equation}
        If $J$ is not diagonal, using that it has all eigenvalues equal to $1$, then at least one entry of $J^i$ has modulus greater than $i$, a contradiction. This shows that $u \circ \gamma^2 (\cdot)=u(\cdot)+u(\gamma^2 \cdot p)$, as we claimed. 
\medskip

        Consider now the map $\phi:H^2 \to \bb{R}^k$ defined by $\phi(\gamma^2):=u(\gamma^2 \cdot p)$. The map is a group homomorphism thanks to the identity $u \circ \gamma^2(\cdot)=u(\cdot)+u(\gamma^2 \cdot p)$. Indeed, given any $\gamma_1,\gamma_2\in H^2$, we have that 
        \begin{align*}
            u \circ \gamma_1 \circ\gamma_2 (\cdot) = & [u(\cdot)+u(\gamma_1 \cdot p)] \circ \gamma_2 
            = u \circ \gamma_2(\cdot) +u(\gamma_1 \cdot p) 
            \\
            = &u (\cdot)+u(\gamma_1 \cdot p) +u(\gamma_2 \cdot p)\, .
        \end{align*}
        Therefore $\phi(\gamma_1 + \gamma_2)=\phi(\gamma_1 )+\phi( \gamma_2) $.
 \medskip
 
        To complete the proof, assume by contradiction that $\phi$ is not injective. Under such assumption there exists a subgroup $H'<H$ isomorphic to $\bb{Z}$ such that $u(H' \cdot p)=\{0\}$. Since we assumed at the beginning of the proof that $X$ does not split off $\bb{R}^{n-2}$ isometrically, by Corollary \ref{L|LemmaTriv} the blow-down of $X$ is pointed isometric to $(\bb{R}^k \times Y,(0,y))$, with $Y$ containing no lines, and it satisfies $\mathrm{Iso}(\bb{R}^k \times Y) \cdot (0,y) \subset \bb{R}^k \times \{y\}$.
        Moreover, the orbit $H' \cdot p$ pointed Hausdorff converges along a blow down sequence to an orbit $H'_\infty \cdot (0,y)$ for some $H'_\infty < \mathrm{Iso}(\bb{R}^k \times Y)$. By Lemma \ref{L|Gruppilimite}, $H_{\infty}$ has a closed subgroup isomorphic to $\R$. Hence, $\pi_{\bb{R}^k}(H'_\infty \cdot (0,y))$ contains a set homeomorphic to $\bb{R}$.
        At the same time, by item \ref{item2HuangHuang}, the map $\pi_{\bb{R}^k}$ is approximated, along a chosen blow down sequence, by maps of the form $T_i u/R_i$ for some $R_i \uparrow + \infty$ and some matrices $T_i \in \bb{R}^{k \times k}$. Since all these maps vanish on $H' \cdot p$,
        we deduce that $\pi_{\bb{R}^k}(H'_\infty \cdot (0,y)) = \{0\}$, a contradiction.  
    \end{proof}
\end{proposition}

Th next lemma will be needed for the proof of Proposition \ref{P|slicing}. The result holds for actions of $\bb{Z}^k$, but we only state it for $k=n-2$, as this is the case that we will use later on.

\begin{lemma} \label{L|newlemma}
    Let $(X,\sd)$ be a proper separable metric space with a closed subgroup $H < \mathrm{Iso}\{X\}$ isomorphic to $\bb{Z}^{n-2}$ and generated by $\{\gamma_1,\cdots,\gamma_{n-2}\}$. Let $R >0$. Then, there are natural numbers $\{k_{j}\}_{j=1}^{n-2}$ and $C>0$ such that 
        \begin{equation}
            R/2n \leq \sd \Big(\mathrm{span}(\gamma_1,\cdots,\gamma_{j-1}) \cdot p,k_{j}\gamma_j \cdot p \Big) \leq R/2n+C \quad \text{for } j=1, \dots,n-2\, ,
        \end{equation}
        and for every natural number $t < k_{j}$ we have
        \begin{equation}
            \sd \Big(\mathrm{span}(\gamma_1,\cdots,\gamma_{j-1}) \cdot p,t \gamma_j \cdot p \Big) \leq R/2n\, .
        \end{equation}
    \begin{proof}
        Let $k_{1}$ be the smallest natural number such that 
        \begin{equation}
            \sd \Big(p,k_{1}\gamma_1 \cdot p\Big) \geq R/2n\, .
        \end{equation} 
        Observe that such $k_{1}$ exists since $\mathrm{span}(\gamma_1) \cdot p$ is non-compact.
        We claim that for every $j \in \{1,\cdots,n-3\}$ we have
        \begin{equation}\label{E|prepreslicing}
            \limsup_{s \to + \infty} \sd(\mathrm{span}(\gamma_1,\cdots,\gamma_j) \cdot p,s\gamma_{j+1} \cdot p)=+\infty\, .
        \end{equation}
        Assume by contradiction that this is not the case, i.e.,
        \begin{equation}
            \limsup_{s \to + \infty} \sd(\mathrm{span}(\gamma_1,\cdots,\gamma_j) \cdot p,s\gamma_{j+1} \cdot p)<c<+\infty\, .
        \end{equation}
        Then for every $s \in \bb{N}$ there exists $h_s \in \mathrm{span}(\gamma_1,\cdots,\gamma_j)$ such that $\sd(p,(h_s+s\gamma_{j+1}) \cdot p) \leq c$. Up to a not relabelled subsequence, $h_s+s\gamma_{j+1} \to h_\infty \in H$, thus contradicting the assumption that $H \cong \bb{Z}^{n-2}$ as a topological group. Hence \eqref{E|prepreslicing} holds.
\medskip

        For every $j \in \{1,\cdots,n-3\}$, we let $k_{j}$ be the smallest natural number such that 
        \begin{equation} \label{E|keyslicingdefk}
            \sd \Big(\mathrm{span}(\gamma_1,\cdots,\gamma_{j}) \cdot p,k_{j}\gamma_{j+1} \cdot p \Big) \geq R/2n\, .
        \end{equation}
        These numbers exist by \eqref{E|prepreslicing}. At the same time, there exists a constant $C>0$ such that
        \begin{equation}
            \sd(p,\gamma_j \cdot p) \leq C \quad \text{for every } j \in \{1, \cdots n-2\}\, .
        \end{equation}
        Therefore
         \begin{align}
            \sd  \Big(\mathrm{span}(\gamma_1,\cdots,\gamma_{j-1}) \cdot p,k_{j}\gamma_j \cdot p\Big) 
             \leq&\,  \sd \Big(\mathrm{span}(\gamma_1,\cdots,\gamma_{j-1}) \cdot p,(k_{j}-1)\gamma_j \cdot p \Big)\\
            &+
            \sd \Big((k_{j}-1)\gamma_j \cdot p,k_{j}\gamma_j \cdot p \Big) \\
             \leq &\,  R/2n+C\, .
        \end{align}
        Combining with \eqref{E|keyslicingdefk}, we obtain
        \begin{equation} \label{E|keyslicingsumup}
            R/2n \leq \sd \Big(\mathrm{span}(\gamma_1,\cdots,\gamma_{j-1}) \cdot p,k_{j}\gamma_j \cdot p \Big) \leq R/2n+C\, ,
        \end{equation}
       thus concluding the proof.
    \end{proof}
\end{lemma}

Proposition \ref{P|slicing} formulates and establishes a precise version of the non-degeneracy condition referred to in item iii) of the plan discussed at the beginning of the section.

\begin{proposition} \label{P|slicing}
    Let $(X,\sd,\haus^n)$ be an $\RCD(0,n)$ space with a closed subgroup $\bb{Z}^{n-2}\cong H < \mathrm{Iso}\{X\}$. Assume that $X$ does not split a factor $\R^{n-2}$. Let $k \in \{n-2,n-1\}$ be given by Corollary \ref{L|LemmaTriv} such that the blow-down of $X$ is unique and isometric to $\bb{R}^k \times Y$, where $Y$ does not split any line.
     Let $u:X \to \bb{R}^{k}$ be the harmonic map given by Proposition \ref{P|HarmonicHuang} and let $\eps_i\downarrow 0$, $R_i\uparrow +\infty$, and $T_i\in\R^{k\times k}$ be such that $T_iu:B_{R_i}(p)\to\R^k$ is an $\epsilon_i$-splitting map for every $i\in\N$. 

    If $k=n-1$, for every $i \in \bb{N}$ we let $\Sigma_i \subset \bb{R}^{n-1}$ be the hyperplane containing $T_i u (H^2 \cdot p)$ and let $\pi_i:\bb{R}^{n-1} \to \Sigma_i \cong \bb{R}^{n-2}$ be the orthogonal projection on $\Sigma_i$.
    Define $v_i:B_{R_i}(p) \to \bb{R}^{n-2}$ by $v_i:=T_iu$ if $k=n-2$, and $v_i:=\pi_i \circ T_iu$ if $k=n-1$.
    Then $v_i$ is a $\epsilon_i$-splitting map for every $i\in\N$ and there exist constants $C_1,C_2>0$ such that for every $i$ sufficiently large
    \begin{equation} \label{E|slicingestimate}
        \Leb^{n-2} \Big(v_i(\{x \in B_{R_i}(p): \sd(x,H^2\cdot p) \leq C_1 \}) \Big) \geq C_2R_i^{n-2}.
    \end{equation}
    \begin{proof}
     By Lemmas \ref{L|uniquetangentconemaximal} and \ref{L|uniquetangentconesubmaximal}, either $X \cong \bb{R}^{n-2} \times Y$ 
        isometrically, or the blow-down of $X$ is unique and isomorphic to $\bb{R}^k \times Y'$ for some 
        $k \in \{n-2,n-1\}$.
        Let $\{\gamma_j\}_{j=1}^{n-2}$ be generators of $H^2$ and, for every $i$, let $\{k_{i,j}\}_{j=1}^{n-2}$ be the natural numbers given by Lemma \ref{L|newlemma} relative to the scale $R_i$. 
        We divide the rest of the proof in three steps.
        \medskip
        
        \textbf{Step 1}: Let $\Sigma \subset \bb{R}^k$ be the $(n-2)$-dimensional vector space containing $u(H \cdot p)$ and consider the map ${T_i}_{|\Sigma}:\Sigma \to \Sigma_i$. We claim that
        \begin{equation}
        \label{E|slicingfinalstep}
         \liminf_{i \to +\infty} \frac{\mathrm{det}({T_i}_{|\Sigma}) \Pi_{j=1}^{n-2}k_{i,j}}{R^{n-2}_i} > 0\, .
     \end{equation}
     
        By construction of the $\{k_{i,j}\}_{j=1}^{n-2}$ (see Lemma \ref{L|newlemma}), for every $j \in \{2,\cdots,n-2\}$ there exists $h_{i,j} \in \mathrm{span}(\gamma_1,\cdots,\gamma_{j-1})$ such that
        \begin{align} \label{E|keyslicing1}
        \nonumber
            \frac{\sd \Big (p,(h_{i,j}+k_{i,j}\gamma_j) \cdot p \Big )}{R_i}      
            &=
            \frac{\sd \Big( \mathrm{span}
    (\gamma_1,\cdots,\gamma_{j-1}) \cdot p
            ,k_{i,j}\gamma_j \cdot p
            \Big)}{R_i}\\
            &\to 1/2n
            \quad \text{as } i\uparrow + \infty\, .
        \end{align}
        
        Let $(X_\infty,\sd_\infty,H^2_\infty,p_\infty)$ be an equivariant blow-down obtained as limit of the sequence $(X,\sd/R_i,H^2)$. Recall that $X_\infty \cong \bb{R}^k \times Y$, and set $p_\infty=(0,y_\infty)$. 
        The groups $H^2_j:=\mathrm{span}(\gamma_1,\cdots,\gamma_j)$ converge, up to not relabelled subsequences, to groups $H^2_{\infty,j} <H^2_\infty$, and the orbits $H^2_j \cdot p$ converge to $H^2_{\infty,j} \cdot p_\infty$. Thanks to \eqref{E|keyslicing1} and Lemma \ref{L|LemmaTriv}, the points $q_{i,j}:=(h_{i,j}+k_{i,j}\gamma_j) \cdot p$ converge to limit points 
        \begin{equation}  \label{E|keyslicing4}     
        q_{\infty,j} \in H^2_{\infty,j} \cdot p_\infty \subset \bb{R}^k \times \{y_\infty\} \subset X_\infty
        \end{equation}such that
        \begin{equation} \label{E|keyslicing2}
            \sd_\infty(H^2_{\infty,j-1} \cdot p_\infty,q_{\infty,j})=\sd_\infty(p_\infty,q_{\infty,j})=1/2n\, .
        \end{equation}
    Thanks to \eqref{E|keyslicing4} and \eqref{E|keyslicing2} we have
    \begin{equation} \label{E|indipendence}
        \sd_{\mathrm{eucl}}(\pi_{\bb{R}^k}(H^2_{\infty,j-1} \cdot p_\infty),\pi_{\bb{R}^k}(q_{\infty,j}))=1/2n\, .
    \end{equation}
    
    The maps $T_iu/R_i$ converge to the projection $\pi_{\bb{R}^k}:X_\infty \to \bb{R}^k$ along the blow-down sequence, so that
    \begin{equation}T_iu/R_i(q_{i,j}) \to \pi_{\bb{R}^k}(q_{\infty,j}) \in \bb{R}^k \quad \text{as } i \uparrow + \infty\, .
    \end{equation}
    The set $\pi_{\bb{R}^k}(H^2_{\infty,j-1} \cdot p_\infty)$ is contained in a $(j-1)$-dimensional vector space since it is the limit of the nets $T_iu/R_i(H^2_{j-1} \cdot p)$. Moreover, it contains the set $\{\pi_{\bb{R}^k}(q_{\infty,s})\}_{s=1}^{j-1}$. At the same time, by Lemma \ref{L|Gruppilimite}, $H^2_{\infty,j-1}$ contains a closed group isomorphic to $\bb{R}^{j-1}$. Thus, by invariance of the domain, $\pi_{\bb{R}^k}(H^2_{\infty,j-1} \cdot p_\infty)$ coincides with the aforementioned $(j-1)$-dimensional vector space.    
    It follows from \eqref{E|indipendence} that $\{\pi_{\bb{R}^k}(q_{\infty,j})\}_{j=1}^{n-2}$ are linearly independent. Hence we obtain that
    \begin{align}
         &\frac{\mathrm{det}({T_i}_{|\Sigma}) \Pi_{j=1}^{n-2}k_{i,j}}{R^{n-2}_i}   \Big(  u(\gamma_1 \cdot p) \wedge \cdots \wedge u(\gamma_{n-2} \cdot p)\Big)
         \\
         & = \frac{1}{R^{n-2}_i} \Big( [T_i u(q_{i,1})] \wedge \cdots \wedge [T_i u(q_{i,n-2})]\Big) \\
         &\to \pi_{\bb{R}^{k}}(q_{\infty,1}) \wedge \cdots \wedge \pi_{\bb{R}^{k}}(q_{\infty,n-2}) \neq 0\, ,
     \end{align}
    thus completing the proof of \eqref{E|slicingfinalstep}.
    \medskip

\textbf{Step 2}: 
Arguing as we did in the proof of \eqref{eq:splitnondege}, we can show that
there exists $i_0 \in \bb{N}$ such that the $\epsilon_{i_0}$-splitting map $T_{i_0}u:B_{R_{i_0}}(p) \to \bb{R}^{k}$ satisfies
         \begin{equation}
             \Leb^{k}(T_{i_0}(u(B_{R_{i_0}}(p))))>0\, .
         \end{equation}
         Thus, setting $C_1:=R_{i_0}$, we have 
         \begin{equation} \label{E|slicingpalle}
         \Leb^k(u(B_{C_1}(p)))>0\, .
     \end{equation}
        Moreover, modulo passing to a sublattice of $H^2$, we may suppose that the sets $u(h \cdot B_{C_1}(p))$ are all disjoint as $h$ varies in $H^2$. 
           We claim that 
    \begin{equation} \label{E|keyslicingClaim1}
        \#\big((H^2\cdot p) \cap B_{R_i-C_1}(p)\big) \geq \Pi_{j=1}^{n-2}k_{i,j}\, .
    \end{equation}
    
    To prove \eqref{E|keyslicingClaim1}, fix $j \in \{2,\cdots,n-2\}$. For every natural number $m_{i,j} \leq  k_{i,j}$ there exists $h_{i,j}^{m_{i,j}} \in H^2_{j-1}$ such that
    \begin{equation}
        \sd(p,(h_{i,j}^{m_{i,j}}+m_{i,j}\gamma_j) \cdot p) \leq R_i/2n +C < \frac{R_i-C_1}n\, ,
    \end{equation}
    where $C>0$ is the constant of Lemma \ref{L|newlemma} and the last inequality holds for $i$ large enough.
    Hence by triangle inequality we have
    \begin{equation}\label{eq:allcontained1}
        \sd\Big(p, \bigg[ \sum_{j=1}^{n-2} \big(h_{i,j}^{m_{i,j}}+m_{i,j}\gamma_j\big)\bigg] \cdot p\Big) < R_i-C_1\, , 
        \end{equation}
        for every
        \begin{equation}\label{eq:allcontained2}
           (m_{i,1}, \cdots, m_{i,n-2}) \in \bb{N}^{n-2} \cap \Pi_{j=1}^{n-2} [1,k_{i,j}]\, .
    \end{equation}
     Given \eqref{eq:allcontained1} and \eqref{eq:allcontained2}, to prove \eqref{E|keyslicingClaim1} it suffices to establish the following claim: if 
     \begin{equation}\label{eq;indexdiff}
     (m_{i,1}, \cdots, m_{i,n-2}) \neq (m'_{i,1}, \cdots, m'_{i,n-2})\, , 
     \end{equation}
     then 
     \begin{equation}
         \sum_{j=1}^{n-2} \big(h_{i,j}^{m_{i,j}}+m_{i,j}\gamma_j \big)\neq  \sum_{j=1}^{n-2}\big(h_{i,j}^{m'_{i,j}}+m'_{i,j}\gamma_j\big) \,.
     \end{equation}
     To prove the claim, assume by contradiction that \eqref{eq;indexdiff} holds but
     \begin{equation} \label{E|asskeyslicing1}
         \sum_{j=1}^{n-2}\big(h_{i,j}^{m_{i,j}}+m_{i,j}\gamma_j\big) = \sum_{j=1}^{n-2}\big(h_{i,j}^{m'_{i,j}}+m'_{i,j}\gamma_j\big)\, .
     \end{equation}
     Then $m_{i,n-2}=m'_{i,n-2}$ since $\{\gamma_1,\dots,\gamma_{n-2}\}$ is a basis of $H^2$ and 
     \begin{equation}
         \sum_{j=1}^{n-2}h_{i,j}^{m_{i,j}}\,, \,\,   \sum_{j=1}^{n-2}h_{i,j}^{m'_{i,j}}\in \mathrm{span}
    (\gamma_1,\cdots,\gamma_{n-3})\, .
     \end{equation}

     Hence $h_{i,n-2}^{m_{i,n-2}}=h_{i,n-2}^{m'_{i,n-2}}$ and thus
     \begin{equation} 
         \sum_{j=1}^{n-3} \big(h_{i,j}^{m_{i,j}}+m_{i,j}\gamma_j\big) = \sum_{j=1}^{n-3}\big(h_{i,j}^{m'_{i,j}}+m'_{i,j}\gamma_j\big) \,.
     \end{equation}
     Iterating this procedure, we deduce $(m_{i,1}, \cdots, m_{i,n-2}) = (m'_{i,1}, \cdots, m'_{i,n-2})$, a contradiction.
     \medskip 

     \textbf{Step 3}: To conclude the proof we treat the case $k=n-2$ and $k=n-1$ separately. 
     
     Assume first that $k=n-2$.
     Using \eqref{E|keyslicingClaim1} and the equivariance of $u$, we obtain that for every $i \in \bb{N}$ there holds
     \begin{align*}
         & \Leb^{n-2}  \Big(\frac{T_i}{R_i}\left(u \big(\{x \in B_{R_i}(p): \sd(x,H^2 \cdot p) \leq C_1 \} \big) \right)\Big) \\
         &= \frac{\mathrm{det}(T_i)}{R_i^{n-2}}
         \Leb^{n-2} \Big( u \big(\{x \in B_{R_i}(p): \sd(x,H^2 \cdot p) \leq C_1 \} \big) \Big) \\
         & \geq 
         \frac{\mathrm{det}(T_i) \Pi_{j=1}^{n-2}k_{i,j}}{R^{n-2}_i}
         \Leb^{n-2} \Big( u(B_{C_1}(p)) \Big)\, .
     \end{align*}
    Combining with \eqref{E|slicingfinalstep} and \eqref{E|slicingpalle}, we obtain the result for $k=n-2$.

    Assume now that $k=n-1$. 
    Let $\mathcal{H}_i$ be the set of affine hyperplanes of $\bb{R}^{n-1}$ parallel to $\Sigma_i$. Observe that for every $A \subset \bb{R}^{n-1}$ we have
    \begin{align}\label{eq:slicehyper1}
         \nonumber \Leb^{n-2}  (\pi_i \circ T_i (A)) \geq&\,  \sup_{H \in \mathcal{H}_i} \Leb^{n-2} (T_i (A) \cap H )\\ =&\, \sup_{H \in \mathcal{H}_i} \Leb^{n-2} \Big(  T_i (A \cap T_i^{-1}(H)) \Big)\, . 
     \end{align}
      Let $v \in \bb{R}^{n-1}$ be a unit vector normal to $\Sigma:=T_i^{-1}(\Sigma_i) \subset \bb{R}^{n-1}$. We have that
     \begin{align}\label{eq:slicehyper2}
        \nonumber \sup_{H \in \mathcal{H}_i} \Leb^{n-2} \Big(  T_i (A \cap T_i^{-1}(H)) \Big) & =
         \sup_{t \in \bb{R}} \Leb^{n-2} \Big(  T_i \big((A+tv) \cap T_i^{-1}(\Sigma_i)\big) \Big) \\
         & 
         = \mathrm{det}({T_i}_{|\Sigma}) \sup_{t \in \bb{R}} \Leb^{n-2}     \big((A+tv) \cap \Sigma\big)\, .
     \end{align}
     Applying \eqref{eq:slicehyper1} and \eqref{eq:slicehyper2} with
     \begin{equation}
         A:=u(\{x \in B_{R_i}(p): \sd(x,H^2 \cdot p) \leq C_1 \})\, ,
     \end{equation}
     we deduce that
     \begin{multline}
         \Leb^{n-2} \Big(\pi_i \circ \frac{T_i}{R_i}\left(u(\{x \in B_{R_i}(p): \sd(x,p+H^2) \leq C_1 \}\right) \Big)  
         \\
         \geq \frac{\mathrm{det}({T_i}_{|\Sigma})}{R^{n-2}_i}
        \sup_{t \in \bb{R}} \Leb^{n-2}\big([u(\{x \in B_{R_i}(p): \sd(x,H^2 \cdot p) \leq C_1 \})+tv] \cap \Sigma\big)\, .
     \end{multline}
     By \eqref{E|keyslicingClaim1}, $[u(\{x \in B_{R_i}(p): \sd(x,H^2 \cdot p) \leq C_1 \})+tv] \cap \Sigma$ contains $\Pi_{j=1}^{n-2}k_{i,j}$ disjoint copies of $[u(B_{C_1}(p))+tv] \cap \Sigma$. Hence
     \begin{align}\label{eq:measlowerbdlast}
        \nonumber & \Leb^{n-2}  \Big(\pi_i \circ \frac{T_i}{R_i}u(\{x \in B_{R_i}(p): \sd(x,p+H^2) \leq C_1 \}) \Big)  
         \\
         & \geq \frac{ \mathrm{det}({T_i}_{|\Sigma})\Pi_{j=1}^{n-2}k_{i,j}}{R^{n-2}_i}
        \sup_{t \in \bb{R}} \Leb^{n-2}([u(B_{C_1}(p))+tv] \cap \Sigma)\, .
     \end{align}
     The statement then follows combining \eqref{eq:measlowerbdlast} with \eqref{E|slicingfinalstep} and \eqref{E|slicingpalle}.   
    \end{proof}
\end{proposition}

\subsection{Proof of Theorem \ref{thm:Bettirigid}}\label{sec:proofThmbettirigid}

       Let $(\overline{X},\overline{\dist},\haus^n,p)$ denote the universal cover of $(X,\dist,\haus^n)$. Since $\pi_1(X)$ has a subgroup isomorphic to $\Z^{n-2}$ there is a closed subgroup $\Z^{n-2}\cong H<\mathrm{Iso}(\overline{X})$.
            By Corollary \ref{L|LemmaTriv}, either $\overline{X} \cong \bb{R}^{n-2} \times Y$ 
        isometrically, or the blow-down of $\overline{X}$ is unique up to isometry and isometric to $\bb{R}^k \times Y'$ for some 
        $k \in \{n-2,n-1\}$. 
        \medskip
        
        Let $\epsilon_i 
        \downarrow 0$, 
        $R_i \uparrow + \infty$,
        and $v_i:B_{R_i}(p) \to \bb{R}^{n-2}$ be the 
        numbers and the $\epsilon_i$-splitting maps given by Proposition \ref{P|slicing}.             
        Fix $\tilde{\epsilon}>0$ sufficiently small.
        By the slicing theorem, see \cite{slicingbrena} after \cite[Theorem 1.23]{CheegerNaberCodim4}, for every $i$ large enough 
        there exists $A_{i,\tilde{\epsilon}} \subset B^{\bb{R}^{n-2}}_{R_i/2}(0)$ with $\Leb^{n-2}(B^{\bb{R}^{n-2}}_{R_i/2}(0) \setminus A_{i,\tilde{\epsilon}})<\tilde{\epsilon}R_i^{n-2}$ such that the following holds:
        For every $x \in v_i^{-1}(A_{i,\tilde\epsilon}) \cap B_{R_i}(p)$, and for every  $r \in (0,R_i/2]$, there exists a matrix $A_{r} \in \bb{R}^{(n-2) \times (n-2)}$ such that $A_{r}{v_i}_{|B_r(x)} \to \bb{R}^{n-2}$ is an $\tilde\epsilon$-splitting map.

        By \eqref{E|slicingestimate} in Proposition \ref{P|slicing}, there exists $C_1>0$ such that for $i$ large enough we have
        \begin{equation}
            A_{i,\tilde\epsilon} \cap v_i\big(\{x \in B_{R_i}(p): \sd(x,H^2\cdot p) \leq C_1 \} \big) \neq \emptyset\, .
        \end{equation}
        Fix now a sequence $\tilde{\epsilon}_j \downarrow 0$. For every $j\in\N$ let $i(j)$ be such that 
         \begin{equation}
            A_{i(j),\tilde\epsilon_j} \cap v_{i(j)}\big(\{x \in B_{R_{i(j)}}(p): \sd(x,H^2\cdot p) \leq C_1 \} \big) \neq \emptyset\, .
        \end{equation}
        Accordingly, for every $j$ we take $x_{j} \in B_{R_{i(j)}}(p)$ such that 
        \begin{itemize}
\item[i)] $v_{i(j)}(x_{j}) \in A_{i(j),\tilde\epsilon_j}$; 
\item[ii)] $\sd(x_j,H^2 \cdot p) \leq C_1$.
\end{itemize}
\medskip

        Fix any $R \gg 1$. Applying \cite[Theorem 3.8]{zbMATH07514027} {(after \cite[Lemma 1.21]{CheegerNaberCodim4})} we obtain a sequence $\delta_j \downarrow 0$ such that the balls $B_{R}(x_j)$ are $\delta_j R$-close in GH sense to balls 
        \begin{equation}
        B_R^{\bb{R}^{n-2} \times Y_j}((0,y_i))\, , 
        \end{equation}
        for some pointed metric spaces $(Y_j,y_j)$. Since $\sd(x_j,H^2 \cdot p) \leq C_1$, modulo applying an isometry of $H^2$, we can assume that $x_j \in \bar{B}_{C_1}(p)$ for every $j$.
        Up to another subsequence, $x_j \to x_\infty \in \bar{B}_{C_1}(p)$, so that $B_R(x_\infty)$ is isometric to $B_R^{\bb{R}^{n-2} \times Y}((0,y))$ for some pointed metric space $(Y,y)$. Hence, the same holds for $B_{R-C_1}(p)$. Since $R \gg 1$ is arbitrary, the statement follows.

\section{A vanishing theorem for the simplicial volume}

The goal of this section is to prove the following vanishing result for the simplicial volume of $\RCD(0,n)$ manifolds:

\begin{theorem}\label{thm:simplicialvolthm}
For every $n\ge 2$, the following holds. If $(X,\dist,\haus^n)$ is an $\RCD(0,n)$ space and $X$ is an oriented triangulable topological $n$-manifold, then $||X||=0$. 
\end{theorem}

Theorem \ref{thm:simpvolintro} clearly follows from Theorem \ref{thm:simplicialvolthm}, since smooth manifolds are triangulable by \cite{Whiteheadtriang}.

\begin{remark}
The assumption that $X$ is triangulable is related to the presence of the same assumption in \cite[Corollary 11]{FrigerioMoraschini} on which our proof depends. As remarked in \cite{FrigerioMoraschini} it seems possible that such assumption is not necessary.  
\end{remark}

The closest available result in the literature is due to J.~Lott and Z.~Shen \cite[Corollary 1.1]{LottShen}. They show that a complete Riemannian manifold $(M^n,g)$ with quadratic curvature decay and such that 
\begin{equation}
    \liminf_{r\to\infty}\frac{\mathrm{vol}(B_r(p))}{r^n}=0
\end{equation}
for some $p\in M$ satisfies $||M||=0$.
\medskip

For the proof of Theorem \ref{thm:simplicialvolthm} we are going to distinguish two cases: the Euclidean volume growth case, i.e., the case when
\begin{equation}
    \limsup_{r\to \infty}\frac{\haus^n(B_r(p))}{r^n}>0\, ,
\end{equation}
and the non-Euclidean volume growth case, when 
\begin{equation}
        \limsup_{r\to \infty}\frac{\haus^n(B_r(p))}{r^n}=0\, ,
\end{equation}
discussed in Section \ref{sec:vanish_Eucl} and Section \ref{sec:vanish_nonEucl} respectively. 
In both cases, the key technical tool is provided by \cite[Corollary 11]{FrigerioMoraschini}, which we state below as Theorem \ref{T|Frig}. We will need some background and terminology which we proceed to recall.

\begin{definition}[Amenable subsets of a topological space]
    Let $X$ be a path-connected topological space, $Y$ be a subset of $X$ and $i:Y\to X$ be the inclusion map. We say that $Y$ is amenable (in $X$) if for every path-connected component $Y'$ of $Y$ the image of $i_*: \pi_1(Y') \to \pi_1(X)$ is an amenable subgroup of $\pi_1(X)$.
\end{definition}

\begin{definition}[Covers and multiplicity]
Let $X$ be a path-connected topological space, and
    let $\{U_i\}_{i\in I}$ be a cover of $X$, i.e.~$U_i\subset X$ for every $i\in I$ and $X=\bigcup_{i\in I} U_i$. We say that the cover is open if $U_i$ is open in $X$ for every $i\in I$, and amenable if each if $U_i$ is amenable in $X$ for every $i\in I$. The multiplicity of $\{U_i\}_{i\in I}$ is defined as
    \begin{align*}
    \mathrm{mult}  \big(\{U_i\}_{i\in I}\big) := \sup \{n\in \N \,:\, \exists \, i_1,\dots,i_n \in I\, ,\, i_h \neq i_k \\
    \text{ for }h\neq k,\, U_{i_1} \cap \dots \cap U_{i_n} \neq \emptyset\} \in \N\cup \{\infty\}\, .
    \end{align*}
\end{definition}

\begin{definition}\label{def:amenableatinfinity}
    A sequence of subsets $\{U_i\}_{i\in \N}$ of a path-connected non-compact topological space $X$ is said to be amenable at infinity if there exists a sequence $\{W_i\}_{i\in \N}$ of open subsets of $X$ such that the following conditions hold:
    \begin{itemize}
        \item[(i)] each $W_i$ is large, i.e.~its complement in $X$ is relatively compact;
        \item[(ii)] the family $\{W_i\}_{i\in \N}$ is locally finite;
        \item[(iii)] $U_i\subset W_i$ for every $i\in \N$;
        \item[(iv)] there exists $\bar i \in \N$ such that $U_i$ is an amenable subset of $W_i$ for every $i\geq \bar i$.
    \end{itemize}
\end{definition}

The following statement corresponds to \cite[Corollary 11]{FrigerioMoraschini}.

\begin{theorem} \label{T|Frig}
    Let $M^n$ be an oriented, open, and triangulable manifold of dimension $n$. Suppose that $\{U_i\}_{i\in \N}$ is an amenable open cover of $M$ such that each $U_i$ is relatively compact in $M$, the sequence $\{U_i\}_{i\in \N}$ is amenable at infinity and $\mathrm{mult}\big(\{U_i\}_{i\in \N}\big) \leq n$. Then $\norm{M}=0$.
\end{theorem}

\subsection{Proof of Theorem \ref{thm:simplicialvolthm}: Euclidean volume growth case}\label{sec:vanish_Eucl}

The goal of this section is to prove Theorem \ref{thm:simplicialvolthm} under the assumption that the space has maximal volume growth. Thanks to Theorem \ref{T|Frig}, it is sufficient to prove Theorem \ref{T|2overlap} below.

\begin{theorem} \label{T|2overlap}
    Let $(X,\sd,\haus^n,p)$ be an $\RCD(0,n)$ space. If $X$ has maximal volume growth, 
    then it admits a cover $\{ U_i\}_{i\in \N}$ which is amenable, open, amenable at infinity and has $\mathrm{mult}\big(\{ U_i\}_{i\in \N}\big) \leq 2$.
    \end{theorem}

Theorem \ref{T|2overlap}, in turn, is an immediate consequence of Theorem \ref{T|ZhouAnelli} below, which is the focus of this section.

    \begin{theorem} \label{T|ZhouAnelli}
    Let $n \geq 3$, and let $(X,\sd,\haus^n,p)$ be an $\RCD(0,n)$ space with maximal volume growth
    \begin{equation}
        \lim_{r \to + \infty} \frac{\aH^n(B_r(p))}{r^n} \geq v>0.
    \end{equation} 
    Then, for every $r$ sufficiently large, $B_{4r}(p) \setminus B_{3r}(p)$ is contained in a path-connected component $C_r$ of $B_{5r}(p) \setminus B_{2r}(p)$. Moreover, there exists a constant $C(n,v)>0$ such that
    the image through the push-forward of the inclusion map
    \begin{equation}\label{eq:inclusion-cc}
        i_*:\pi_1(C_r,c) \to \pi_1(B_{6r}(p) \setminus B_r(p),c)
    \end{equation}
    has cardinality at most $C(n,v)$ for every sufficiently large $r>0$ and every $c \in C_r$.  
    \end{theorem}

    The proof of the first part of the statement of Theorem \ref{T|ZhouAnelli} is straightforward.
    To prove \eqref{eq:inclusion-cc} we argue by contradiction. Hence, we find a sequence of annuli $A_i:=B_{6r_i}(p) \setminus B_{r_i}(p)$ such that the images of the maps as in \eqref{eq:inclusion-cc} have large cardinality.
    Then we consider their universal coverings $\tilde{A}_i$, and suitable subsets $\tilde{A}'_i \subset \tilde{A}_i$ which are coverings of the corresponding path-connected components $C_{r_i}$.
    Thanks to our assumption by contradiction, the coverings $\tilde{A}_i'$ are non-compact. At the same time, scaling down distances, the sets $C_{r_i}$ converge  to a metric space $Y$ which is isometric to an annulus in a metric cone over an $\RCD(n-2,n-1)$ space with volume bounded below by a constant depending on $v$. Such annulus is homotopically equivalent to the cross-section. Thus, in particular, if $n\ge 3$ then $\#\pi_1(Y)<C(n,v)$ as an application of \cite[Theorem 1.2]{MondinoWei}.
    Similarly, the coverings $\tilde{A}'_i$ equipped with their deck transformations converge in equivariant Gromov-Hausdorff sense to a pair $(\tilde{Y},H)$ with $H< \mathrm{Iso}(\tilde{Y})$ such that $\tilde{Y}/H \cong Y$ isometrically. 
    
    To finish the proof, we are going to combine the statement that $\# \pi_1(Y)<C(n,v)$ with some ideas originally developed in \cite{Anderson,Andersonshort,zbMATH00444969} to deduce that $H$ has also cardinality bounded by $C(n,v)$. This then leads to a contradiction.
    \medskip

Making completely rigorous the first part of the plan sketched above involves a number of steps. The proofs of the next two lemmas  follow the lines of \cite[Lemma 3.3]{zbMATH00444969}.

\begin{lemma} \label{L|boundgenerateor}
    Let $(X,\sd,\haus^n,p)$ be an $\RCD(K,n)$ space with 
    \begin{equation}\label{eq:AVRfg}
        \aH^n(B_1(p)) \geq v>0.
    \end{equation} 
    Then there exists $k(v,n,K)\in \bb{N}$ such that the image through the push-forward of the inclusion map
    \begin{equation}
    i_*:\pi_1(B_{1}(p),p) \to\pi_1(B_{4}(p),p) 
    \end{equation}
    contains at most $k$ elements that can be represented by loops of length less than or equal to $2$ centered in $p$.
    \begin{proof}
    We denote the image of $i_{*}:\pi_1(B_{1}(p),p) \to \pi_1(B_{4}(p),p)$ by $G$.
        Let $\pi:Y \to B_{4}(p)$ be the universal covering of $B_{4}(p)$. Consider on $Y$ the distance $\tilde{\sd}$ obtained by lifting the length structure of $B_{4}(p)$.
        Let $\tilde{\haus}^n$ denote
        the corresponding Hausdorff measure. 
        Let $\tilde{p} \in \pi^{-1}(p)$, and let $F \subset Y$ be the set containing $\tilde{p}$ constructed in Lemma \ref{L|fundamentaldomain}. 
        Call $G_{2} \subset G$  the set of elements that can be represented by loops of length at most $2$ centered in $p$. Let
        \begin{equation}
            B:=\bigcup_{g \in G_{2}} g(F \cap B^{\tilde{\sd}}_{1}(\tilde{p}))\subset Y\, .
        \end{equation}
        We claim that $B \subset B^{\tilde{\sd}}_{3}(\tilde{p})$. Indeed, if $g \in G_{2}$, we can join $\tilde{p}$ to $g(\tilde{p})$ with a curve of length at most $2$ and we can join $g(\tilde{p})$ to any element of $g(F \cap B^{\tilde{\sd}}_{1}(\tilde{p}))$ by a curve of length at most $1$. Hence, $B \subset B^{\tilde{\sd}}_{3}(\tilde{p})$ by triangle inequality. 
        
        Combining such inclusion with Lemma \ref{L|fundamentaldomain} we get
        \begin{align}
            (\# G_{2} ) \cdot \haus^{n}(B_1(p)) 
            =&\,  \sum_{g \in G_{2}} \tilde{\haus}^n\big( g(F \cap B^{\tilde{\sd}}_{1}(\tilde{p}))
            \big)\\
            \leq &\, \tilde{\haus}^n(B)\leq 
           \tilde{\haus}^n(B^{\tilde{\sd}}_{3}(\tilde{p}))\, .
        \end{align}
       Thus, by Lemma \ref{L|lift}, it holds
        \begin{equation}
            (\# G_{2} ) \cdot \haus^{n}(B_1(p))  \leq  v_{K,n}(3),
        \end{equation}
        where $v_{K,N}(r)$ denotes the volume of the ball of radius $r$ in the model space for the $\CD(K,N)$ condition.
        Hence, combining with \eqref{eq:AVRfg}, we obtain that
        \begin{equation}
            \# G_{2} \leq  v_{K,n}(3)/v\, ,
        \end{equation}
        thus concluding the proof.
    \end{proof}
\end{lemma}

\begin{lemma} \label{L|Zhou}
    Let $(X,\sd,\haus^n,p)$ be an $\RCD(0,n)$ space with 
    \begin{equation}\label{eq:FGAVR2}
        \lim_{r \to + \infty} \frac{\aH^n(B_r(p))}{r^n} \geq v>0 \, .
    \end{equation}
    Then there exist $\epsilon(n,v)>0$, $C(n,v) \in \bb{N}$ such that the image through the push-forward of the inclusion map
    \begin{equation}
    i_*:\pi_1(B_{\epsilon}(p),p) \to \pi_1(B_{1}(p),p) 
    \end{equation}
    has cardinality at most $C$.
    \begin{proof}
        Let $m_0 := \lceil 1+\omega_n 3^n/v \rceil$, and choose $\epsilon(n,v):=(10m_0)^{-1}$. We denote the image of $i_*:\pi_1(B_{\epsilon}(p),p) \to\pi_1(B_{1}(p),p)$ by $G$. It is a standard observation that $G$ can be generated by loops of length at most $2\epsilon$. Thus, by Lemma \ref{L|boundgenerateor}, $G$ is generated by at most $k(n,v)$ elements $\{\gamma_1,\cdots,\gamma_k\}$ of length at most $2\epsilon$ centered in $p$. For every $m \in \bb{N}$, let $U(m)$ be the ball of radius $m$ in $G$ centered in the identity w.r.t.\ the word metric and the set of generators considered above. 
        Let $\pi:Y \to B_{1}(p)$ be the universal covering of $B_{1}(p)$ and let $\tilde{p}$ be a preimage of $p$. 
        Endow $Y$ with the distance $\tilde{\sd}$ induced by the lift of the length structure of $B_{1}(p)$ (cf. with Remark \ref{rmk:lengthdistanceoncovering}) and the corresponding Hausdorff measure $\tilde{\haus}^n$. Let $F \subset Y$ be the set containing $\tilde{p}$ as in Lemma \ref{L|fundamentaldomain}. 
        \medskip
        
        We define
        \begin{equation}
            B:=\bigcup_{g \in U(m_0)} g(F \cap B^{\tilde{\sd}}_{m_0  \epsilon}(\tilde p))\subset Y\, .
        \end{equation}
        We claim that if $x \in B$ then $x$ is joined to $\tilde{p}$ by a curve of length at most $3m_0 \epsilon$. Indeed, if 
        \begin{equation}
        x \in g(F \cap B^{\tilde{\sd}}_{m_0 \epsilon }(\tilde p))\quad  \text{for some}\quad  g \in U(m_0)\, , 
        \end{equation}
        then there is a curve from $\tilde{p}$ to $g(\tilde{p})$ of length at most $2m_0 \epsilon$, and another curve of length at most $m_0 \epsilon$ joining $g(\tilde{p})$ and $x$.
        \medskip

        Since $B \subset B^{\tilde{\sd}}_{3m_0 \epsilon}(\tilde{p})$, Lemma \ref{L|lift} implies that
        \begin{equation} \label{E|lemmazhou1}
            \tilde{\haus}^n(B) \leq \omega_n(3m_0 \epsilon)^n\, .
        \end{equation}
        We claim that $\# G \leq  \# U(m_0)$. Note that if the claim holds then the statement immediately follows.
 
        Assume by contradiction that
        \begin{equation} 
            \# G > \# U(m_0)\, .
        \end{equation}
        Since $\# U(m+1)>\# U(m)$ unless $U(m)=G$ for every $m \in \bb{N}$, it follows that $m_0 \leq \#U(m_0)$. Hence, we obtain that
        \begin{equation}
            m_0 \cdot \haus^{n}(B_{m_0 \epsilon}(p)) \leq \#U(m_0)\cdot \tilde{\haus}^n(F \cap B^{\tilde{\sd}}_{m_0 \epsilon }(\tilde p))=\tilde{\haus}^n(B)\, .
        \end{equation}
        Combining with \eqref{E|lemmazhou1}, we deduce that
        \begin{equation}
            m_0 \haus^{n}(B_{m_0 \epsilon}(p)) \leq \omega_n(3 m_0 \epsilon)^n\, .
        \end{equation}
        Exploiting \eqref{eq:FGAVR2}, we obtain that
        \begin{equation}
            m_0 \leq \omega_n (3m_0 \epsilon )^n/\haus^{n}(B_{m_0 \epsilon}(p)) \leq \omega_n 3^n/v\, ,
        \end{equation}
        thus contradicting our choice of $m_0$.
    \end{proof}
\end{lemma}

The following lemma is well-known (cf. with \cite{ZamoraAnderson}), and it can be obtained arguing as for the proof of the analogous result in the smooth setting, see \cite{Andersonshort}. We discuss the proof for the sake of completeness.

\begin{lemma} \label{L|variantionzhou}
    Let $K>0$ and let $(X,\sd,\haus^n)$ be an $\RCD(K,n)$ space with 
    \begin{equation}
        \haus^n(X) \geq v>0\, .
    \end{equation}
    Then, $\pi_1(X)$ has cardinality at most $C(n,v,K)$.
    \begin{proof}
        Let $(\tilde{X},\tilde{\sd},\tilde{\aH}^n)$ be the universal covering of $X$, which is simply connected and is an $\RCD(K,n)$ space by \cite{RCDsemilocally} and \cite{MondinoWei}. By Bonnet-Myers \cite[Corollary 2.6]{S2}, it holds $\mathrm{diam}(\tilde{X}) \leq C(n,K)$, so that Bishop-Gromov's inequality \cite[Theorem 2.3]{S2} implies that ${\aH}^n(\tilde{X}) \leq C(n,K)$. Consider now a fundamental domain $F \subset \tilde{X}$ as given by Lemma \ref{L|fundamentaldomain}. For every $g \in \pi_1(X)$, it holds $g \cdot F \cap F=\emptyset$ and $\tilde{\aH}^n(g \cdot F)=\aH^n(X)$. Hence,
        \begin{equation}
            \# \pi_1(X) \, v \leq \# \pi_1(X) \,\aH^n(X) \leq \tilde{\aH}^n(\tilde{X}) \leq C(n,K)\, ,
        \end{equation}
        concluding the proof.
    \end{proof}
    \end{lemma}

The next two lemmas contain technical results concerning sequences of metric spaces converging in the pGH sense to metric cones.

\begin{lemma}\label{lem:distanceannuli}
    Let $Z=C(Y)$ be the metric cone defined over some metric space $(Y,d_Y)$ with $\diam(Y)\leq \pi$ and let $\sd$ be the associated cone distance. Calling $\bar p$ its tip, for every $\varepsilon \in (0,1/2)$, we define the annuli 
    \begin{equation}
        A:= B_{6}(\bar p) \setminus \bar B_{1}(\bar p)\, , \quad  A_\varepsilon :=  \bar B_{6 - \varepsilon}(\bar p) \setminus  B_{1 + \varepsilon}(\bar p)\, , \quad A^\varepsilon:= \bar B_{6 + \varepsilon}(\bar p) \setminus B_{1 - \varepsilon}(\bar p)\, .
    \end{equation}
    Denote by $\hat\sd$, $\hat\sd_\varepsilon$ and $\hat\sd^\varepsilon$ the length distances induced by $\sd$ on the annuli $A$, $A_\varepsilon$ and $A^\varepsilon$, respectively. Then, for every $\varepsilon\in (0,1/2)$ we have that
    \begin{equation}\label{eq:lengthcones}
        \frac{1}{1+\varepsilon}\hat\sd_\varepsilon (p,q) \leq \hat\sd (p,q)\leq \frac{1}{1-\varepsilon}\hat\sd^\varepsilon (p,q)\, ,
    \end{equation}
    for every $p,q\in \tilde A:=\bar{B}_{5}(\bar p) \setminus B_{2}(\bar p)$.
\end{lemma}

\begin{proof}
    Recall that $Z=C(Y)$ is defined as
\begin{equation}
    Z = (Y \times [0, \infty)) / \sim \, ,
\end{equation}
where the equivalence relation $\sim$ identifies all points of the form $(y,0)$ for all $y \in Y$, and the cone distance $\sd$ is defined by
\begin{equation}
    \sd\big((x, r), (y, s)\big) = \sqrt{r^2 + s^2 - 2rs\, \cos( d_Y(x, y))}, \quad \text{where }x, y \in Y, \, r, s \ge 0\, .
\end{equation}
We are only going to prove the first inequality in \eqref{eq:lengthcones}, as the proof of the second is analogous. Given any two points $p,q\in \tilde A$, for every sufficiently small $\delta>0$ there exists a curve $\gamma:[0,1]\to A$ such that $\gamma(0)=p$, $\gamma(1)=q$ and $\length^{\sd}(\gamma)< \hat \sd(p,q) + \delta$. If $\gamma(t)=(y(t),r(t))$, with $y:[0,1]\to Y$ and $r:[0,1]\to (1,6)$, we define $\tilde\gamma:[0,1]\to A_\varepsilon$ as 
\begin{equation}
    \tilde \gamma(t):=(y(t),\tilde r(t)):= \big(y(t),  (1+\varepsilon) \vee r(t)\wedge (6-\varepsilon)\big)\, .
\end{equation}
Since $\varepsilon\in(0,1/2)$, $\tilde\gamma(0)=p$ and $\tilde\gamma(1)=q$.
Then, we observe that, for every $s,t \in [0,1]$, there holds
\begin{align}
        \sd (\tilde\gamma(s),\tilde\gamma(t))^2 
        =& (\tilde r(s)-\tilde r(t))^2 +2\tilde r(s)\tilde r(t)\,\big(1- \cos( d_Y(y(s), y(t)))\big)\\
        \leq& (r(s)-r(t))^2 + 2(1+\varepsilon)^2  r(s) r(t)\,\big(1- \cos( d_Y(y(s), y(t))\big)\\
        \leq &(1+\varepsilon)^2 \sd (\gamma(s),\gamma(t))^2\, .
\end{align}
Thus, we conclude that 
\begin{equation}\label{eq:almostlenght1}
\hat\sd_\varepsilon (p,q) \leq  \length^{\sd}(\tilde \gamma) \leq (1+\varepsilon) \length^{\sd}(\gamma)<(1+\varepsilon) (\hat\sd (p,q)+\delta)\, . 
\end{equation}
Since \eqref{eq:almostlenght1} holds for every $\delta>0$, the first inequality in \eqref{eq:lengthcones} follows.
\end{proof}

\begin{lemma} \label{L|convanelli}
    Let $(X,\sd,\haus^{n},\bar p)$ be an $\RCD(0,n)$ space with \begin{equation}
        \lim_{r \to + \infty} \frac{\haus^n(B_r(\bar p))}{r^n} \geq v>0\, .
    \end{equation}
    Let $r_i \uparrow + \infty$ be a sequence such that $(X,\sd_i:=\sd/r_i,\bar p)$ converges in the pGH sense to a metric cone $(C(Y),\sd_\infty, p_\infty)$. For every $i\in \N$, set $A_i:= B_{6r_i}(\bar p) \setminus \bar B_{r_i}(\bar p)$ and $\tilde A_i:=\bar{B}_{5r_i}(\bar p) \setminus B_{2r_i}(\bar p)$. Moreover, we define $A:=B^{\sd_\infty}_{6}(p_\infty) \setminus \bar B^{\sd_\infty}_{1}(p_\infty)$ and $\tilde A:=\bar{B}^{\sd_\infty}_{5}(p_\infty) \setminus B^{\sd_\infty}_{2}(p_\infty)$. For every $i\in \N$, we call $\hat\sd_i$ the length distance induced by $\sd_i$ on $A_i$ and, similarly, we call $\hat\sd_\infty$ the length distance induced by $\sd_\infty$ on $A$. Then, $(\tilde A_i,\hat\sd_i)$ converges in the GH sense to $(\tilde A,\hat\sd_\infty)$. 
\end{lemma}

\begin{proof}
Since $(X,\sd_i,\bar p)$ converges in the pGH sense to $(C(Y),\sd_\infty, p_\infty)$, there exists a sequence $(\varepsilon_i)_{i\in \N}\to 0$ such that for every $i\in \N$ there are $\varepsilon_i$-GH maps $f_i : B^{\sd_\infty}_{1/\varepsilon_i}(p_\infty) \to B^{\sd_i}_{1/\varepsilon_i}(\bar p)$ and $g_i : B^{\sd_i}_{1/\varepsilon_i}(\bar p)\to B^{\sd_\infty}_{1/\varepsilon_i}(p_\infty)$ with $f_i(\bar p)=p_\infty$ and $g_i(p_\infty)=\bar p$. In particular, $f_i\big(B^{\sd_\infty}_{1/\varepsilon_i}(p_\infty)\big)$ is an $\varepsilon_i$-net in $B^{\sd_i}_{1/\varepsilon_i}(\bar p)$, $g_i\big(B^{\sd_i}_{1/\varepsilon_i}(\bar p)\big)$ is an $\varepsilon_i$-net in $B^{\sd_\infty}_{1/\varepsilon_i}(p_\infty)$ and
\begin{equation}
\begin{split}
     |\sd_i (f_i(x_\infty),f_i(y_\infty))- \sd_\infty (x_\infty,y_\infty)| <\varepsilon_i\, , \quad &\forall\,  x_\infty,y_\infty\in B^{\sd_\infty}_{1/\varepsilon_i}(p_\infty)\, ,\\
      |\sd_\infty (g_i(x_i),g_i(y_i))- \sd_i (x_i,y_i)| <\varepsilon_i\, , \quad &\forall \, x_i,y_i\in B^{\sd_i}_{1/\varepsilon_i}(\bar p)\, .
\end{split}
\end{equation}
Moreover, we can construct such $f_i$ and $g_i$ so that
\begin{equation}\label{eq:almostinverse}
    \sd_\infty (g_i (f_i (x_\infty) , x_\infty)) < 2 \varepsilon_i, \quad \forall x_\infty\in B^{\sd_\infty}_{1/\varepsilon_i}(p_\infty)\, .
\end{equation}
In the following we are implicitly going to consider only indexes $i$ sufficiently large so that $\varepsilon_i<1/7$ and thus $f_i$ and $g_i$ are defined on $B^{\sd_\infty}_{7}(p_\infty)$ and $B^{\sd_i}_{7}(\bar p)$, respectively.
\medskip

Fix any two points $p,q \in \tilde{A}$ and $\varepsilon > 0$ sufficiently small. Consider the annulus
\begin{equation}
A_\varepsilon := \bar B^{\sd_\infty}_{6 - \varepsilon}(p_\infty) \setminus  B^{\sd_\infty}_{1 + \varepsilon}(p_\infty) \subset C(Y)\, .
\end{equation}
According to Lemma \ref{lem:distanceannuli}, take a constant-speed curve $\gamma_\varepsilon : [0,1] \to A_\varepsilon$ such that $\gamma_\varepsilon(0)=p$, $\gamma_\varepsilon(1)=q$ and $\length^{\sd_\infty}(\gamma_\varepsilon) \le (1 + \varepsilon)\hat{\sd}_\infty(p,q)$. Partition the interval $[0,1]$ with times $t_0 = 0< t_1< \cdots< t_N = 1$ such that $|t_{j+1}-t_j|< \varepsilon (8 \hat{\sd}_\infty(p,q))^{-1}$ for every $j=0, \dots, N-1$. Setting $x_j = \gamma_\varepsilon(t_j)$ for $j=0, \dots, N$, we have that 
\begin{equation}
\sd_\infty(x_j, x_{j+1}) < \varepsilon/4\, , \quad \text{for }j=0, \dots, N-1\, .
\end{equation}
Moreover,
\begin{equation}
\sum_{j=0}^{N-1} \sd_\infty(x_j, x_{j+1}) \le \length^{\sd_\infty}(\gamma_\varepsilon) \le (1 + \varepsilon)\hat{\sd}_\infty(p,q)\, .
\end{equation}
Take any $i\in \N$ sufficiently large such that $\varepsilon_i< \varepsilon/8$. Setting $x^i_j = f_i(x_j)$ and recalling that $f_i$ is a $\varepsilon_i$-GH map, on the one hand we have that
\begin{equation}\label{eq:approx1}
    \sd_i(x_j^i, x_{j+1}^i) <  \varepsilon_i + \sd_\infty(x_j, x_{j+1}) < \varepsilon/2\, , \quad \text{for }j=0, \dots, N-1\, .
\end{equation}
On the other hand, recalling that $f_i(\bar p)=p_\infty$, we get that 
\begin{equation}\label{eq:approx2}
     x_j^i \in \bar B^{\sd_i}_{6 - 3\varepsilon/4}(\bar p) \setminus B^{\sd_i}_{1 + 3\varepsilon/4}(\bar p) \subset X_i\, , \quad \text{for }j=0, \dots, N\, .
\end{equation}
Combining \eqref{eq:approx1} and \eqref{eq:approx2}, we deduce that $\hat \sd_i (x_j^i, x_{j+1}^i)= \sd_i (x_j^i, x_{j+1}^i)$ for $j=0, \dots, N-1$. Thus, introducing the simplified notations $p_i=f_i(p)=f_i(x_0)= x_0^i$ and $q_i=f_i(q)=f_i(x_N)=x_N^i$, we obtain that 
\begin{align}\label{eq:upperbound}
\hat{\sd}_i(p_i, q_i)=& \hat{\sd}_i(x^i_0, x^i_N)\le \sum_{j=0}^{N-1} \hat{\sd}_i(x^i_j, x^i_{j+1})\\ 
=& \sum_{j=0}^{N-1} \sd_i(x^i_j, x^i_{j+1}) 
\le \sum_{j=0}^{N-1} \sd_\infty(x_j, x_{j+1}) + \varepsilon_i N \\
\le& (1+\varepsilon)\hat{\sd}_\infty(p,q) + \varepsilon_i N\, .
\end{align}

Consider a constant-speed curve $\eta_i : [0,1] \to A_i$ such that $\eta_i(0) = p_i$, $\eta_i(1) = q_i$
and $\length^{\sd_i}(\eta_i) = \hat{\sd}_i(p_i, q_i)$. Partition the interval $[0,1]$ with times $s_0 = 0< s_1< \cdots< s_M = 1$ such that $|s_{j+1}-s_j|< \varepsilon (16 \hat{\sd}_\infty(p,q) + \varepsilon N)^{-1}$. Then, calling $y^i_j = \eta_i(s_j)$ for $j=0,\dots, M$, the inequality \eqref{eq:upperbound} guarantees that
\begin{equation}
\sd_i(y^i_j, y^i_{j+1}) < \varepsilon/8\, .
\end{equation}
Define $z^i_j := g_i(y^i_j)\in C(Y)$ for $j=0,\dots, M$, and observe that by construction (see \eqref{eq:almostinverse}) we have that
\begin{equation}
\sd_\infty(z^i_0, p) < 2\varepsilon_i \quad \text{and}\quad \sd_\infty(z^i_M, q) < 2\varepsilon_i\, .
\end{equation}
 On the other hand, since $g_i$ is a $\varepsilon_i$-GH map, 
\begin{equation}\label{eq:approx11}
    \sd_\infty(z_j^i, z_{j+1}^i) <  \varepsilon_i + \sd_\infty(y^i_j, y^i_{j+1}) < \varepsilon/2\, , \quad \text{for }j=0, \dots, M-1\, .
\end{equation}
Moreover, recalling that $g_i(p_\infty)=\bar p$, we get that 
\begin{equation}\label{eq:approx22}
     z_j^i \in \bar B^{\sd_\infty}_{6 + \varepsilon/4}(p_\infty) \setminus B^{\sd_\infty}_{1 -\varepsilon/4}(p_\infty) \subset C(Y)\, , \quad \text{for }j=0, \dots, M\, .
\end{equation}
Therefore
\begin{multline}\label{eq:estimategi}
\hat{\sd}_i(p_i, q_i) \ge \sum_{j=0}^{M-1} \sd_i(y^i_j, y^i_{j+1}) \geq \sum_{j=0}^{M-1} \sd_\infty(z_j^i, z_{j+1}^i) - \varepsilon_i M \\ \geq \sd_\infty(z^i_0,p) + \sum_{j=0}^{M-1} \sd_\infty(z_j^i, z_{j+1}^i) +\sd_\infty(z^i_M, q)  - \varepsilon_i (M+4)\, .
\end{multline}
Let $\hat{\sd}^\varepsilon$ denote the length distance on
$\bar B^{\sd_\infty}_{6 + \varepsilon}(p_\infty) \setminus B^{\sd_\infty}_{1 - \varepsilon}(p_\infty)$. 
Combining \eqref{eq:approx11} and \eqref{eq:approx22}, we know that $\hat \sd^\varepsilon (z_j^i, z_{j+1}^i)= \sd_\infty (z_j^i, z_{j+1}^i)$ for $j=0, \dots, M-1$, and similarly $\hat\sd^\varepsilon(z^i_0,p)= \sd_\infty(z^i_0,p)$ and $\hat\sd^\varepsilon(z^i_M, q)=\sd_\infty(z^i_M, q)$. Then, the right-hand side of \eqref{eq:estimategi} is equal to 
\begin{align*}
& \hat\sd^\varepsilon(z^i_0,p) + \sum_{j=0}^{M-1} \hat\sd^\varepsilon(z_j^i, z_{j+1}^i) +\hat\sd^\varepsilon(z^i_M, q)  - \varepsilon_i (M+4) \\
& \geq 
\hat{\sd}^\varepsilon(p,q) - \varepsilon_i (M+4) \ge (1-\varepsilon)\hat{\sd}_\infty(p,q) - \varepsilon_i(M+4)\, ,
\end{align*}
where we applied Lemma \ref{lem:distanceannuli} for the last inequality.
In particular,
\begin{equation}\label{eq:epsGH2}
    \hat{\sd}_i(p_i, q_i) \geq (1-\varepsilon)\hat{\sd}_\infty(p,q) - \varepsilon_i(M+4)\, .
\end{equation}
Since $p$, $q$ and $\varepsilon$ were chosen arbitrarily and $\hat\sd_\infty$ is clearly bounded on $\tilde A \times \tilde A$, the combination of \eqref{eq:upperbound} and \eqref{eq:epsGH2} implies that given any $\epsilon>0$ there exists $N(\epsilon)$ such that for every $i\geq N(\epsilon)$ we have
\begin{equation}
    \big|\hat{\sd}_i(f_i(p), f_i(q)) - \hat{\sd}_\infty(p,q) \big| \leq \epsilon/3\, , \quad \forall p,q\in \tilde A\, .
\end{equation}
Consider the closest point projection $\pi:C(Y) \to \bar B^{\sd_\infty}_{5 - \epsilon/3}(p_\infty) \setminus B^{\sd_\infty}_{2 + \epsilon/3}(p_\infty)$ and observe that
\begin{equation}
    \hat{\sd}_\infty(\pi(p),p) \leq \frac \epsilon 3\, , \quad \forall p\in \tilde A\, .
\end{equation}
Thus for every $p,q \in \tilde A$ and every $i\geq N(\epsilon)$ we have that
\begin{multline}
    \big|\hat{\sd}_i(f_i(\pi(p)), f_i(\pi(q))) - \hat{\sd}_\infty(p,q) \big| \\
    \leq \big|\hat{\sd}_i(f_i(\pi(p)), f_i(\pi(q))) - \hat{\sd}_\infty(\pi(p),\pi(q)) \big| + \big|\hat{\sd}_\infty(\pi(p),\pi(q))  - \hat{\sd}_\infty(p,q) \big| \\
    <\frac \epsilon 3 +
    \big|\hat{\sd}_\infty(\pi(p),\pi(q))- \hat{\sd}_\infty(\pi(p),q) \big| + \big|\hat{\sd}_\infty(\pi(p),q)  - \hat{\sd}_\infty(p,q) \big| < \epsilon\, .
\end{multline}
On the other hand, since $f_i$ is an $\varepsilon_i$-GH map,  
\begin{equation}
    \sd_i \big(x_i , f_i \big(\bar B^{\sd_\infty}_{5 + 2\varepsilon_i }(p_\infty) \setminus B^{\sd_\infty}_{2 -2 \varepsilon_i}(p_\infty)\big)\big) \leq\varepsilon_i\, , \quad \text{for every }x_i \in \tilde A_i\, .
\end{equation}
Therefore, noting that 
\begin{equation}
f_i\big(\bar B^{\sd_\infty}_{5 - \epsilon/3}(p_\infty) \setminus B^{\sd_\infty}_{2 + \epsilon/3}(p_\infty)\big)\subset \tilde A_i 
\end{equation}
for every $i$ sufficiently large, we have that 
\begin{equation}
f_i(\pi(\tilde A))=f_i\big(\bar B^{\sd_\infty}_{5 - \epsilon/3}(p_\infty) \setminus B^{\sd_\infty}_{2 + \epsilon/3}(p_\infty)\big) 
\end{equation}
is a $(4 \varepsilon_i + \epsilon/3)$-net in $\tilde A_i$, with respect to $\sd_i$. Thus, up to taking a bigger $N(\epsilon)$, for every $i\geq N(\epsilon)$ we have that $f_i(\pi(\tilde A))$ is a $\epsilon$-net in $\tilde A_i$, with respect to $\sd_i$. Since the distances $\sd_i$ and $\hat\sd_i$ coincide when restricted to balls of radius $\frac 12$ centered in $\tilde A_i$, for $\epsilon$ sufficiently small we conclude that $f_i(\pi(\tilde A))$ is an $\epsilon$-net in $\tilde A_i$ with respect to $\hat\sd_i$.
This argument proves that for every $i\geq N(\epsilon)$ the map $f_i \circ \pi:(\tilde A,\hat \sd) \to (\tilde A_i,\hat \sd_i)$ is an $\epsilon$-GH map. The conclusion follows from the arbitrariness of $\epsilon$.
\end{proof}

    \begin{proof}[Proof of Theorem \ref{T|ZhouAnelli}]
    
    We first show that, for every $r$ sufficiently large, $B_{4r}(p) \setminus B_{3r}(p)$ is contained in a path-connected component $C_r$ of $B_{5r}(p) \setminus B_{2r}(p)$.

    Assume that this is not the case, so that there exists a sequence $r_i \uparrow + \infty$ violating the statement.
    Hence, for every $i \in \bb{N}$, there are two points $x_i,y_i \in B_{4r_i}(p) \setminus B_{3r_i}(p)$ such that no path in $B_{5r_i}(p) \setminus B_{2r_i}(p)$ connects them.
     By \cite{DePhilippisGiglicone,DephilGigli} after \cite[Theorem 7.6]{ChCo1}, up to passing to a subsequence, the spaces $(X,\sd/r_i)$ converge in pGH sense to a metric cone $C(Z)$ centered in its tip, where $(Z,\sd_Z)$ is a path-connected metric space. Hence, there are limit points $x_i \to x_\infty$, $y_i \to y_\infty$ in $\bar{B}^{C(Z)}_{4}(p) \setminus {B}^{C(Z)}_{3}(p)$. Consider a collection $\{z_j\}_{j=1}^k \subset \bar{B}^{C(Z)}_{4}(p) \setminus {B}^{C(Z)}_{3}(p)$ for some $k \in \bb{N}$ such that $z_1=x_\infty$, $z_k=y_\infty$ and $\sd^{C(Z)}(z_j,z_{j+1}) \leq 1/2$, for every $j \in \{1,\cdots,k-1\}$. For every $i \in \bb{N}$ and every $j \in \{2,\cdots,k-1\}$ consider points $x^i_j \in B_{5r_i}(p) \setminus B_{2r_i}(p)$ such that $x^i_j \to z_j$, and set $x^i_1:=x_i$ and $x^i_k:=y_i$. Consider a curve obtained by concatenating geodesics connecting $z^i_j$ and $z^{i}_{j+1}$ as $j$ varies between $1$ and $k-1$. For $i$ large enough, this curve is contained in $B_{5r_i}(p) \setminus B_{2r_i}(p)$, a contradiction.
    
\medskip

        Let $L>0$ be large to be chosen later, and assume by contradiction that the images of the maps in \eqref{eq:inclusion-cc} have cardinality greater than $L$ for some sequence $r_i \uparrow + \infty$. Let $A_i:=B_{6r_i}(p) \setminus \bar{B}_{r_i}(p)$, and let $\bar{\sd}_i$ be the induced length distance on $A_i$.
        Since $A_i$ is an open set of an $\RCD(0,n)$ space, it is semi-locally simply connected by \cite{RCDsemilocally}. Hence, there exists a simply connected universal covering $\tilde{A}_i$ of $A_i$.
         On $\tilde{A}_i$, we consider the distance $\tilde{\sd}_i$ obtained by lifting the length structure of $A_i$ as in Remark \ref{rmk:lengthdistanceoncovering}. The deck transformations of this covering form of a group of isometries of $(\tilde{A}_i,\tilde{\sd}_i)$ isomorphic to $\pi_1(A_i)$, and the corresponding quotient space is isometric to $(A_i,\bar{\sd_i})$.
\medskip

        Let $p_i:\tilde{A}_i \to A_i$ be the covering map, and let $\tilde{A}'_i$ be a path-connected component of $p_i^{-1}(C_{r_i})$. 
        Let $G_i < \pi_1(A_i)$ be the subgroup of deck transformations of $p_i:\tilde{A}_i \to A_i$ sending $\tilde{A}'_i$ into itself.
        Then, since $G_i$ acts freely on $\tilde{A}_i$, it can be naturally identified with a group of isometries of $(\tilde{A}'_i,\tilde{\sd}_i)$. Moreover, since $\pi_1(A_i)$ acts transitively on the fibers of $p_i:\tilde{A}_i \to A_i$, it follows that $G_i$ acts transitively on the fibers of the restriction 
        \begin{equation}
        {p_i}|_{\tilde{A}'_i}:\tilde{A}'_i \to  C_{r_i}\, . 
        \end{equation}
        Thus, $(\tilde{A}'_i,\tilde{\sd}_i)/G_i \cong (C_{r_i},\bar{\sd}_i)$ isometrically. 
        In addition, since loops contained in $C_{r_i}$ can be lifted to paths contained in $\tilde{A}'_i$, each group $G_i$ contains an isomorphic copy of the image of $i_*:\pi_1(C_{r_i},\bar{c}_i) \to \pi_1(A_i,\bar{c}_i)$ where $\bar{c}_i \in C_{r_i}$.
\medskip

        The spaces $(\tilde{A}'_i,\tilde{\sd}_i/r_i)$ are locally uniformly doubling by Lemma \ref{L|lift}. Hence, modulo passing to a subsequence, one has the equivariant GH convergence 
        \begin{equation}
        (\tilde{A}'_i,\tilde{\sd}_i/r_i,G_i) \to (Z,\sd_z,G)\, .
        \end{equation} Observe now that by an argument similar to the one used at the beginning of the proof, for every $\epsilon>0$, each $(C_{r_i},\bar{\sd}_i/r_i)$ is $\epsilon$-close to $(B_{5r_i}(p) \setminus B_{2r_i}(p),\bar{\sd}_i/r_i)$ in GH sense for $i$ large enough. Combining with Lemma \ref{L|convanelli}, we deduce that
        the spaces $(C_{r_i},\bar{\sd}_i/r_i)$ converge to an annulus $(C_\infty,\bar{\sd}_\infty)$ of a tangent cone at infinity  $C(Y)$ of $X$, where $Y$ is a non-collapsed $\RCD(n-2,n-1)$ space, and $\bar{\sd}_\infty$ is the length distance induced by $C(Y)$ on $C_\infty$.
        By \cite[Theorem 1.2]{DephilGigli} (after \cite{ColdingConv}), $\aH^{n-1}(Y)$ is bounded from below by a constant depending on $n$ and $v$, so that $\# \pi_1(C_\infty) < C(n,v)$ by Lemma \ref{L|variantionzhou}.
         Note that here we are using the assumption that $n\ge 3$. In the remainder of the proof, we denote by $C(n,v)$ a generic constant depending on $n$ and $v$ only, without renaming it at every step.
        Since $\tilde{A}'_i/G_i \cong C_{r_i}$, by Theorem \ref{T|EquivaraintConvergence}, we have that 
        \begin{equation}
            (Z,\sd_z)/G \cong (C_{\infty},\bar{\sd}_\infty)\, .
        \end{equation}
        We now claim that $\# G \leq  C(n,v)$. We divide the proof of this fact in steps.
        \medskip 

        \textbf{Step 1}: We claim that $\Gamma_\epsilon:=\{g \in G: g \cdot x \in B_\epsilon(x) \text{ for some } x \in Z\}$ has cardinality at most $C(n,v)$, where $\epsilon:=\epsilon'/10$ and $\epsilon'(n,v)$ is given by Lemma \ref{L|Zhou}. 

        Assume that there are $N>0$ distinct elements $\{g_k\}_{k =1}^N \subset \Gamma_\epsilon$ and (possibly repeated) points $x_k \in Z$ such that $g_k \cdot x_k \in B_\epsilon(x_k)$.
        For every $k \in \{1,\dots,N\}$, choose points $x_{i,k} \in \tilde{A}'_i$ converging to $x_k$, and isometries $g_{i,k} \in G_i$ converging to $g_k$. If $k \neq k'$, since $g_k \neq g_{k'}$, we have that $g_{i,k} \neq g_{i,k'}$ for $i$ large enough.

        Consider the correspondence between deck transformations of $\tilde{A}_i$ and $\pi_1(A_i)$. Each $g_{i,k}$ corresponds to a loop $\gamma_{i,k}$ of $\pi_1(A_i)$ centered in $p_i(x_{i,k}) \in A_i$ obtained by projecting a curve of $\tilde{A}_i$ from $x_{i,k}$ to $g_{i,k} \cdot x_{i,k}$. In particular, such curve can be chosen to satisfy 
        \begin{equation}
            \mathrm{length}^{\bar{\sd}_i}(\gamma_{i,k})\leq 2\tilde{\sd}_i(x_{i,k},g_{i,k} \cdot x_{i,k})\, .
        \end{equation}
        Since $x_{i,k} \to x_k$ and $g_{i,k}\cdot x_{i,k} \to g_k\cdot x_k \in B^{\sd_z}_\epsilon(x_k)$ as
        $(\tilde{A}'_i,\tilde{\sd}_i/r_i) \to (Z,\sd_z)$, it follows that  
        \begin{equation} \label{E|simplvol1}
            \limsup_{i \to + \infty} \, \mathrm{length}^{\bar \sd_i/r_i}(\gamma_{i,k}) \leq 2
            \limsup_{i \to + \infty}
            \frac{\tilde{\sd}_i(x_{i,k},g_{i,k} \cdot x_{i,k})}{r_i} \leq  2 \epsilon\, .
        \end{equation}
        Observe now that $C_\infty \cong Z/G$ is $C(n)$-doubling and has diameter bounded above by a possibly different $C(n)$. As a consequence, we deduce that there exists $T(N,n)>0$ (with the property that $T(N,n) \uparrow + \infty$ as $N \uparrow + \infty$) such that, if $C_\infty$ contains a collection of points  $\{c_l\}_{l=1}^N$, then there are at least $T$ points of such collection contained in a ball $B^{\bar{\sd}_\infty}_{\epsilon/2}(c)$.
        Using this principle with the collection $\{G \cdot x_k\}_{k=1}^N \subset C_\infty$, we find $T$ different indexes $k_1,\dots,k_T$ such that $\bar \sd_\infty (G\cdot x_{k_j}, G\cdot x_{\infty})<\epsilon/2$ for all $j=1,\dots,T$ and some $G \cdot x_\infty \in Z/G$.
        By Theorem \ref{T|EquivaraintConvergence}, the projections $G_i \cdot x_{i,k_j} \in (C_{r_i},\bar{\sd}_i/r_i)$ converge to $G \cdot x_{k_j}$ in the space realizing the convergence $(C_{r_i},\bar{\sd}_i/r_i) \to (C_\infty,\bar{\sd}_\infty)$, for all $j=1,\dots,T$.
        Moreover, in the same space, there are points $a_i \in C_{r_i}$ such that $a_i \to G \cdot x_\infty$.
        Hence, for $i$ large enough
\begin{equation}
    G_i \cdot x_{i,k_j}\in B^{\bar{\sd}_i/r_i}_{\epsilon}(a_i) \subset C_{r_i}
\end{equation}
     for all $j=1,\dots,T$.   
        
        Observe that for $\epsilon$ small enough we have
        $B^{\bar{\sd}_i/r_i}_{\epsilon}(a_i) = B^{\sd/r_i}_{\epsilon}(a_i)$. Moreover, for every $k$ and every $i$, it holds that $G_i \cdot x_{i,k} \in C_{r_i}$ is identified with $p_{i}(x_{i,k}) \in A_i$ via the inclusion $C_{r_i} \to A_i$.
        Hence, combining with \eqref{E|simplvol1}, the balls
        $B^{\sd}_{3r_i\epsilon}(a_i) \subset X$ contain $T$ loops $\gamma_{i,k}$ corresponding to different elements of $\pi_1(A_i)$. 
        By Lemma \ref{L|Zhou}, it follows that $T \leq C(n,v)$, so that $N \leq C(n,v)$ as well.
 \medskip    

        We record two important consequences of Step 1:
        \begin{itemize}
        \item[i)] the set $\mathrm{Fix}(G):=\{g \in G: g\cdot x =x \text{ for some } x \in Z\}$ has cardinality at most $C(n,v)$; \item[ii)] for every $z \in Z$, there is $\epsilon>0$ such that the almost stabilizer $\Gamma_\epsilon(z):=\{g \in G:g \cdot z \in B_\epsilon(z)\}$ has cardinality at most $C(n,v)$.
        \end{itemize}
        \medskip

        \textbf{Step 2}: We claim that $H:=\langle\mathrm{Fix}(G)\rangle$ is normal and has cardinality at most $C(n,v)$.

        Let $h \in \mathrm{Fix}(G)$ be such that $h \cdot z=z$ for some $z \in Z$, and let $g \in G$. Then $g^{-1}hg$ fixes $g^{-1}(z)$, so that $g^{-1}hg \in \mathrm{Fix}(G)$.
        Let now $h \in H$, so that $h=g_1 \cdots g_k$ for some $\{g_j\}_{j=1}^{k} \subset \mathrm{Fix}(G)$. For every $g \in G$, we have that
        \begin{equation}
            g^{-1}hg=g^{-1}g_1g \, g^{-1}g_2 \cdots g_{k-1}g \, g^{-1}g_kg\, .
        \end{equation}
        Since $g^{-1}g_jg$ belongs to $\mathrm{Fix}(G)$ for every $j$, it follows that $g^{-1}hg \in H$, proving that $H$ is normal.

        To prove that $H$ has cardinality at most $C(n,v)$, we first claim that $H/Z(H)$ has cardinality at most $C(n,v)$, where $Z(H)<H$ is the center of $H$. 
        Let $T:H \to \mathrm{Sym}(\mathrm{Fix}(G))$ be the homomorphism defined so that $T(h)$ is the permutation of $\mathrm{Fix}(G)$ given by $T(h)(g):=h^{-1}gh$. By Step 1, $H/\mathrm{Ker}(T)$ has cardinality at most $C(n,v)$. Since $\mathrm{Ker}(T)$ is in the center of $H$ it follows that $H/Z(H)$ has cardinality at most $C(n,v)$, as claimed.

        By Schur's Theorem, the derived group $H' \leq H$ has cardinality at most $C(n,v)$. Therefore it suffices to show that $H/H'$ has cardinality at most $C(n,v)$ as well. The group $H/H'$ is abelian, and it is generated by at most $C(n,v)$ elements of order at most $C(n,v)$ (i.e.\ the classes of the elements in $\mathrm{Fix}(G)$). Hence, $H/H'$ has cardinality at most $C(n,v)$, concluding the proof of Step 2. 
        \medskip

        \textbf{Step 3:} We claim that $G/H$ acts properly discontinuously on $Z/H$.

        We first show that $G/H$ acts freely. Indeed if $(g+H)\cdot x=H\cdot x$ for some $H\cdot x \in Z/H$ and some $g+H \in G/H$ then there exists $h \in H$ such that $(g+h)\cdot x=x$, so that $g+h \in H$ and $g+H=H$. Hence, the action is free.

        To conclude, it suffices to show that if $H \cdot x \in Z/H$ and $g_i+H \in G/H$ are such that $(g_i+H) \cdot x \to H \cdot x$, then $g_i+H=H$ for $i$ large enough. To this aim, observe that there are elements $h_i \in H$ such that $(g_i+h_i) \cdot x \to x$. By Step 1, it follows that $g_i+h_i$ fixes $x$ for $i$ large enough, so that $g_i \in H$ because $G/H$ acts freely.
        \medskip

        To complete the proof that $\# G \leq C(n,v)$, note that since $G/H$ acts properly discontinuously on $Z/H$ by Step 3, $G/H$ is a subgroup of $\pi_1(C_\infty)$. Thus it has cardinality at most $C(n,v)$. Combining with Step 2, we deduce that $G$ has cardinality at most $C(n,v)$ as well. 
        
        We now conclude the proof of the theorem. Since $Z/G$ is compact, we deduce that $Z$ is compact.
         Let $a_i \in \tilde{A}_i'$ be elements such that $a_i \to z \in Z$ as $(\tilde{A}_i',\tilde{\sd}_i/r_i) \to (Z,\sd_z)$. 
         Let $\epsilon(n,v)>0$ be given by Lemma \ref{L|Zhou}, so that the almost stabilizers $\Gamma_{\epsilon,a_i}:=\{g \in G_i:\tilde{\sd}_i(a_i,g \cdot a_i)< \epsilon r_i \}$ all have cardinality at most $C(n,v)$. Since we are assuming by contradiction that $\# G_i \geq L$, the index of $\Gamma_{\epsilon,a_i}$ in $G_i$ is at least $S:=L/C(n,v)$. Consider elements $\{g^j_i\}_{j=1}^S \subset G_i$ belonging to different left cosets of $\Gamma_{\epsilon,a_i}$ in $G_i$. Since $Z$ is compact, for each $j \in \{1,\cdots S\}$, there exists $g^j \in G$ such that $g^j_i \to g^j$.

         We claim that the elements $\{g^j\}_{j=1}^S \subset G$ are all distinct. If this is the case, we conclude that
         \begin{equation}
             L/C(n,v) =S \leq \# G \leq C(n,v)\, ,
         \end{equation}
         contradicting a suitably large choice of $L$.
        We show that $g^1 \neq g^2$, the other cases being analogous. If $g^1=g^2$, then $\sd_z(g^1 \cdot z,g^2 \cdot z)=0$, so that
        \begin{equation}
            \tilde{\sd}_i (g^1_i \cdot a_i,g^2_i \cdot a_i)< \epsilon r_i
        \end{equation}
        if $i$ is large enough. This implies that $g^1_i \in g^2_i + \Gamma_{\epsilon,a_i}$, contradicting that the elements $g^j_i$ belong to different left cosets of $\Gamma_{\epsilon,a_i}$ in $G_i$.
    \end{proof}

\subsection{Proof of Theorem \ref{thm:simplicialvolthm}: non-Euclidean volume growth case}\label{sec:vanish_nonEucl}

The goal of this section is to prove Theorem \ref{thm:simplicialvolthm} under the assumption that the space $X$ does not have maximal volume growth. We observe that if $X$ is compact, then $\pi_1(X)$ is virtually abelian and thus amenable by \cite[Theorem 1.3]{MondinoWei}, so that $\|X\|=0$ by \cite[Corollary C, page 248]{GromovVolbCoho}.
In the following, we will assume that $X$ is non-compact. 
To prove the vanishing theorem in this case, thanks to Theorem \ref{T|Frig}, it is sufficient to prove the following:

\begin{theorem} \label{T|Noverlap}
    Let $(X,\sd,\Haus^n)$ be a non-compact $\RCD(0,n)$ space such that 
    \begin{equation}\label{eq:nonEuclvg}
        \lim_{R\to\infty} \frac{\Haus^n(B_R(p))}{R^n} = 0\, ,
    \end{equation}
    for some $p\in X$.
    Then $X$ admits a cover $\{ U_i\}_{i\in \N}$ which is amenable, open, amenable at infinity and has $\mathrm{mult}\big(\{ U_i\}_{i\in \N}\big) \leq n$.
\end{theorem}

As said before, Theorem \ref{T|Noverlap} generalizes some results due to Gromov in \cite{GromovVolbCoho}, where he assumed a uniform small upper bound on the volume of unit balls. 
The rough idea behind the proof of Theorem \ref{T|Noverlap} is the following. In view of \eqref{eq:nonEuclvg}, \cite[Theorem 1.2]{DephilGigli} implies that that all the blow-downs of $X$ have Hausdorff dimension $\leq n-1$. In particular, they have topological dimension at most $n-1$ and thus they admit arbitrarily small open covers with multiplicity at most $n$. Ideally, one would like to combine the existence of such covers with small multiplicity with the Margulis lemma from \cite{KapovitchWilking} (and the subsequent partial generalization to $\RCD$ spaces in \cite{DengSantosZamoraZhao}) to construct the sought cover of $X$.
There are several difficulties when trying to implement this strategy. The most severe ones are that:
\begin{itemize}
    \item[i)] in general, it is not clear whether the pointed GH-approximations from the (rescaled)-space to its blow-downs can be chosen to be continuous, due to the possibly poor topological behaviour of blow-downs in the collapsed case;
    \item[ii)] the blow-down of $X$ need not be unique and need not be invariant under rescaling. Thus it is hard to turn the local information obtained via a blow-down argument into a global one;
    \item[iii)] a local version of the Margulis lemma analogous to the one obtained in \cite{KapovitchWilking} in the smooth case is not available yet for $\RCD$ spaces, although a more global Margulis lemma was recently obtained in such context in \cite{KapovitchWilking}.
\end{itemize}

The key tool to circumvent the difficulties in i) and ii) will be a recent result by Papasoglu stated below as Theorem \ref{thm:papasoglu}. 
Using Theorem \ref{thm:papasoglu} we show in Proposition \ref{prop:margulis}
that a metric space with volume growth of polynomial degree less than $n$, after a suitable conformal rescaling of the metric, has small $(n-1)$-Urysohn width (see Definition \ref{D|width}) outside of sufficiently large compact sets. 
To complete the proof of Theorem \ref{T|Noverlap} we are going to combine such control on the $(n-1)$-width with Theorem \ref{thm:margulis}, where we obtain a weaker version of the Margulis lemma from \cite{KapovitchWilking} with no quantitative bound on the index of the nilpotent subgroup.

\begin{definition} \label{D|width}
   Let $q\in\N$. We say that a metric space $(X,\sd)$ has $q$-Uryson width less than or equal to $W$, and write $UW_q(X,\sd)\leq W$ if for every $\epsilon>0$ there exist a $q$-dimensional simplicial complex $Y$ and a continuous map $\pi: X \to Y$ such that every fiber $\pi^{-1}(y)$, $y\in Y$, has diameter $\leq W+\epsilon$.
\end{definition}
In the following, we will denote by $\haus_\infty^{n}$ the $n$-dimensional Hausdorff content for $n\ge 0$. The next statement corresponds to \cite[Theorem 3.3]{PapasogluGAFA}.

\begin{theorem}\label{thm:papasoglu}
    For every $n\in \N$ there exists $\varepsilon_n>0$ with the following property. Let $(X,\sd)$ be a proper metric space and $R>0$ be such that for every $x\in X$ there is an open set $U_x$ with the properties:
    \begin{itemize}
        \item[(i)] $B_R(x)\subset U_x \subset B_{10 R}(x)$
        \item[(ii)]$\haus_\infty^{n-1}(\partial U_x) \leq  \varepsilon_n R^{n-1}$.
    \end{itemize}
    Then for every $\epsilon>0$ there exist a \emph{locally finite} $(n-1)$-dimensional simplicial complex $Y$ and a \emph{proper} continuous map $\pi: X \to Y$ such that every fiber $\pi^{-1}(y)$, $y\in Y$, has diameter $\leq R+\epsilon$. In particular,
    $UW_{n-1}(X)\leq R$. 
\end{theorem}

\noindent The statement of \cite[Theorem 3.3]{PapasogluGAFA} does not mention that, in the previous theorem, the simplicial complex $Y$ and the map $\pi$ can be taken to be locally finite and proper, respectively. However, this is a consequence of the proof in \cite{PapasogluGAFA}.
\medskip

The following Coarea Inequality follows from \cite[Theorem 1.1]{coarea}.

\begin{proposition}[Coarea inequality] \label{prop:coarea}
    Let $(X,\sd)$ be a proper metric space and $n>1$. Then for every $x\in X$ and $0<\alpha<\beta< \infty$ we have that
    \begin{equation}
        \int^\beta_{\alpha} \haus^{n-1}(\partial B_s(x)) d s \leq C_n \, \big( \haus^n(B_\beta(x)) - \haus^n(B_\alpha(x))\big)\, ,
    \end{equation}
    where $C_n$ is a dimensional constant.
\end{proposition}

To prove Theorem \ref{T|Noverlap} the key idea is to apply Theorem \ref{thm:papasoglu} to a conformal rescaling of $(X,\dist)$. Lemma \ref{lem:modifieddistance} below introduces the relevant conformal rescaling and establishes its basic properties. The forthcoming Proposition \ref{prop:margulis} establishes a the sought estimate for the Urysohn width of the conformal rescaling away from a bounded set.

\begin{lemma}\label{lem:modifieddistance}
    Let $(X,\sd,p)$ be a complete, separable and geodesic pointed metric space. Given any curve $\gamma \in AC ([0,1], (X,\sd))$ we set
    \begin{equation}
        \tilde{L}(\gamma):= \int_0^1 \frac{|\dot\gamma|(t)}{\sd (p,\gamma(t)) \vee 1} dt\, ,
    \end{equation}
    where $|\dot\gamma|$ denotes the metric derivative. Accordingly, we define the function $\tilde \sd:X\times X\to \R_+$ as 
    \begin{equation}
        \tilde \sd (x,y) := \inf \{\tilde L(\gamma)\,:\, \gamma \in AC ([0,1], (X,\sd)), \gamma(0)=x,\,\gamma(1)=y\}\, .  
    \end{equation}
    Then $\tilde\sd$ is a distance on $X$. Moreover, there exists $0<\delta_0\leq \frac 16$ such that for all $\delta\in(0, \delta_0]$ and every $q\in X$ with $\sd(p,q)>10$, we have that
    \begin{equation}\label{eq:inclusionsconfresc}
        \tilde B_\delta (q) \subset B_{2 \delta \sd(p,q)} (q) \subset B_{3 \delta \sd(p,q)} (q) \subset\tilde B_{6\delta} (q)\, ,
    \end{equation}
    where we denote by $\tilde B$ balls with respect to the distance $\tilde \sd$. 
\end{lemma}

\begin{proof}
    It is straightforward to check that $\tilde\sd$ is symmetric, finite, and satisfies the triangle inequality. Thus, to prove that $\tilde \sd$ is a distance we need only to check positivity. To this aim, consider $a,b\in X$ with $\tilde \sd(a,b)=0$ and assume by contradiction that $\delta:=\sd(a,b)>0$. Let $\gamma \in AC ([0,1], (X,\sd))$ be any curve connecting $a$ and $b$ and call $t_0\in [0,1]$ the first intersection time of $\gamma$ with $\partial B_{\delta/2}(a)$. Then we can estimate
    \begin{align}
        \int_0^1 \frac{|\dot\gamma|(t)}{\sd (p,\gamma(t)) \vee 1} dt \geq & \int_0^{t_0} \frac{|\dot\gamma|(t)}{\sd (p,\gamma(t)) \vee 1} dt \\
         \geq & \int_0^{t_0} \frac{|\dot\gamma|(t)}{\sd (p,a) + \delta/2 + 1} dt\\
        \geq &\frac{\delta/2}{\sd (p,a) + \delta/2 + 1} >0\, .
    \end{align}
    Since $\gamma$ was arbitrary, we obtain the desired contradiction. 
\medskip

     Until the end of the proof we fix $\delta_0 \leq 1/6$ with the following property: for every $\delta\in(0, \delta_0]$, if $x>0$ is such that $\log(1+x)\leq \frac 32 \delta$, then $x\leq 2 \delta$. 
\medskip
     
 We start by proving the third inclusion in \eqref{eq:inclusionsconfresc}. Take $q\in X$ with $\sd(p,q)>10$, any $\delta\in(0, \delta_0]$ and consider any $a\in B_{3 \delta \sd(p,q)} (q)$. Let $\gamma \in AC ([0,1], (X,\sd))$ be a constant speed geodesic (with respect to $\sd$), joining $a$ to $q$. In particular, we have that 
 \begin{equation}
 |\dot \gamma|\equiv \sd(a,q)\leq 3 \delta \sd(p,q)\, ,\quad \gamma([0,1]) \subset B_{\sd(p,q)/2}(q) \subset X \setminus B_{\sd(p,q)/2}(p)\, . 
 \end{equation}
 Thus, for every $t \in [0,1]$ we have that $\sd (p,\gamma(t)) \vee 1 \geq \frac 12 \sd(p,q)$ and therefore
    \begin{equation}
        \tilde \sd(a,q) \leq \tilde L(\gamma) = \int_0^1 \frac{|\dot\gamma|(t)}{\sd (p,\gamma(t)) \vee 1} dt \leq \int_0^1 \frac{3 \delta \sd(p,q)}{\frac 12 \sd (p,q)} dt = 6 \delta\, .
    \end{equation}
    Since $a$ was arbitrary, we deduce that $ B_{3 \delta \sd(p,q)} (q) \subset\tilde B_{6\delta} (q)$.
\medskip

    For the first inclusion in \eqref{eq:inclusionsconfresc}, take any $b\in\tilde B_{\delta} (q)$ and consider a curve $\gamma \in AC ([0,1], (X,\sd))$ joining $q$ to $b$ such that $\tilde L(\gamma)\leq \frac 32 \delta$. Since $\tilde L$ is invariant by reparametrisation, we can assume that $\gamma$ has constant speed with respect to the distance $\sd$, i.e.~$|\dot \gamma|\equiv L$ for some $L>0$. We claim that 
    \begin{equation}\label{eq:curveincl}
    \gamma([0,1])\subset X\setminus B_2(p)\, . 
    \end{equation}
    Arguing by contradiction, let $t_1 \in [0,1]$ be the first intersection time of $\gamma$ with $\partial B_2(p)$ and let $t_0 \in [0,1]$ be defined as
    \begin{equation}
        t_0 := \sup \{t<t_1 \,:\, \gamma(t)\not \in B_4(p)\}\, .
    \end{equation}
    With such choices we can estimate
    \begin{equation}
        \frac 32 \delta \geq \int_{t_0}^{t_1} \frac{|\dot\gamma|(t)}{\sd (p,\gamma(t)) \vee 1} dt = \int_{t_0}^{t_1} \frac{|\dot\gamma|(t)}{\sd (p,\gamma(t))} dt \geq \int_{t_0}^{t_1} \frac{|\dot\gamma|(t)}{4} dt \geq \frac 12\, ,
    \end{equation}
    thus yielding the desired contradiction, since $\delta\leq \delta_0\leq 1/6$.
    On the other hand, we have that 
    \begin{equation}
        \sd(p,\gamma(t)) \leq \sd(p,q) + \sd(q,\gamma(t)) \leq \sd(p,q) + Lt\, .
    \end{equation}
    Thus, taking into account \eqref{eq:curveincl}, we deduce that 
    \begin{equation}
        \frac 32 \delta \geq \int_{0}^{1} \frac{|\dot\gamma|(t)}{\sd (p,\gamma(t))} dt \geq \int_{0}^{1} \frac{L}{\sd (p,q)+ Lt} dt = \log \bigg( 1 + \frac{L}{\sd(p,q)}\bigg)\, .
    \end{equation}
    Then, our choice of $\delta_0$ implies that $L/\sd(p,q) \leq 2\delta$. We can conclude by observing that $\sd(q,b) \leq L \leq 2\delta \sd(p,q)$. Since $b \in\tilde B_{\delta} (q)$ was arbitrary,  we proved that $\tilde B_{\delta} (q)\subset B_{2 \delta \sd(p,q)} (q)$. \end{proof}

\begin{proposition}\label{prop:margulis}
    Fix $n\in\N$. Let $(X,\sd,p)$ be a complete, separable, proper, non-compact and geodesic pointed metric space. Let $\tilde \sd$ and $\delta_0$ be the distance and the parameter given by Lemma \ref{lem:modifieddistance}. Assume that 
    \begin{equation}\label{eq:NOTmaximalvolume}
        \lim_{R\to\infty} \frac{\Haus^n(B_R(p))}{R^n} = 0\, .
    \end{equation}
    Then, for every $\delta\in(0, \delta_0]$ there exists $T(\delta)>20$ such that the following holds. There exist a locally finite $(n-1)$-dimensional simplicial complex $Y$ and a proper continuous map $\pi:(X\setminus B_{T(\delta)}(p), \tilde \sd) \to Y$ with fibers of diameter less than $\delta$. In particular, 
    \begin{equation}
    UW_{n-1} (X\setminus B_{T(\delta)}(p), \tilde \sd)\leq \delta\, .
    \end{equation}
\end{proposition}

    \begin{proof}
Our goal is to show that $(X\setminus B_{T(\delta)}(p), \tilde \sd)$ satisfies the assumptions of Theorem \ref{thm:papasoglu}. 
    
    Fix any $\delta\in(0, \delta_0]$ and take any $q\in X$ such that $\sd(p,q)>20$. Given any $\eta\in [0,\delta]$, observe that for every $q' \in \partial B_{(2 \delta + \eta) \sd(p,q)}(q)$ we have
    \begin{equation}\label{eq:qprime}
        \sd(p,q') \geq \sd(p,q) - (2 \delta + \eta) \sd(p,q) \geq \frac 12 \sd(p,q)\, .
    \end{equation}
    Suppose now that $A\subset \partial B_{(2 \delta + \eta) \sd(p,q)}(q)$ has $\diam_{\tilde \sd}(A) =\rho \ll \delta_0$, where $\diam_{\tilde \sd}$ denotes the diameter computed with respect to the distance $\tilde \sd$. In particular, there exist $a,b\in A$ such that $\tilde \sd(a,b)> \frac \rho 2$, thus $b\not\in \tilde B_{\rho/2}(a)$. Taking into account \eqref{eq:qprime} (which also ensures that $\sd(p,a)>10$) and applying Lemma \ref{lem:modifieddistance}, we obtain
    \begin{equation}
        B_{\rho \sd(p,q)/8} (a) \subset B_{\rho \sd(p,a)/4} (a) \subset \tilde B_{\rho/2} (a)\, .
    \end{equation}
    Thus, $b \not\in B_{\rho \sd(p,q)/8} (a)$ and, denoting by $\diam_{\sd}$ the diameter computed with respect to the distance $\sd$, we conclude that 
    \begin{equation}\label{eq:diamestimate}
        \diam_\sd (A) \geq \frac{\sd(p,q)}{8} \rho = \frac{\sd(p,q)}{8} \, \diam_{\tilde \sd} (A)\, .
    \end{equation}

    In the following, for every $d,\zeta>0$, we denote by $\Haus_\zeta^d$ and $\tilde\Haus_\zeta^d$ the $\zeta$-approximate $d$-dimensional Hausdorff measures with respect to $\sd$ and $\tilde \sd$, respectively. Accordingly, we denote by $\Haus^d$ and $\tilde\Haus^d$ the respective $d$-dimensional Hausdorff measures.
    Thanks to \eqref{eq:diamestimate}, for every $\rho \ll \delta_0$ we obtain that
    \begin{equation}
    \begin{split}
        & \tilde \Haus_\rho^{n-1} \big(\partial B_{(2 \delta + \eta) \sd(p,q)} (q)\big)\\
        = \inf &\left\{ \sum_{i\in \N} \diam_{\tilde \sd}(A_i)^{n-1} \,:\, \bigcup_{i\in \N} A_i = \partial B_{(2 \delta + \eta) \sd(p,q)}(q), \, \diam_{\tilde \sd}(A_i) \leq \rho \right\}\\
        \leq \inf& \Bigg\{ \frac{8^{n-1}}{\sd(p,q)^{n-1}}\sum_{i\in \N} \diam_{\sd}(A_i)^{n-1} \,:\, \bigcup_{i\in \N} A_i = \partial B_{(2 \delta + \eta) \sd(p,q)} (q)\, ,\,\\
        &\diam_{\sd}(A_i) \leq \rho\, \frac{\sd(p,q)}{8} \Bigg\}
        =\frac{8^{n-1}}{\sd(p,q)^{n-1}} \,\, \Haus_{\rho\, \frac{\sd(p,q)}{8}}^{n-1} \big(\partial B_{(2 \delta + \eta) \sd(p,q)}(q)\big)\, .
    \end{split}
    \end{equation}
    Taking the limit as $\rho\to 0$, we conclude that for every $q\in X$ with $\sd(p,q)>20$ and every $\eta\in [0,\delta]$ there holds
    \begin{equation}\label{eq:Hequiv}
        \tilde \Haus^{n-1} \big(\partial B_{(2 \delta + \eta) \sd(p,q)}(q)\big)\leq \frac{c}{\sd(p,q)^{n-1}} \,\, \Haus^{n-1} \big(\partial B_{(2 \delta + \eta) \sd(p,q)}(q)\big)\, ,
    \end{equation}
    where, for notational simplicity, we put $c=8^{n-1}$.
\medskip

    Applying Proposition \ref{prop:coarea} with $x=q$, $\alpha=2\delta\sd(p,q)$ and $\beta=3\delta\sd(p,q)$ we obtain 
    \begin{equation}
        \fint^{3\delta\sd(p,q)}_{2\delta\sd(p,q)} \Haus^{n-1}(\partial B_s(q)) d s \leq C_n \, \frac{ \Haus^n(B_{3\delta\sd(p,q)}(q)) - \Haus^n(B_{2\delta\sd(p,q)}(q))}{\delta \sd(p,q)}\, .
    \end{equation}
    Thus there exists $\eta(q)\in (0,\delta)$ such that 
    \begin{equation}\label{eq:lessthanaverage}
        \Haus^{n-1}(\partial B_{(2 \delta + \eta (q)) \sd(p,q)}(q)) \leq C_n \, \frac{ \Haus^n(B_{3\delta\sd(p,q)}(q)) - \Haus^n(B_{2\delta\sd(p,q)}(q))}{\delta \sd(p,q)}\, .
    \end{equation}
    The combination of \eqref{eq:Hequiv} and \eqref{eq:lessthanaverage} yields 
    \begin{align}
        \tilde \Haus^{n-1} \big(\partial B_{(2 \delta + \eta(q)) \sd(p,q)}(q)\big) &\leq c \,C_n\frac{\Haus^n(B_{3\delta\sd(p,q)}(q))}{\delta \sd(p,q)^n} \\
        &\leq c \, C_n \frac{\Haus^n(B_{2 \sd(p,q)}(p))}{\delta \sd(p,q)^n}\, . 
    \end{align}
    Thanks to \eqref{eq:NOTmaximalvolume} there exists $T(\delta)>20$ such that 
    \begin{equation}
        \tilde \Haus^{n-1} \big(\partial B_{(2 \delta + \eta(q)) \sd(p,q)}(q)\big)\leq \varepsilon_n  \delta^{n-1}\, , \qquad \forall q \in X\setminus B_{T(\delta)}(p)\, ,
    \end{equation}
    where $\varepsilon_N$ is the constant given by Theorem \ref{thm:papasoglu}. 
    Finally, we deduce that 
    \begin{align}
        \tilde \Haus_\infty^{n-1} \big(\partial B_{(2 \delta + \eta(q)) \sd(p,q)} (q) \cap (X\setminus B_{T(\delta)}(p))\big) 
        \leq &\tilde \Haus_\infty^{n-1} \big(\partial B_{(2 \delta + \eta(q)) \sd(p,q)}(q)\big) \\
        \leq& \tilde \Haus^{n-1} \big(\partial B_{(2 \delta + \eta(q)) \sd(p,q)}(q)\big)\\
        \leq & \varepsilon_n  \delta^{n-1}\, .
    \end{align}
    Thanks to Lemma \ref{lem:modifieddistance} we can apply Theorem \ref{thm:papasoglu} to the metric space $(X\setminus B_{T(\delta)}(p),\tilde \sd)$ and conclude the proof.
    \end{proof}

    \begin{proof}[Proof of Theorem \ref{T|Noverlap}]
    We fix $\delta < \min \{\delta_0/2, \varepsilon/8\}$, where $\delta_0$ is the parameter identified in Lemma \ref{lem:modifieddistance} and $\varepsilon=\varepsilon(n)$ is the constant appearing in Corollary \ref{rmk:margulis}.
\medskip

    By Proposition \ref{prop:margulis} there exist a locally finite $(n-1)$-dimensional simplicial complex $Y$ and a proper continuous map $\pi: X\setminus B_{T(\delta)}(p) \to Y$ such that every fiber $\pi^{-1}(y)$, $y\in Y$, has diameter $\leq \delta$.
    Thus, every $y\in Y$ has an precompact open neighborhood $V_y$ such that $\diam_{\tilde \sd}(\pi^{-1}(V_y))\leq 2\delta$. Indeed, if this is not the case, then we can find $y\in Y$ and a sequence of decreasing open neighborhoods $\{V_y^i\}_{i\in \N}$ of $y$ such that 
    \begin{equation}
    \diam_{\tilde \sd}(\pi^{-1}(V_y^i))> 2\delta\, ,\quad 
    \cap_{i\in \N}\,\bar V_y^i= \{y\}\, . 
    \end{equation}
    Thus, for every $i$ we can find $x_i\in \pi^{-1}(V_y^i)$ such that $\tilde \sd (x_i,\pi^{-1}(y)) > \frac \delta 2$. Since $\pi$ is proper, the sequence $\{x_i\}_{i\in \N}$ is contained in a compact set and thus (up to subsequences) it converge to a limit $\bar x\in X$. 
    By passing to the limit in the previous conditions, on the one hand we have $\tilde \sd (\bar x,\pi^{-1}(y)) \geq \frac \delta 2$. On the other hand, we have that 
    \begin{equation}
    \bar x \in \cap_{i\in \N} \, \pi^{-1}(\bar V^i_y)= \pi^{-1}(y)\, , 
    \end{equation}
    resulting in a contradiction.
\medskip

    Consider the covering $\{V_y\}_{y\in Y}$ of $Y$. Since $Y$ is proper, because it is a locally finite simplicial complex, and has topological dimension $n-1$, there exists a \emph{locally finite} refinement $\{\hat V_\alpha\}_{\alpha\in \N}$ of $\{V_y\}_{y\in Y}$ with multiplicity at most $n$. After discarding all sets $\hat V_\alpha$ such that $\hat V_\alpha \cap \pi(X\setminus B_{T(\delta)}(p)) =\emptyset$, we obtain a locally finite covering of $\pi(X\setminus B_{T(\delta)}(p))$, still denoted as $\{\hat V_\alpha \}_{\alpha \in \bb{N}}$, with multiplicity at most $n$.
    For every $\alpha \in \N$ we define $\hat U_\alpha:= \pi^{-1}(\hat V_\alpha)$. Observe that $\{\hat U_\alpha\}_{\alpha \in \N}$ is a locally finite open cover of $X\setminus B_{T(\delta)}(p)$ with $\mathrm{mult}\big(\{\hat U_\alpha\}_{\alpha\in \N}\big) \leq n$. Set
    \begin{equation}
         A := \{\alpha \in \N\,:\, \hat U_\alpha \cap B_{2 T(\delta)}(p) = \emptyset\}\, .
    \end{equation}
    Consider an enumeration $m:\N\setminus\{0\} \to A$
    and define the open cover $\{ U_i\}_{i\in \N}$ of $X$ as
    \begin{equation}
          U_0 := B_{2 T(\delta)}(p) \cup \bigcup_{\alpha\in \N \setminus A} \hat U_\alpha\, , \qquad U_i = \hat U_{m(i)}, \quad  i\in \N\, .
    \end{equation}

    By construction, $\{U_i\}_{i \in \N}$ is a locally finite open cover of $X$ such that $\mathrm{mult}\big(\{U_i\}_{i\in \N}\big) \leq n$. 
     Moreover, Corollary \ref{rmk:margulis} guarantees that for every bounded (open) set $U\subset X$, and $u \in U$, the image of the natural map induced by inclusion $i_*: \pi_1(U,u) \to \pi_1(X,u)$ is a virtually nilpotent and thus amenable subgroup. 
    Therefore, the cover $\{ U_i\}_{i\in \N}$ is amenable. For the sake of clarity we note that the main result of \cite{DengSantosZamoraZhao} would be sufficient to check the last claim.
\medskip

   For the application of Theorem \ref{T|Frig}, 
    we are left to prove that $\{ U_i\}_{i\in \N}$ is amenable at infinity. For every $\alpha \in \N$ the set $\hat V_\alpha$ is contained in $V_{y(\alpha)}$ for some $y(\alpha)\in Y$. Thus $\diam_{\tilde \sd}( U_i)\leq 2 \delta$ for every $i\in \N \setminus\{0\}$. In particular, for every $i\in \N \setminus\{0\}$, taking any $p_i \in U_i \subset X \setminus B_{T(\delta)}(p)$, Lemma \ref{lem:modifieddistance} yields that
    \begin{equation} \label{eq:inclusionamenability}
         U_i \subset \tilde B_{2 \delta} (p_i) \subset B_{4 \delta \sd(p,p_i)} (p_i)\, ,
    \end{equation}
    where we used that $\delta \leq \delta_0/2$. We introduce sequence $\{ W_i\}_{i\in \N}$ by setting
    \begin{equation}
        W_0:= X\, , \qquad  W_i := X \setminus B_{\frac 12 \sd(p,p_i)} (p)\, , \quad i\in \N \setminus\{0\}\, .
    \end{equation}
    We claim that conditions (i)-(iv) in Definition \ref{def:amenableatinfinity} holds. Condition (i) is trivial since $X$ is proper, and condition (iii) follows from \eqref{eq:inclusionamenability}. Moreover, since $\{U_i\}_{i\in \N}$ is locally finite, we deduce that $\sd(p,p_i)\to \infty$ as $i\to \infty$ and therefore (ii) holds. To check (iv), we apply Corollary \ref{rmk:margulis} to obtain that for every $i\in \N \setminus\{0\}$ the image of the inclusion map
    \begin{equation}
        \pi_1\big(B_{4 \delta \sd(p,p_i)} (p_i),p_i \big) \to \pi_1\big(B_{4 \delta \sd(p,p_i) / \varepsilon}(p_i),p_i \big)
    \end{equation}
    is virtually nilpotent. Our choice of $\delta$ guarantees that $B_{4 \delta \sd(p,p_i) / \varepsilon}(p_i)\subset  W_i$. Thus $U_i$ is amenable in $W_i$ for every $i\in \N \setminus\{0\}$. By Theorem \ref{T|Frig} we conclude that $||X||=0$.
\end{proof}

\section{Classification of \texorpdfstring{$\RCD(0,3)$}{RCD(0,3)} spaces with empty boundary}

The goal of this section is to prove Theorem \ref{thm:classRCD03intro}. 
Theorem \ref{T|classification} then follows by combining Theorem \ref{thm:classRCD03intro} with the results from \cite{Brenaorientability,QinDenRCD} as shown below.

\begin{proof}[Proof of Theorem \ref{T|classification} assuming Theorem \ref{thm:classRCD03intro}]
If $X$ is a non-compact topological manifold, the result follows by Theorem \ref{thm:classRCD03intro}. If $X$ is a compact topological manifold, we consider two separate cases. If $\pi_1(X)$ is infinite, then the universal covering $\tilde{X}$ of $X$ splits off a line by \cite[Theorem 1.3]{MondinoWei} (combined with \cite{RCDsemilocally}), so that $\tilde{X}$ is an Alexandrov space of non-negative sectional curvature by \cite{LytchakStadler}. Being a quotient of an Alexandrov space of non-negative sectional curvature by a discrete group, also $X$ is then an Alexandrov space of non-negative sectional curvature by, for instance, \cite[Corollary 8.7]{zbMATH07802912}. If $\pi_1(X)$ is finite, then $X$ is a spherical manifold by Perelman's solution of the Geometrization Conjecture.

Hence, we assume that $X$ is not a topological manifold. By \cite[Theorem 1.8]{BruePigSem}, at least one point of $X$ has all the tangent cones homeomorphic to $C(\bb{RP}^2)$. By \cite[Theorem 1.1 point 3]{QinDenRCD}, $X$ is non-orientable in the sense of \cite[Definition 1.6]{Brenaorientability}. Hence, by \cite[Theorem 3.1]{Brenaorientability}, $X$ admits a ramified double cover $(\tilde{X},\tilde{\sd})$, equipped with an involutive isometry $\Gamma:\tilde{X} \to \tilde{X}$ such that $\tilde{X}/\langle \Gamma \rangle \cong X$ isometrically. By \cite[Theorem 1.1 and Lemma 4.8]{QinDenRCD}, $(\tilde{X},\tilde{\sd},\aH^3)$ is a $\RCD(0,3)$ space without boundary and a topological manifold, and the set of fixed points of $\Gamma$ is locally finite and non-empty. 

Assume now that $X$ is compact. In this case, $\tilde{X}$ is compact as well. By the previous part of the proof, either $\tilde{X}$ is an Alexandrov space of non-negative sectional curvature, or it has finite fundamental group. In the former case, $X$ is an Alexandrov space with non-negative sectional curvature as well. Hence, we assume that $\tilde{X}$ has finite fundamental group. In this case, $X$ is a spherical suspension over $\bb{RP}^2$ by the same proof of \cite[Corollary 4.16]{QinDenRCD} (they only use the assumption $K>0$ to obtain that $\tilde{X}$ has finite fundamental group).

Assume now that $X$ is non-compact. In this case, $\tilde{X}$ is non-compact as well. By Theorem \ref{thm:classRCD03intro}, $\tilde{X}$ is either an Alexandrov space with non-negative sectional curvature, or it is homeomorphic to $\bb{R}^3$. In the former case, $X$ has non-negative sectional curvature as well, so that we may assume that $\tilde{X}$ is homeomorphic to $\bb{R}^3$. Consider the compactification at infinity of $\tilde{X}$, which, using the obvious identification, we denote $\bb{R}^3 \cup \{\infty\}$. Consider the homeomorphism $\Gamma_\infty:\bb{R}^3 \cup \{\infty\} \to \bb{R}^3 \cup \{\infty\}$ induced by $\Gamma:\tilde{X} \to \tilde{X}$, and observe that $\Gamma_\infty(\infty)=\infty$. By \cite[Theorem 4]{zbMATH03035170}, $\Gamma_\infty$ has exactly two fixed points, so that $(\bb{R}^3 \cup \{\infty\})/\langle \Gamma_\infty \rangle $ is homeomorphic to a spherical suspension over $\bb{RP}^2$ by \cite[Theorem 1.1]{zbMATH03148107}. Hence, $X \cong \tilde{X}/\langle \Gamma \rangle $ is homeomorphic to $C(\bb{RP}^2)$, concluding the proof.
\end{proof}

We turn to the proof of Theorem \ref{thm:classRCD03intro}.
The starting point, similarly to \cite{SchoenYau,Liu3d}, is the following:

\begin{lemma}\label{lemma:trivialpi_2}
Let $(X,\dist,\haus^3)$ be an $\RCD(0,3)$ space which is a topological manifold. Then $\pi_2(X)=\{0\}$ unless the universal cover of $X$ splits isometrically as $\R\times Z$ where $(Z,\dist_Z)$ is an Alexandrov space with nonnegative curvature homeomorphic to $S^2$.  
\end{lemma}

The proof follows from the verbatim arguments as in the smooth case discussed in \cite[Lemma 2]{SchoenYau}. We discuss the details for the sake of completeness.

\begin{proof}[Proof of Lemma \ref{lemma:trivialpi_2}]
Assume that $\pi_2(X)\neq 0$. Denoting by $(\overline{X},\overline{\dist},\haus^3)$ the universal cover of $X$ endowed with its natural $\RCD(0,3)$ structure (see \cite[Theorem 1.1]{MondinoWei} and \cite{RCDsemilocally}), $\pi_2(\overline{X})\neq 0$ as well, and $\overline{X}$ is simply connected. 
By Hurewicz's theorem $H_2(\overline{X},\Z)\equiv \pi_2(\overline{X})$. Thus $H_2(\overline{X},\Z)\neq 0$ and in particular $\overline{X}$ must be noncompact, otherwise by Poincar\'e duality and the universal coefficients theorem we would have $H_2(\overline{X},\Z)\equiv H_1(\overline{X},\Z)=0$, since $\overline{X}$ is simply connected.
By the sphere theorem \cite{Papakyriakopoulos} there exists an embedded $S^2\subset \overline{X}$ which is homotopically nontrivial. 
We claim that $\overline{X}\setminus S^2$ has at least two connected components. If not, then there exists a loop in $\overline{X}$ intersecting the $S^2$ transversally exactly once (note that we can endow the $3$-manifold $X$ with the unique smooth structure compatible with its topology). Thus its linking number with such $S^2$ would be $\pm 1$, a contradiction since the linking number is homotopically invariant and $\overline{X}$ is simply connected. Next we claim that no connected component of $\overline{X}\setminus S^2$ can be compact. Indeed, if one of them is compact, then the $S^2$ would bound a compact submanifold of $\overline{X}$. Therefore it would be homologically, and hence homotopically, trivial, a contradiction. Since $\overline{X}\setminus S^2$ has (at least) two connected components, $\overline{X}$ has two ends. Thus it splits a line, by \cite{gigli2013splittingtheoremnonsmoothcontext}. More precisely, $\overline{X}=\R\times Z$ as metric measure spaces, where $(Z,\dist_Z,\haus^2)$ is a compact $\RCD(0,2)$ space. By \cite[Theorem 1.1]{LytchakStadler}, $(Z,\dist_Z)$ is an Alexandrov space with nonnegative curvature. Since we are assuming that $X$ is a topological manifold (with empty boundary), $Z$ must be a topological surface (with empty boundary) as well. Since $\overline{X}$ is simply connected, $Z$ is simply connected as well. Thus $Z$ is homeomorphic to $S^2$, as we claimed. 
\end{proof}

We now distinguish two cases, the case when $X$ is simply connected, and the case when it is not. In the first case, the universal cover must split a line. We remark, that although the statement is completely analogous to the smooth, our proof departs from \cite{Liu3d} in the last step, where we rely on Theorem \ref{thm:Bettirigid} instead of minimal surfaces methods. 

\begin{lemma}\label{lemma:notsimplyconn}
Let $(X,\dist,\haus^3)$ be a noncompact $\RCD(0,3)$ space such that $X$ is a topological $3$-manifold. If $X$ is not simply connected, then the universal cover $(\overline{X},\overline{\dist})$ splits a line and $(X,\dist)$ is an Alexandrov space with nonnegative curvature.
\end{lemma}

\begin{proof}
   Thanks to Lemma \ref{lemma:trivialpi_2}, either $\pi_2(\overline{X})=0$, or $(\overline{X},\overline{\dist})$ is an Alexandrov space with nonnegative curvature, since it is a product of $\R$ with an Alexandrov space with nonnegative curvature. In the second case, $(X,\dist)$ is an Alexandrov space with nonnegative curvature as well. Thus we can assume that $\pi_2(\overline{X})=0$. Under such assumption $\overline{X}$ is contractible. Indeed, since $X$ is a noncompact $3$-manifold, $H_3(\overline{X},\Z)=0$. Since $\pi_1(\overline{X})=\pi_2(\overline{X})=0$, by Hurewicz's theorem, $\pi_3(\overline{X})=0$ as well. Hence by Whitehead's theorem $\overline{X}$ is homotopically equivalent to a point, i.e., it is contractible. By Smith theory, $\pi_1(X)$ must be torsion-free. Therefore, if $X$ is not simply connected, then $\pi_1(X)$ has a subgroup isomorphic to $\Z$. By Theorem \ref{thm:Bettirigid}, $\overline{X}$ splits a line, i.e., it is isomorphic to $\R\times Z$ where $(Z,\dist_Z,\haus^2)$ is an $\RCD(0,2)$ space. By \cite[Theorem 1.1]{LytchakStadler}, $(Z,\dist_Z)$ is an Alexandrov space with nonnegative curvature. Therefore $(\overline{X},\overline{\dist})$ and $(X,\dist)$ are Alexandrov spaces with nonnegative curvature as well. 
\end{proof}

We are ready to complete the proof of Theorem \ref{thm:classRCD03intro}.

\begin{proof}[Proof of Theorem \ref{thm:classRCD03intro}]
 If $\pi_2(X)\neq 0$ or $\pi_1(X)\neq 0$ then by Lemma \ref{lemma:trivialpi_2} and Lemma \ref{lemma:notsimplyconn} respectively $(X,\dist)$ is an Alexandrov space with nonnegative curvature whose universal cover splits a line isometrically. The classification of $3$-dimensional noncompact Alexandrov spaces with nonnegative curvature obtained in \cite[Part 2]{ShioyaYamaguchi} implies that any such space admits a smooth complete metric with nonnegative sectional curvature inducing the same topology. Thus we can assume that $\pi_1(X)=\pi_2(X)=0$. By the same argument as in the proof of Lemma \ref{lemma:notsimplyconn} above, $X$ is contractible under these assumptions. Since any $3$-manifold is triangulable, by  Theorem \ref{thm:simplicialvolthm} we have $\|X\|=0$. By \cite[Theorem 1]{BargagnatiFrigerio}, $X$ is homeomorphic to $\R^3$. 
\end{proof}

\begin{remark}
    It follows from the proof of Theorem \ref{thm:classRCD03intro} that the following holds: if $(X,\dist,\haus^3)$ is an $\RCD(0,3)$ space homeomorphic to a $3$-manifold and with Euclidean volume growth, i.e., 
    \begin{equation}
        \lim_{r\to \infty}\frac{\haus^3(B_r(p))}{r^3}>0\, 
    \end{equation}
    for some (and thus all) $p\in X$, then $X$ is homeomorphic to $\setR^3$. Indeed, $\pi_2(X)=0$ by Lemma \ref{lemma:trivialpi_2} and the Euclidean volume growth assumption. Thus the universal cover $(\overline{X},\overline{\dist},\haus^3)$ is contractible and therefore $\pi_1(X)$ must be torsion free as observed in the proof of Lemma \ref{lemma:notsimplyconn}. On the other hand, any $\RCD(0,n)$ space with Euclidean volume growth has finite $\pi_1$, as proved in \cite[Theorem 1.6]{MondinoWei} (note that by \cite{RCDsemilocally} the revised fundamental group coincides with the usual fundamental group). Thus $\pi_1(X)=0$ and the conclusion follows.

    We remark that the same conclusion was obtained with a different method recently in \cite[Theorem 1.9]{BruePigSem}.
\end{remark}

\section{Further applications}

In this section, we provide new topological restrictions for manifolds with non-negative Ricci curvature and $\RCD(0,n)$ spaces. First, we exploit the vanishing Theorem \ref{thm:simpvolintro} for the simplicial volume to obtain new restrictions on the topology of noncompact manifolds admitting a complete metric with $\mathrm{Ric}\ge0$ in dimension $4$.
Then, relying on Theorem \ref{T|ZhouAnelli}, we obtain restrictions for manifolds with nonnegative Ricci curvature (and $\RCD(0,n)$ spaces) and maximal volume growth in any dimension. The two main results are Propositions \ref{prop:applsimplvol} and \ref{P|toprestr2} below.

\begin{proposition}\label{prop:applsimplvol}
    There exists an infinite family of noncompact, one-ended, simply connected $4$-manifolds which do not admit any complete metric with nonnegative Ricci curvature. 
\end{proposition}

\begin{proposition} \label{P|toprestr2}
   For every $n\geq 4$, the manifold $S^{n-2} \times \bb{R}^2$ admits no complete metric of non-negative Ricci curvature and Euclidean volume growth.
\end{proposition}

\begin{remark}
 It is well-known that $S^{n-3}\times\R^3$ admits a complete metric with nonnegative Ricci curvature and Euclidean volume growth for every $n\ge 5$. Such metrics can be constructed with a doubly warped product ansatz by viewing $S^{n-3}\times\R^3$ as $\R_+\times S^2\times S^{n-3}$ away from a central $S^{n-3}$. 
\end{remark}

For the proof of Proposition \ref{prop:applsimplvol} we will rely on the following:

\begin{lemma}\label{lemma:simplv0graph}
    Let $\overline{M}^4$ be a smooth compact compact oriented manifold with boundary of dimension $4$. If $\|\overline{M}\setminus{\partial{\overline{M}}}\|=0$ then all the connected components of $\partial{\overline{M}}$ are graph manifolds. 
\end{lemma}

\begin{proof}
    By \cite[Example, p. 17]{GromovVolbCoho}, $\|\partial{\overline{M}}\|=0$, see also \cite[Theorem 6.4]{LohMunstJourn} and the discussion following it. Thus all connected components of the boundary of $\overline{M}$ have vanishing simplicial volume. By Perelman's resolution of the geometrization conjecture and \cite{Soma} (see also \cite[Theorem 2.7]{AndersonPaternain}) a closed orientable $3$-manifold with vanishing simplicial volume is a graph manifold. 
\end{proof}

\begin{proof}[Proof of Proposition \ref{prop:applsimplvol}]
 Let $N$ be any oriented closed hyperbolic $3$-manifold. By \cite[Theorem 2, Chapter VII]{KirbyTop4}, there exists a smooth compact connected and oriented $4$-manifold $\overline{M}$ with boundary $\partial{\overline{M}}$ diffeomorphic to $N$. By performing surgeries iteratively on a family of disjoint loops generating $\pi_1(\overline{M})$ we can assume that $\overline{M}\setminus \partial{\overline{M}}$ is simply connected. Moreover, it is clearly one-ended. We claim that any such manifold does not admit a complete Riemannian metric with $\Ric\ge 0$. Indeed, $N$ is not a graph manifold, hence by Lemma \ref{lemma:simplv0graph}, $\|\overline{M}\setminus\partial{\overline{M}}\|>0$. The conclusion follows by Theorem \ref{thm:simpvolintro}. 
\end{proof}

\begin{remark}
The very same argument, relying on the more general Theorem \ref{thm:simplicialvolthm}, rules out also the existence of $\RCD(0,4)$ structures on the class of manifolds considered in Proposition \ref{prop:applsimplvol}.   
\end{remark}

Proposition \ref{P|toprestr2} is a consequence of the following more general result and its subsequent corollary.

\begin{proposition}
    Let $n \geq 3$, and let $(X,\sd,\aH^n,p)$ be an $\RCD(0,n)$ space with Euclidean volume growth
    \begin{equation}
        \lim_{r \to + \infty}\frac{\aH^n(B_r(p))}{r^n} \geq v >0.
    \end{equation}
    Assume that there exists a compact set $K \subset X$ such that $X \setminus K$ is homeomorphic to $Y \times (0,+\infty)$ for a path-connected topological space $Y$. Then, there exists $C=C(n,v)>0$ such that $\pi_1(Y)$ has cardinality at most $C(n,v)$.
    \begin{proof}
         Let $C(n,v)$ be the constant given by \ref{T|ZhouAnelli}. Assume by contradiction that $\pi_1(Y)$ has cardinality (strictly) bigger than $2C(n,v)$.
         Let $x \in K$, and let $r>0$ be large enough so that $K \subset B_r(x)$, and the conclusion of Theorem \ref{T|ZhouAnelli} holds. Let $C' \subset X$ be the interior of the path-connected component of $B_{5r}(x) \setminus B_{2r}(x)$ containing $B_{4r}(x) \setminus B_{3r}(x)$, which exists by Theorem \ref{T|ZhouAnelli}. Let then $C \subset X$ be the path-connected component of $C'$ containing $B_{4r}(x) \setminus B_{3r}(x)$. One can check that this exists by the same method used in Theorem \ref{T|ZhouAnelli}. In particular, $C$ is an open path-connected set containing $B_{4r}(x) \setminus B_{3r}(x)$.
         Set 
    \begin{equation}
A:=\big(C \cup B_{4r}(x) \big) \setminus K\, , \quad B:=C \cup \big(X \setminus \bar{B}_{3r}(x) \big)\, .
        \end{equation}
        Observe that since $Y$ is path connected these sets are path-connected. In addition, their intersection $A \cap B$ is equal to the path-connected set $C$. Moreover, $A \cup B=X \setminus K$, and each element of $\pi_1(X \setminus K)$ admits representatives in $A$ and $B$. Consider a point $q\in A\cap B$ and the push-forward maps induced by the natural inclusions:
        \begin{equation}
            i_*:\pi_1(A \cap B,q) \to \pi_1(X \setminus K,q)\, , \quad 
            j_*:\pi_1(A \cap B,q) \to \pi_1(A,q)\, ,
        \end{equation}
        and
        \begin{equation}
            h_*:\pi_1(A \cap B,q) \to \pi_1(B,q)\, .
        \end{equation}
        By Theorem
        \ref{T|ZhouAnelli}, $i_*(\pi_1(A \cap B,q))$ has cardinality at most $C(n,v)$. Since $\pi_1(A)$ and $\pi_1(B)$ contain copies of $\pi_1(X \setminus K)$, which has cardinality bigger than $2C(n,v)$, there exists $g \in \pi_1(X \setminus K)$ with representatives $g_a \in \pi_1(A,q) \setminus j_*(\pi_1(A \cap B,q)) $ and $g_b \in \pi_1(B,q) \setminus h_*(\pi_1(A \cap B,q)) $.
        By the Seifert-Van Kampen theorem, the natural map
        \begin{equation}
            k_*:\pi_1(A,q) *_{\pi_1(A \cap B,q)} \pi_1(B,q) \to \pi_1(U)
        \end{equation}
        is an isomorphism.  In particular, since $k_*(g_a g_b^{-1})=1$, we deduce that $g_a g_b^{-1}=1$ in the amalgamated product, thus contradicting \cite[Theorem 2.6, page 187]{zbMATH01554175}.
    \end{proof}
\end{proposition}

\begin{corollary} 
    Let $n \geq 3$, and let $(M^n,g)$ be a smooth compact manifold with non-empty connected boundary. If $M \setminus \partial M$ admits a non-collapsed $\RCD(0,n)$ structure with Euclidean volume growth, then $\pi_1(\partial M)$ is finite.
\end{corollary}

\appendix

\section{A Margulis lemma for \texorpdfstring{$\CD(K,N)$}{CD(K,N)} spaces}\label{appendixA}

In the proof of the vanishing theorem for the simplicial volume (Theorem \ref{thm:simplicialvolthm}) we relied on the following Margulis lemma for $\RCD$ spaces:

\begin{theorem}\label{thm:margulis}
    For every $K\leq 0$ and $1\le N<\infty$ there exists a constant $\varepsilon:=\varepsilon(K,N)>0$ such that for every $\RCD(K,N)$ space $(X,\sd,\m)$ and $p\in X$, the image through the push-forward of the inclusion map
    \begin{equation}\label{eq:inclusionpi1}
        i_*:\pi_1(B_\varepsilon(p),p) \to \pi_1(B_{10}(p),p)
    \end{equation}
    is virtually nilpotent.
\end{theorem}

By scaling, Theorem \ref{thm:margulis} yields the following:

\begin{corollary}\label{rmk:margulis}
 For every $1\le N<\infty$ there exists a constant $\varepsilon:=\varepsilon(N)>0$ such that for every $\RCD(0,N)$ space $(X,\sd,\m)$, $p\in X$ and $R>0$ the image of the inclusion map
    \begin{equation}
        \pi_1(B_{\varepsilon R}(p),p) \to \pi_1(B_{ R}(p),p)
    \end{equation}
    is virtually nilpotent.
\end{corollary}

\begin{remark}
    With the very same strategy as in the proof of Theorem \ref{thm:margulis} one can prove a version of Margulis lemma for non-branching and semi-locally simply connected $\CD(K,N)$ spaces.
\end{remark}

For smooth Riemannian manifolds, Theorem \ref{thm:margulis} follows from a stronger result due to V.~Kapovitch and B.~Wilking in \cite{KapovitchWilking}. They proved that there exist constants $0<\varepsilon(K,N)<1$ and $C(K,N)>0$ such that if $(M^N,g)$ is a complete Riemannian manifold with $\mathrm{Ric}\ge K$ then the image of 
    \begin{equation}
        i_*:\pi_1(B_\varepsilon(p),p) \to \pi_1(B_{10}(p),p)
    \end{equation}
has a nilpotent subgroup whose index is bounded above by $C(K,N)$. In the recent \cite{DengSantosZamoraZhao} the same effective virtual nilpotency statement was obtained for $\RCD(K,N)$ spaces $(X,\dist,\meas,p)$ for the image of 
   \begin{equation}
        i_*:\pi_1(B_\varepsilon(p),p) \to \pi_1(X,p)\, .
    \end{equation}
It is expected that fueling the strategy of \cite{KapovitchWilking} with the ideas introduced to deal with the synthetic framework in \cite{DengSantosZamoraZhao} should lead to a complete generalization of the Margulis lemma from \cite{KapovitchWilking} to $\RCD(K,N)$ spaces. Here we obtain a weaker but local version of the Margulis lemma where the index of the nilpotent subgroup is not controlled effectively. While such statement seems non-optimal, it is sufficient for the purposes of the present paper. The proof follows a different strategy with respect to \cite{KapovitchWilking,DengSantosZamoraZhao}, and it is based on \cite{BreuillardGreenTao}.

The proof of Theorem \ref{thm:margulis} is deferred to the end of the appendix. Besides \cite{BreuillardGreenTao}, the main technical tool will be a local Bishop-Gromov inequality valid for universal covers of open subsets of $\RCD(K,N)$ spaces endowed with their natural length structure.

\begin{remark}\label{rmk:lengthdistanceoncovering}
Let $(X,\sd)$ be a length metric space (not necessarily complete). Let $\tilde{X}$ be a path connected topological space, and let $\pi:\tilde{X} \to X$ be a covering map.  We define a distance $\tilde{\sd}$ on $\tilde{X}$ as
\begin{equation}
    \tilde{\sd}(a,b):= \inf \{\length^{\sd}(\pi \circ \gamma):\gamma \in C([0,1],\tilde{X})\, \, \,  \gamma(0)=a, \gamma(1)=b\}\, .
\end{equation}
With this definition, it is clear that $\pi:(\tilde{X},\tilde{\sd}) \to (X,\sd)$ is a local isometry. Precisely, for every $p\in \tilde X$ there is a neighborhood $N_p\ni p$ such that $\pi|_{N_p}$ is an isometry.
It follows that for every $\gamma \in C([0,1],\tilde{X})$ we have 
\begin{equation}
\length^{\tilde{\sd}}(\gamma)=\length^{\sd}(\pi \circ \gamma)\, , 
\end{equation}
so that $\gamma \in AC([0,1],\tilde{X})$ if and only if $\pi \circ \gamma \in AC([0,1],X)$. Moreover, $(\tilde{X},\tilde{\sd})$ is a length metric space.

Finally, if $\m$ is a measure on $X$, there is a unique measure $\tilde{\m}$ on $\tilde{X}$ such that $\pi_\#(\tilde \m\mres{N_p})= \m\mres{\pi(N_{p})}$ for every $N_p$ as above. 
\end{remark}

We denote by $v_{K,N}(r)$ the volume of the ball of radius $r$ in the model space for the $\CD(K,N)$ condition.

\begin{lemma} \label{L|lift}
    Let $(X,\sd,\m)$ be an $\RCD(K,N)$ space with $K \leq 0$ and $1\le N<\infty$. 
    Let $O \subset X$ be an open, bounded, and path-connected set and call $\bar \sd$ the length distance induced by $\sd$ on $O$. Given a path-connected covering $\pi:\tilde{O} \to O$, endow $\tilde{O}$ with the length distance $\tilde{\sd}$ and the measure $\tilde{\m}$ obtained by lifting $\bar \sd$ and $\m$, cf.~Remark \ref{rmk:lengthdistanceoncovering}. For every $\tilde{p} \in \tilde{O}$, the function 
    \begin{equation}\label{eq:measuredecreasing}
        (0,\sd(\pi(\tilde{p}),X\setminus O)) \ni r \mapsto \frac{\tilde\m \big(B^{\tilde \sd}_r(\tilde{p})\big)}{v_{K,N}(r)}\, ,
    \end{equation}
    is non-increasing. 
    Moreover, if $\m =\Haus^N$ then the function in \eqref{eq:measuredecreasing} is bounded from above by $1$.
\end{lemma}

\begin{proof}
    Observe that $(O,\sd)$ and $(O,\bar\sd)$ are locally isometric, meaning that for every $p\in O$ there is a neighborhood $N_p\ni p$ such that $\sd$ and $\bar \sd$ agree on $N_p\times N_p$.
    In particular, since $(X,\dist)$ is non-branching by \cite{zbMATH08038284}, the length space $(O,\bar\sd)$ is non-branching as well.
    Moreover, since $\pi:(\tilde O,\tilde \sd)\to (O,\bar\sd)$ is a local isometry, $(\tilde O,\tilde\sd)$ is non-branching as well.
\medskip

    \textbf{Claim 1:} every pair $\tilde p,\tilde q\in \tilde O$ such that $\tilde \sd (\tilde p,\tilde q)< \sd(\pi(\tilde{p}),X\setminus O)$ is connected by a geodesic in $(\tilde O,\tilde \sd)$. 
\medskip

    To prove Claim 1 observe that, since $(\tilde O, \tilde \sd)$ is a length space, there is a sequence of curves $(\tilde \gamma_n) \subset  AC([0,1],\tilde O)$ such that 
    \begin{equation}
    \length^{\tilde \sd}(\tilde \gamma_n)\leq \tilde \sd (\tilde p,\tilde q) +\frac 1n 
    \end{equation}
    for every $n\in \N$. Consider the sequence $(\gamma_n) \subset  AC([0,1],O)$ defined as $\gamma_n:=\pi (\tilde \gamma_n)$ for every $n$. Eventually, these curves will be contained in a compact subset of $O$. Thus, up to not relabelled subsequences, $\gamma_n\to \gamma_\infty \in AC([0,1],O)$ uniformly. Moreover, by what we observed before, 
    \begin{equation}
    \length^{\bar\sd}(\gamma_n)=\length^{\tilde \sd}(\tilde \gamma_n)\leq \tilde \sd (\tilde p,\tilde q) +\frac 1n 
    \end{equation}
    for every $n$ and therefore $\length^{\bar \sd}(\gamma_\infty)\leq \sd (\tilde p,\tilde q)$ by lower semicontinuity of the length. Let $\tilde \gamma_\infty$ be the lifting of $\gamma_\infty$ starting from $\tilde{p}$. 
    On the one hand, we have that 
    \begin{equation}
    \length^{\tilde \sd}(\tilde \gamma_\infty)=\length^{\bar \sd}( \gamma_\infty)\leq \tilde \sd (\tilde p,\tilde q)\, . 
    \end{equation}
    On the other hand, since $O$ is semi-locally simply connected, the curves $\gamma_n$ are eventually homotopically equivalent to $\gamma_\infty$. Therefore $\tilde \gamma_\infty(1)=\tilde q$. This concludes the proof of the claim.

\medskip
    
    Take any $\tilde p \in \tilde O$ and any $r_0< \sd(\pi(\tilde{p}),X\setminus O)$.
    Consider the $1$-Lipschitz function $\tilde u:\tilde O\to \R$, defined as $\tilde u(q):=\tilde \sd (\tilde p,q) \wedge r_0$. We are going to consider the disintegration of the measure $\tilde\m$ with respect to $\tilde u$, following \cite[Section 7]{CavallettiMilman}. In particular, we define the transport relation  
    \begin{equation}
        \tilde R_{\tilde u} := \{ (\tilde x,\tilde y) \in \tilde O\times \tilde O\, :\, \tilde x\neq \tilde y \,\text{ and } \,|\tilde u (\tilde x)-\tilde u(\tilde y)|=\sd(\tilde x,\tilde y)\}
    \end{equation}
    and the associated transport set $\tilde T_{\tilde u}:= \mathtt p_1(\tilde R_{\tilde u})$, where $\mathtt p_1$ denotes the projection on the first factor. By what we observed before, 
    \begin{equation}
    \tilde T_{\tilde u}= \bar B^{\tilde \sd}_{r_0} (\tilde p)\, .
    \end{equation}
    We introduce also the set of branching points
    \begin{equation}
        \tilde A:= \{x\in \tilde T_{\tilde u} \,:\, \exists z,w \in \tilde X \text{ with } (x,z),(x,w)\in \tilde R_{\tilde u}, (z,w)\not\in \tilde R_{\tilde u}\}\, .
    \end{equation}
    Accordingly, we consider the non-branched transport set $\tilde T_{\tilde u}^b :=\tilde T_{\tilde u} \setminus \tilde A$ and the non-branched transport relation 
    \begin{equation}
    \tilde R_{\tilde u}^b:= \tilde R_{\tilde u} \cap (\tilde T_{\tilde u}^b\times \tilde T_{\tilde u}^b)\, .
    \end{equation}
    Since $(\tilde O,\tilde \sd)$ is non-branching, we have that 
    \begin{equation}
    \tilde T^b_{\tilde u}= \bar B^{\tilde \sd}_{r_0} (\tilde p) \setminus (\mathrm{Cut}(\tilde p) \cup \{\tilde p\})\, . 
    \end{equation}
    With our choice of $\tilde u$, all the defined sets are easily seen to be Borel. 
    \medskip
    
    Note that $\tilde R_{\tilde u}^b$ is an equivalence relation on $\tilde T_{\tilde u}^b$. Denoting by $\tilde Q\subset X$ the set of its equivalence classes, every $\alpha \in \tilde Q$ corresponds to a set $\tilde X_\alpha \subset \tilde O$ which is a geodesic and thus isometric to an interval $\tilde I_\alpha \subset \R$. We equip $\tilde Q$ with the $\sigma$-algebra induced by the equivalence projection $\mathfrak Q: \tilde T^b_{\tilde u} \to \tilde Q$ and with the measure $\tilde{\mathfrak q}:= \mathfrak Q_\# \big(\tilde \m\mres{\tilde T_{\tilde u}^b }\big)$.     
    The Disintegration Theorem (see for instance \cite[Theorem 6.19]{CavallettiMilman}) guarantees the existence of a (essentially unique) \emph{consistent} (see \cite[Definition 6.18]{CavallettiMilman}) disintegration, i.e.~a measurable map $\tilde Q \ni \alpha \mapsto \tilde\m_\alpha\in \mathcal{P}(\tilde O)$ such that 
    \begin{equation}
        \tilde\m\mres{\tilde T^b_{\tilde u}} = \int_{\tilde Q} \tilde \m_\alpha \,\tilde{\mathfrak q }(d\alpha)\, .
    \end{equation}
 \medskip
    
    Consider now any point $\tilde q\in B^{\tilde \sd}_{r_0} (\tilde p)$. By construction, there is an open neighborhood $U_{\tilde q}$ of $\tilde q$ such that the map 
    \begin{equation}
    \pi_{\tilde q} :=\pi|_{U_{\tilde q}}: (U_{\tilde q},\tilde \sd) \to (O,\bar \sd) 
    \end{equation}
    is an isometry. Up to taking a smaller $U_{\tilde q}$, we can assume that $\pi_{\tilde q} :(U_{\tilde q},\tilde \sd) \to (\pi(U_{\tilde q}), \sd)$ is a bijective isometry. 
    Then, consider a $1$-Lipschitz function $u:X\to \R$ such that $u(x):= \tilde \sd (\tilde p, \pi_{\tilde q}^{-1}(x))$ for every $x\in \pi(U_{\tilde q})$.
    Similarly as before, we can get a disintegration of the measure $\m$ on $(X, \sd)$, with respect to $u$ (see \cite[Section 3]{CM18} for the adaptation of the discussion above to $\sigma$-finite measures). In particular, similarly as before we can define the objects $R_u$, $T_u$, $R_u^b$, $T_u^b$, $Q$, $\{X_\beta\}_{\beta \in Q}$ and $\mathfrak q$, and according to \cite[Theorem 3.4]{CM18} we find a disintegration
    \begin{equation}\label{eq:disintegrationbelow}
        \m\mres{T_u^b} = \int_Q  \m_\beta \,{\mathfrak q }(d\beta)\, .
    \end{equation}
    Such disintegration is \emph{strongly consistent} (see \cite[Definition 6.18]{CavallettiMilman}), meaning that for $\mathfrak q$-a.e.\ $\beta\in Q$ the measure $ \m_\beta$ is concentrated on $ X_\beta$. Moreover, since $(X,\sd,\m)$ is an $\RCD(K,N)$ space, \cite[Theorem 3.6]{CM18} ensures that for $\mathfrak q$-a.e.~$\beta\in Q$ the space $(  X_\beta,   \sd,  \m _\beta)$ is isomorphic (as metric measure space) to $(  I_\beta,|\cdot|,  h_\beta \Leb^1)$, for an interval $  I_\beta \subset \R$ and a measurable function $ h_\beta$ which is a $\CD(K,N)$ density. 
    Thus, by \cite[Appendix A]{CavallettiMilman}, for $\mathfrak q$-a.e.~$\beta \in Q$ the differential inequality 
    \begin{equation}\label{eq:convexontransportrays}
        \big( h^{1/(N-1)}_\beta\big)'' + \frac{K}{N-1}\, h^{1/(N-1)}_\beta \leq 0 
    \end{equation}
    holds in a weak sense on $I_\beta$. Finally, \cite[Lemma 3.5]{CM18} guarantees $\m(T_u\setminus T_u^b)=0$.
\medskip

    Define the sets 
    \begin{equation}
        \tilde Q^{\tilde q}:= \{\alpha \in \tilde Q \,:\, \tilde X_\alpha \cap U_{\tilde q} \neq \emptyset \} \qquad Q^{\tilde q}:= \{\beta \in Q \,:\, X_\beta \cap \pi_{\tilde q}(U_{\tilde q}) \neq \emptyset\}\, .
    \end{equation}
    Both sets are easily seen to be measurable.
\medskip
    
    \textbf{Claim 2:} there exists a measurable bijective map $\Pi: \tilde Q^{\tilde q} \to Q^{\tilde q}$ such that  
    \begin{equation} \label{eq:Piandpi}
        \pi_{\tilde q}(\tilde X_\alpha \cap U_{\tilde q}) = X_{\Pi(\alpha)} \cap \pi_{\tilde q}(U_{\tilde q})\, . 
    \end{equation}
    Such map is defined as follows. Given $\alpha \in \tilde{Q}^{\tilde q}$, take any $x \in X_\alpha \cap U_{\tilde{q}}$ and define $\Pi(\alpha)$ as the equivalence class of $\pi_{\tilde q}(x)$.
    To show that such definition is well posed, given any $\tilde q_1,\tilde q_2 \in U_{\tilde q}$ such that $(\tilde q_1,\tilde q_2) \in \tilde R_{\tilde u}^b$, we observe that by definition
    \begin{equation}
        |u(\pi_{\tilde q}(\tilde q_1)) - u(\pi_{\tilde q}(\tilde q_2))| = \sd(\pi_{\tilde q}(\tilde q_1),\pi_{\tilde q}(\tilde q_2))\, ,
    \end{equation}
    thus $\big(\pi_{\tilde q}(\tilde q_1),\pi_{\tilde q}(\tilde q_2)\big)\in R_u$. Suppose that $\big(\pi_{\tilde q}(\tilde q_1),\pi_{\tilde q}(\tilde q_2)\big)\in R_u\setminus R_u^b$. Then w.l.o.g.~there exists $x\in X$ such that $(\pi_{\tilde q}(\tilde q_1), x)\in R_u$ and $(\pi_{\tilde q}(\tilde q_2), x)\not\in R_u$. Up to replacing $x$ with a suitable point on the geodesic connecting $\pi_{\tilde q}(\tilde q_1)$ and $x$ inside $\pi_{\tilde q}(U_{\tilde q})$, we can assume that $x\in \pi_{\tilde q}(U_{\tilde q})$. Setting $\tilde x := \pi_{\tilde q}^{-1}(x)$, we obtain that $(\tilde q_1, \tilde x)\in \tilde R_{\tilde u}$ and $(\tilde q_2, \tilde x)\not\in \tilde R_{\tilde u}$, a contradiction. This proves that the map $\Pi$ is well defined. Reverting the same strategy, one shows that $\Pi$ is bijective and that \eqref{eq:Piandpi} holds. The measurability of $\Pi$ simply follows from \eqref{eq:Piandpi}. 
    \medskip

    Recalling that $\m(T_u\setminus T_u^b)=0$, an argument similar to the one discussed above shows that $\tilde \m(\tilde T_{\tilde u}\setminus \tilde T_{\tilde u}^b)=0$. In particular,
    \begin{equation}
        \tilde\m\mres{U_{\tilde q}} = \int_{\tilde Q^{\tilde q}} \tilde \m_\alpha\mres{U_{\tilde q}} \,\,\tilde{\mathfrak q }(d\alpha) \quad \text{ and } \quad \m\mres{\pi_{\tilde q}(U_{\tilde q})} = \int_{Q^{\tilde q}}  \m_\beta\mres {\pi_{\tilde q}(U_{\tilde q})} \,\,{\mathfrak q }(d\beta)\, .
    \end{equation}

    By definition of $\tilde \m$, we know that $(\pi_{\tilde q})_\#(\tilde \m\mres{U_{\tilde q}})= \m\mres{\pi_{\tilde q}(U_{\tilde q})}$. Thus
    \begin{align}
        \tilde\m\mres{U_{\tilde q}} =& \big( \pi_{\tilde q}^{-1}\big)_\# [\m\mres{\pi(U_{\tilde q})}] = \int_{Q^q} \big( \pi_{\tilde q}^{-1}\big)_\# [\m_\beta\mres{\pi_{\tilde q}(U_{\tilde q})}] \,{\mathfrak q }(d\beta) \\
        =& \int_{\tilde Q^{\tilde q}} \big( \pi_{\tilde q}^{-1}\big)_\# [\m_{(\Pi(\alpha))}\mres{\pi_{\tilde q}(U_{\tilde q})}] \,{[ (\Pi^{-1})_\#\mathfrak q] }(d\alpha)\, .
    \end{align}
    This gives a disintegration of $\tilde\m\mres{U_{\tilde q}}$ which is (strongly) consistent because \eqref{eq:disintegrationbelow} is strongly consistent. By essential uniqueness of consistent disintegrations, see \cite[Theorem 6.19]{CavallettiMilman}, we deduce that 
    \begin{equation}
        \tilde \m_\alpha\mres{U_{\tilde q}} = \big( \pi_{\tilde q}^{-1}\big)_\# (\m_{\Pi(\alpha)}\mres{\pi(U_{\tilde q})})\, , \qquad \text{for }\tilde{\mathfrak q}\text{-a.e. }\alpha \in \tilde Q^{\tilde q}\, .
    \end{equation}
    In particular, $\tilde \m_\alpha|_{U_{\tilde{q}}}$ is concentrated on $\tilde X_\alpha$, which is isometric to a closed interval $\tilde I_\alpha$. Moreover, by \eqref{eq:convexontransportrays}, for $\tilde{\mathfrak q}$-a.e. $\alpha\in \tilde Q^{\tilde q}$ the metric measure space $(\tilde X_\alpha \cap U_{\tilde q}, \tilde \sd, \tilde \m_\alpha)$ is isomorphic to $(  \tilde I_\alpha \cap U_{\tilde q},|\cdot|,  \tilde h_\alpha \Leb^1)$, where $\tilde h_\alpha$ is such that
    \begin{equation}\label{eq:CDdiff}
        \big( \tilde h^{1/(N-1)}_\alpha\big)'' + \frac{K}{N-1}\, \tilde h^{1/(N-1)}_\alpha \leq 0\, \, \quad \text{weakly on}\quad  \tilde I_\alpha \cap U_{\tilde q}\, .
    \end{equation}
    Here and in the following, with a slight abuse of notation, given any open set $E\subset \tilde O$ we denote by $\tilde I_\alpha \cap E$ the subset of $\tilde I_\alpha$ that corresponds via isometry to $\tilde X_\alpha \cap E$. This last conclusion holds for every $\tilde q\in B^{\tilde \sd}_{r_0} (\tilde p)$ and for the respective neighborhood $U_{\tilde q}$. 
    \medskip
    
    Take any $r_1<r_0$. Since $\bar B^{\tilde \sd}_{r_1} (\tilde p)$ is compact we can find $\tilde q_1,\dots, \tilde q_M\in B^{\tilde \sd}_{r_0} (\tilde p)$ such that 
    \begin{equation}\label{eq:incllemmaApp}
    B^{\tilde\sd}_{r_1}(\tilde p)\subset U_{\tilde q_1}\cup\cdots\cup U_{\tilde q_M}\, . 
\end{equation}    
    We let $\tilde B$ be the $\tilde {\mathfrak q}$-null set defined as the collection of elements $\alpha \in \tilde Q$ such that \eqref{eq:CDdiff} does not hold for some $\tilde q_i$ and $U_{\tilde q_i}$, with $i=1,\dots,M$. Thus, for every $\alpha \in \tilde Q\setminus \tilde B$ the differential inequality \eqref{eq:CDdiff} holds on $\tilde I_\alpha \cap U_{\tilde q_i}$ for every $i$ such that $\tilde I_\alpha \cap U_{\tilde q_i}\neq \emptyset$. Thanks to \eqref{eq:incllemmaApp} we conclude that for $\tilde {\mathfrak q}$-a.e. $\alpha \in \tilde Q$ the inequality \eqref{eq:CDdiff} holds on $\tilde I_\alpha \cap B^{\tilde\sd}_{r_1}(\tilde p)$. Thus, for every such $\alpha$ the function $\tilde h_\alpha$ is a $\CD(K,N)$ density on $\tilde I_\alpha \cap B^{\tilde\sd}_{r_1}(\tilde p)$. Therefore by Bishop-Gromov we have that 
    \begin{equation}
        \frac{\tilde \m_\alpha \big(B^{\tilde \sd}_s(\tilde p)\big)}{v_{K,N}(s)} \geq \frac{\tilde \m_\alpha \big(B^{\tilde \sd}_t(\tilde p)\big)}{v_{K,N}(t)}\, , \quad \text{for every }s<t<r_1\, .
    \end{equation}
    Thus, for every $s<t<r_1$ we get
    \begin{align}
        \tilde \m \big(B^{\tilde \sd}_s(\tilde p)\big) & = \int_{\tilde Q} \tilde \m_\alpha\big(B^{\tilde \sd}_s(\tilde p)\big) \,\tilde{\mathfrak q}(d \alpha) \\
        & \geq \int_{\tilde Q} \frac{v_{K,N}(s)}{v_{K,N}(t)}\, \tilde \m_\alpha \big(B^{\tilde \sd}_t(\tilde p)\big) \,\tilde{\mathfrak q}(d \alpha) = \frac{v_{K,N}(s)}{v_{K,N}(t)} \, \tilde \m \big(B^{\tilde \sd}_t(\tilde p)\big)\, .
    \end{align}
    The first part of the statement follows since $\tilde p\in \tilde O$ and $r_1<r_0< \sd(\pi(\tilde{p}),X\setminus O)$ were arbitrary.
    \medskip

    Assume now that $\m=\haus^N$, so that $\tilde{\m}=\tilde{\haus}^N$ is the Hausdorff measure induced by $\tilde{\sd}$.
    To prove the final part of the statement, it is sufficient to show that 
    \begin{equation}
        \lim_{r\to 0 }\frac{\tilde{\haus}^n\big(B^{\tilde \sd}_r(\tilde p)\big)}{v_{K,N}(r)} \leq 1\, .
    \end{equation}
    To this aim we observe that for $r>0$ sufficiently small the restriction 
    \begin{equation}
    \pi|_{B^{\tilde \sd}_r(\tilde p)}:B^{\tilde \sd}_r(\tilde p) \to  B_r(p) 
    \end{equation}
    is an isometry. Thus the claim follows from \cite[Corollary 2.14]{DephilGigli}.
\end{proof}

We record below another useful technical result about coverings of balls in $\RCD$ spaces. 

\begin{lemma} \label{L|fundamentaldomain}
    Let $(X,\sd,\m,p)$ be an $\RCD(K,N)$ space for some $K\in\R$ and $1\le N<\infty$. Let $r \in (0,+\infty]$, and let $\pi:Y \to B_r(p)$ be a covering, with the convention that $B_\infty(p):=X$. Let $\tilde{\sd}$ and $\tilde{\m}$ be respectively the distance and the measure induced on $Y$ by lifting the length structure of $B_r(p)$ and the measure $\m$, cf.~Remark \ref{rmk:lengthdistanceoncovering}.     
    Let $\tilde{p} \in \pi^{-1}(p)$ be fixed, and let $F \subset Y$ be the union of the interiors of the lifts to $\tilde{p}$ of the segments of $B_r(p)$ having $p$ as an endpoint. Then the following hold:
    \begin{enumerate}
        \item \label{item|funddomain1} $\pi|_{F}:F \to B_r(p)$ is injective;
        \item \label{item|funddomain2} For every deck transformation $g:Y \to Y$, $F \cap g(F)=\emptyset$;
        \item \label{item|funddomain4} For every $s \in (0,r)$, we have that $F \cap \pi^{-1}(B_s(p)) \subset B^{\tilde{\sd}}_s(\tilde{p})$;
         \item \label{item|funddomain3}
        For every $s \in (0,r)$, we have that $\m(B_s(p))=\tilde{\m}(F \cap B^{\tilde{\sd}}_s(\tilde{p}))$.
    \end{enumerate}
    \begin{proof}
        To prove item \ref{item|funddomain1}, let $a,b \in F$ such that $\pi(a)=\pi(b)$. Let $\gamma_a,\gamma_b \subset X$ be segments having $p$ as an endpoint whose interiors contain $\pi(a)=\pi(b)$, and such that their lifts to $\tilde{p}$ contain $a$ and $b$, respectively. Since $\RCD(K,N)$ spaces are non-branching due to \cite{zbMATH08038284}, either $\gamma_a \subset \gamma_b$ or $\gamma_b \subset \gamma_a$, which implies $a=b$.
\medskip

        To prove item \ref{item|funddomain2}, assume by contradiction that $x \in F \cap g(F)$. Then there are segments $\gamma_{\tilde{p}},\gamma_{g(\tilde{p})} \subset X$ having $p$ as an endpoint, and such that the interiors of the lifts to $Y$ from $\tilde{p}$ and $g(\tilde{p})$ respectively both contain $x$. Again using that $\RCD(K,N)$ spaces are non-branching, we deduce that $\gamma_{\tilde{p}}\subset \gamma_{g(\tilde{p})}$ or $\gamma_{g(\tilde{p})} \subset \gamma_{\tilde{p}}$. Thus their common part has two different lifts starting from $x \in Y$, a contradiction.
\medskip

        To prove item \ref{item|funddomain4}, let $x \in F \cap \pi^{-1}(B_s(p))$, and let $\gamma_x \subset B_s(p)$ be the segment with endpoint $p$ whose lift to $\tilde{p}$ contains $x$. The length of $\gamma_x$ is the same as the length of its lift. Thus $x \in B^{\tilde{\sd}}_s(\tilde{p})$, concluding the proof.
\medskip

        For item \ref{item|funddomain3}, by Vitali's Covering Theorem there exists a countable family of disjoint closed balls 
        \begin{equation}
        \{C_i\}_{i \in \bb{N}} \subset \bar B_s^{\tilde{\sd}}(\tilde p) \subset Y 
        \end{equation}
        centered in points of $F \cap \bar{B}^{\tilde{\sd}}_s(\tilde{p})$ such that $\pi|_{C_i} : C_i\to \pi(C_i) \subset \bar{B}_s(p)$ is a measure preserving isometry and 
        \begin{equation}
            \tilde{\m}\Big(F \cap \bar{B}^{\tilde{\sd}}_s(\tilde{p}) \setminus \bigcup_{i \in \bb{N}}C_i \Big)=0\, .
        \end{equation}
        Vitali's theorem can be applied since $(\bar B_s^{\tilde{\sd}}(\tilde p),\tilde{\sd},\tilde{\m})$ is uniformly locally doubling by Lemma \ref{L|lift}.
        Hence, using item \ref{item|funddomain1} in the last inequality and the fact that $\tilde{\m}(\partial B^{\tilde{\sd}}_s(\tilde{p}))=0$ by Lemma \ref{L|lift}, we deduce that
        \begin{align}
            \tilde{\m}\Big(F \cap B^{\tilde{\sd}}_s(\tilde{p}) \Big) =& \sum_{i \in \bb{N}} \tilde{\m} \Big(F \cap B^{\tilde{\sd}}_s(\tilde{p}) \cap C_i \Big)\\
            =&
\sum_{i \in \bb{N}} {\m} \Big(\pi(F \cap B^{\tilde{\sd}}_s(\tilde{p}) \cap C_i) \Big) \leq \m(B_s(p))\, .
        \end{align}
        To prove the converse inequality, let $\Sigma \subset X$ be the cut-locus of $p$. We have that $\m(\Sigma)=0$ (see, for instance, \cite[Remark 7.5]{CavallettiMilman}). By the very definition of $F$, 
        \begin{equation}
            B_s(p) \setminus \Sigma \subset \pi(F \cap B^{\tilde{\sd}}_s(\tilde{p}))\, .
        \end{equation}
        Hence, using the collection $\{C_i\}_{i \in \bb{N}}$ constructed above and the fact that $\pi:Y \to B_r(p)$ sends $\tilde{\m}$-negligible sets to $\m$-negligible sets, we deduce that
        \begin{align}
            \m(B_s(p)) =&
            \m(B_s(p) \setminus \Sigma)
            \leq\m \Big (\pi(F \cap B^{\tilde{\sd}}_s(\tilde{p}))\Big)\\ 
            \leq&\sum_{i \in \bb{N}} {\m} \Big(\pi(F \cap B^{\tilde{\sd}}_s(\tilde{p}) \cap C_i) \Big) \\
            =&
            \sum_{i \in \bb{N}} \tilde{\m} \Big(F \cap B^{\tilde{\sd}}_s(\tilde{p}) \cap C_i \Big)=
            \tilde{\m} \Big( F \cap B^{\tilde{\sd}}_s(\tilde{p}) \Big)\, ,        
            \end{align}
            thus concluding the proof.
    \end{proof}
\end{lemma}

\begin{proof}[Proof of Theorem \ref{thm:margulis}]
    Following the notation of Lemma \ref{L|lift},  let $\bar \sd$ denote the length distance induced by $\sd$ on $B_{10}(p)$. Endow the universal cover $\pi:\tilde X\to B_{10}(p)$ with the length distance $\tilde{\sd}$ and the measure $\tilde{\m}$ obtained by lifting $\bar \sd$ and $\m$.
    Consider the metric space $(Y,\rho)$, where $Y:=\pi^{-1}(p)$ and $\rho$ is the restriction of $\tilde \sd$ to $Y$. 
 \medskip
 
    \textbf{Claim}: There exists a constant $C(K,N)$ such that $(Y,\rho)$ has \textit{packing constant} $C(K,N)$, meaning that every ball of radius $4$ in $(Y,\rho)$ can be covered by at most $C(K,N)$ balls of radius $1$.   
    
    \medskip
     
     To prove the claim observe that by Lemma \ref{L|lift} there exists a constant $C(K,N)$ such that for every $\tilde p \in \pi^{-1}(p)$ we have
    \begin{equation}\label{eq:9doubling}
        \tilde \m \big(B^{\tilde \sd}_{9}(\tilde p)\big) \leq C(K,N)\, \tilde \m \big(B^{\tilde \sd}_{1/2}(\tilde p)\big)\, .
    \end{equation}
    Fix any $\tilde p \in \pi^{-1}(p)$ and suppose that there are $M$ disjoint balls of radius $1/2$ with centers in $\pi^{-1}(p) \cap B^{\tilde \sd}_4(\tilde p)$, i.e. we find $p_1,\dots,p_M \in \pi^{-1}(p) \cap B^{\tilde \sd}_4(\tilde p)$ such that $B^{\tilde \sd}_{1/2}(p_i) \cap B^{\tilde \sd}_{1/2}(p_j)=\emptyset$ whenever $i\neq j$. Using \eqref{eq:9doubling} we deduce that
    \begin{align*}
        \tilde \m \big(B^{\tilde \sd}_{9/2}(\tilde p)\big) & \geq \sum_{i=1}^M \tilde \m \big(B^{\tilde \sd}_{1/2}(p_i)\big) \\
        & \geq \frac 1{C(K,N)} \sum_{i=1}^M \tilde \m \big(B^{\tilde \sd}_{9}(p_i)\big) \geq \frac M{C(K,N)} \tilde \m \big(B^{\tilde \sd}_{9/2}(\tilde p)\big)\, ,
    \end{align*}
    thus $M\leq C(K,N)$. In particular for any $\tilde p \in Y$ there are at most $C(K,N)$ disjoint balls in $(Y,\rho)$ of radius $1/2$ with centers in $B^\rho_4(\tilde p)$.

    Consider any $p_0\in Y$. If $B^\rho_4(p_0)\setminus B^\rho_1(p_0)\neq\emptyset$ take $p_1\in B^\rho_4(p_0)\setminus B^\rho_1(p_0)$. Iterate this process by selecting at step $i$ a point $p_i \in B^\rho_4(p_0)\setminus \big(B^\rho_1(p_0) \cup \cdots \cup B^\rho_1(p_{i-1})\big)$ if $B^\rho_4(p_0)\setminus \big(B^\rho_1(p_0) \cup \cdots \cup B^\rho_1(p_{i-1})\big)\neq\emptyset$. Observe that at each step $i$ the balls $B^\rho_{1/2}(p_0), \dots , B^\rho_{1/2}(p_i)$ are pairwise disjoint by construction. Therefore this process has to stop within step $C(K,N)-1$, thus producing a cover of $B^\rho_4(p_0)$ with at most $C(K,N)$ balls of radius $1$. 
    
    \medskip 

    Since $X$ is semi-locally simply connected by \cite{RCDsemilocally}, $B_{10}(p)$ is semi-locally simply connected as well. Thus we can
   identify $\pi_1(B_{10}(p),p)$ with the deck transformations of the universal cover $\pi:\tilde X \to B_{10}(p)$. By restricting any such deck transformation to $Y=\pi^{-1}(p)$ we get that $\pi_1(B_{10}(p),p) \leq \text{Iso}(Y,\rho)$, with a slight abuse of notation. By \cite[Corollary 11.17]{BreuillardGreenTao} there exists a constant $\varepsilon=\varepsilon(K,N)>0$ such that for every $\tilde p\in Y$ the \emph{almost stabiliser} $\Gamma_{3\varepsilon}(\tilde p):= \langle S_{3\varepsilon}(x)\rangle$, where $S_{3\varepsilon}(x):=\{\gamma\in \pi_1(B_{10}(p),p)\,:\, \rho(\gamma\cdot \tilde p, \tilde p)<3 \varepsilon\}$, is virtually nilpotent. On the other hand, it is a standard observation that every loop in $\pi_1(B_{\varepsilon}(p),p)$ is homotopically equivalent to a concatenation of loops of length less than $3\varepsilon$. In particular the image of the map \eqref{eq:inclusionpi1} is isomorphic to a subgroup of the almost stabiliser $\Gamma_{3\varepsilon}(\tilde p)$. Therefore it is virtually nilpotent.
\end{proof}

\printbibliography

\end{document}